\documentclass[11pt]{article}
\usepackage{amsmath,amssymb,amsthm}
\usepackage{authblk}
\usepackage{cite}
\usepackage{indentfirst}
\usepackage{array,booktabs,tabularx}
\usepackage{float}
\usepackage{hyperref}

\hypersetup{
  colorlinks=true,
  linkcolor=blue,
  citecolor=blue,
  urlcolor=blue
}

\newtheorem{theorem}{Theorem}[section]
\newtheorem{lemma}{Lemma}[section]
\newtheorem{proposition}{Proposition}[section]

\newtheorem{remark}{Remark}[section]

\numberwithin{equation}{section}
\allowdisplaybreaks[4]
\begin{document}

\title{Global stability and sharp decay for the two-dimensional non-rotating primitive equations without thermal diffusion}
\author[1]{Fengjie Luo}
\author[1]{Quan Wang \thanks{Corresponding author.}\thanks{Email addresses:\ Luofengjie664@163.com\ (F. Luo),\ xihujunzi@scu.edu.cn\ (Q. Wang).}}

\affil[1]{College of Mathematics, Sichuan University, Chengdu, Sichuan 610065, PR China}
\date{}
\maketitle

\begin{abstract}
\noindent
We study the two-dimensional non-rotating primitive equations with positive horizontal and vertical viscosity and without thermal diffusion. Our main results are fourfold. First, for every integer \(s\ge4\), we prove that sufficiently small, compatible, doubly periodic \(H^s\) perturbations of a stably stratified rest state with an affine temperature profile generate unique global strong solutions; the horizontal velocity and temperature remain uniformly small at the initial Sobolev order despite the derivative loss in the full diagnostic vertical velocity, and strict stratification persists. Second, we establish an inverse-map conservation law that determines the terminal temperature solely from the initial temperature: the terminal label is identified as the normalized increasing rearrangement of the initial label modulo the vertical period, and the terminal perturbation admits a quadratic expansion with cubic remainder. Third, the limiting equilibrium uniquely minimizes the quadratic perturbation energy among compatible, strictly stratified states with the same periodic label distribution. Finally, we establish sharp algebraic decay in suitable lower-order Sobolev norms for both velocity components and for the temperature and pressure differences from equilibrium. Exact nonlinear solutions provide matching lower bounds, establishing optimality for uniform estimates over the small-data class. The proofs combine weighted energy methods, spectral analysis, and inverse-map identities.
\end{abstract}

\medskip
\noindent\textbf{Keywords.}
Primitive equations; stable stratification; global stability;
asymptotic stratification; sharp decay.

\medskip
\noindent\textbf{2020 Mathematics Subject Classification.}
35B35, 35B40, 35Q35, 35Q86, 76D50.

%============================================================================================================================================================
% Section 1: Introduction
%============================================================================================================================================================

\section{Introduction}

The primitive equations (PE) describe large-scale atmospheric and
oceanic flows, whose ratio of vertical to horizontal length scales is
small \cite{Vallis2017}. Under the hydrostatic approximation, the
vertical momentum equation is replaced by the hydrostatic balance
between the vertical pressure gradient and gravity. The mathematical
formulation of the PE was developed by Lions, Temam, and Wang
\cite{LionsTemamWang1992}. For suitably scaled anisotropic
Navier--Stokes systems, the hydrostatic limit has been justified within
a weak-solution framework \cite{AzeradGuillen2001} and via global,
uniform-in-time strong convergence \cite{LiTiti2019}.

We investigate the long-time dynamics of the two-dimensional non-rotating primitive equations with positive horizontal and vertical viscosities and without thermal diffusion. The reference state is motionless and stably stratified, with an affine temperature profile. The central questions are threefold: whether small perturbations remain uniformly bounded at their initial Sobolev order, which hydrostatic equilibrium is selected by the initial temperature, and how rapidly the solution converges to that equilibrium. Since temperature is transported without smoothing, relaxation must arise from its coupling to the viscous horizontal momentum equation.

Let \(g,\beta,\mu_x,\mu_z>0\), where \(g\) denotes the gravitational acceleration, \(\beta\) the coefficient of thermal expansion, and \(\mu_x,\mu_z\) the horizontal and vertical viscosities. We write \(V,W,T,\pi\) for the total horizontal velocity, vertical velocity, temperature, and kinematic pressure, respectively. For a reference temperature \(T_0\), the hydrostatic balance and the temperature transport equation read
\[
 \partial_z\pi=-g\bigl(1-\beta(T-T_0)\bigr),
 \qquad
 \partial_tT+V\partial_xT+W\partial_zT=0.
\]
For \(\gamma>0\), the hydrostatic rest state is
\[
 (V_s,W_s,T_s)=(0,0,T_0+\gamma z),
 \qquad \pi_s(z)=-gz+\frac{g\beta\gamma}{2}z^2.
\]
Set \(V=v\), \(W=w\), \(T=T_s+\theta\), and \(\pi=\pi_s+p\). We consider real-valued perturbations that are periodic on \(\Omega=\mathbb T_x\times\mathbb T_z\), with periods \(L_x,L_z>0\). The background fields are defined on \(\mathbb R\), so the total temperature and pressure need not be vertically periodic. Upon subtracting the background hydrostatic balance and inserting the affine temperature profile into the transport equation, the terms \(g\beta\theta\) and \(\gamma w\) arise in the vertical momentum and temperature equations, respectively. The corresponding perturbation system is
\begin{equation}\label{eq:1.1.1}
\left\{
\begin{aligned}
 &\partial_tv+v\partial_xv+w\partial_zv
   -\mu_x\partial_{xx}v-\mu_z\partial_{zz}v+\partial_xp=0,\\
 &\partial_zp=g\beta\theta,\\
 &\partial_t\theta+v\partial_x\theta+w\partial_z\theta+\gamma w=0,\\
 &\partial_xv+\partial_zw=0.
\end{aligned}
\right.
\end{equation}
The background density decreases with height, and \(N^2:=g\beta\gamma>0\) is the squared buoyancy frequency. Strict stratification of the perturbed state means that \(\gamma+\partial_z\theta>0\).

We denote by \(\langle\cdot\rangle_z\), \(\langle\cdot\rangle_x\), and
\(\langle\cdot\rangle\) the normalized vertical, horizontal, and spatial
averages, respectively, as defined in Subsection~\ref{subsec:2.1}.
The periodicity of \(p\) requires that
\(\langle\theta\rangle_z=0\) for every \(x\).
Incompressibility and the vertically averaged momentum equation imply
that \(\langle v\rangle_z\) is constant in space and time, and a
Galilean transformation reduces it to zero. The preservation of the
temperature constraint then determines the full vertical velocity:
\[
 w=w^0+\langle w\rangle_z,
 \qquad w^0=-\partial_z^{-1}\partial_xv,
 \qquad \langle w\rangle_z=-\frac1\gamma\partial_x\langle v\theta\rangle_z.
\]
Here \(\partial_z^{-1}\) is the periodic primitive with zero
vertical mean. The pressure is recovered from \eqref{eq:2.2.2} under
the normalization \(\langle p\rangle=0\). Neither \(w\) nor \(p\) is
prescribed independently at the initial time.

Global strong solutions are known for the two-dimensional constant-density hydrostatic Navier--Stokes equations in basins of positive depth \cite{BreschKazhikhovLemoine2005}, as well as for the three-dimensional primitive equations with full viscosity and thermal diffusion \cite{CaoTiti2007}. An \(L^p\)-theory for the isothermal system was developed in \cite{HieberKashiwabara2016}. In the presence of horizontal viscosity, global well-posedness also holds with horizontal thermal diffusion for compatible \(H^1\cap L^\infty\) data satisfying \(\partial_zv_0\in L^p\) for some \(p>2\) \cite{CaoLiTiti2017}, and with vertical thermal diffusion under anisotropic regularity assumptions \cite{CaoLiTiti2020}. A regularized Gent--McWilliams--Redi model with nonlinear anisotropic temperature and salinity diffusion was treated in \cite{KornTiti2024}. For an affine thermal equilibrium in a periodic channel with full diffusion and fixed wall temperatures, Jiang and Wang \cite[Theorems~1.1--1.2]{JiangWang2025} proved \(H^2\) exponential decay below a critical temperature difference and \(H^1\) asymptotic stability at the threshold for compatible \(H^2\) perturbations. These temperature-coupled theories rely on scalar diffusion in their estimates.

The hydrostatic theory is also sensitive to the background flow
structure. Renardy \cite{Renardy2009} proved the ill-posedness of the
hydrostatic Euler equations linearized about certain shear flows,
whereas Masmoudi and Wong \cite{MasmoudiWong2012} established local
Sobolev well-posedness under the Rayleigh condition. Bianchini and
Duch\^ene \cite{BianchiniDuchene2024} developed a Sobolev theory and
derived the hydrostatic limit for stably stratified fluids with
isopycnal diffusivity and no momentum viscosity. Bianchini et al.\
\cite{BianchiniCotiZelatiErtzbischoff2025} constructed unstable
stratified shear flows violating the Miles--Howard criterion and proved
both the nonlinear ill-posedness of the hydrostatic Euler--Boussinesq
equations and the failure of the hydrostatic limit near these profiles.

The non-hydrostatic Boussinesq equations (BE) serve as a natural
comparison for relaxation in the absence of scalar diffusion. On
\(\mathbb R^2\), global well-posedness in suitable functional classes is
known under full viscosity \cite{HouLi2005,Chae2006}, horizontal
viscosity \cite{DanchinPaicu2011}, and critical fractional velocity
dissipation \cite{HmidiKeraaniRousset2010}. On smooth bounded domains
with no-slip velocity boundary conditions, regularity in integer Sobolev
spaces of equal order persists over finite time intervals under the
required compatibility conditions \cite{AydinJayanti2025}. Uniform
scalar-gradient bounds, however, need not hold: smooth viscous solutions
on \(\mathbb R^2\) and \(\mathbb T^2\) may exhibit unbounded \(L^2\)
norms of the scalar gradient along sequences of times tending to
infinity \cite{KiselevParkYao2025}. Near Couette flow on
\(\mathbb T\times\mathbb R\), mixing yields global viscous stability for
small, suitably localized Gevrey perturbations
\cite{MasmoudiSaidHouariZhao2022}. In the inviscid case with Richardson
number greater than \(\frac14\), nonlinear damping of nonzero horizontal
modes and shear-buoyancy growth for nondegenerate Gevrey data persist up
to times of order \(\varepsilon^{-2}\) for small perturbations of size
\(\varepsilon\) \cite{BedrossianBianchiniCotiZelatiDolce2023}. Control
over this timescale near a stratified rest state on \(\mathbb R^2\) was
established for small high-order Sobolev data with additional
anisotropic localization and vector-field bounds
\cite{JurjaWidmayer2026}.

Without background shear, velocity decay and \(H^{-1}\) convergence
of the hydrostatic imbalance to zero were obtained on rectangles
and a class of bounded domains with impermeable, zero-vorticity
boundaries \cite{DoeringWuZhaoZheng2018}, and for periodic
perturbations of an affine stratified state
\cite[Theorem~1.4]{TaoWuZhaoZheng2020}.  In the periodic case,
mean-zero, divergence-free \(H^2\) velocity and \(L^2\cap L^\infty\)
scalar data suffice, without smallness.  Further scalar asymptotics
and algebraic velocity decay use additional Wiener and time-weighted
bounds, respectively \cite[Theorems~1.5--1.6]{TaoWuZhaoZheng2020}.
Small-data stability on the infinite strip \(\mathbb R\times(0,1)\)
and slab \(\mathbb R^2\times(0,1)\) was proved under high-order Sobolev,
localization, and boundary compatibility assumptions
\cite{DongSun2022,DongSun2024}.  For bounded domains with Navier-slip
conditions, velocity decay and weak \(L^2\) convergence of the
hydrostatic imbalance to zero were proved in
\cite{BleitnerCarlsonNobili2025}, with strong hydrostatic balance
conditional on strong \(L^2\) scalar convergence.

A closer quantitative comparison is given by Jang and Kim
\cite{JangKim2023}.  In a two-dimensional horizontally periodic
stress-free strip, small \(H^\ell\) data for the full velocity and
scalar, with integer \(\ell\ge5\),
yield global stability and full-scalar \(H^q\) decay to a stratified
profile at rate \((1+t)^{-\frac{\ell-q}{4}}\), \(0\le q\le\ell\).
Their theorem assumes a mean-zero, divergence-free initial velocity
and higher-order wall compatibility, including zero scalar
perturbation at the walls.  The sharpness lower bounds concern the
linearized system \cite[Proposition~7.1]{JangKim2023}, and the authors
observe that the method also applies to periodic boxes.
For Stokes--transport in no-slip periodic channels, the full scalar
converges in \(L^2\) at rate \((1+t)^{-1}\) to its decreasing
rearrangement from small \(H^6\cap H_0^2\) perturbations of an affine
density \cite{DalibardGuillodLeblond2025}, or small
\(H^4\cap H_0^2\) perturbations of \(H^4\) stratified profiles with
uniformly negative vertical derivative
\cite[Theorem~1.6]{Park2025Stokes}.  For viscous BE in a no-slip
periodic channel, Park \cite[Theorem~1.3]{Park2024} obtained the
same scalar rate and rearrangement limit, together with velocity
\(L^2\) decay at that rate, for small \(H^4\cap H_0^2\) perturbations
of profiles whose vertical derivative belongs to \(H^4\) and is
uniformly negative.  His argument uses the potential-energy
characterization of the rearranged density.

For PE, global control must overcome hydrostatic derivative loss
without temperature diffusion.  In \eqref{eq:1.1.1}, the vertical
velocity has no independent momentum energy, and
\((v_0,\theta_0)\in H^s\times H^s\) generally gives only
\(w_0\in H^{s-1}\).  Periodic compatibility further imposes
\[
 \partial_x^s\langle w\rangle_z
 =-\frac1\gamma\partial_x^{s+1}\langle v\theta\rangle_z.
\]
Estimating this mean separately would cost an additional temperature
derivative, so the energy estimates must retain the full diagnostic
velocity and use the coupling.  The possible nonzero temperature
limit also requires terminal-state selection beyond decay of the
horizontal oscillations.  The local Hadamard theory for compatible
data at every integer \(s\ge4\), without smallness, is established
in the companion paper \cite{PaperI}.

We prove that, for every integer \(s\ge4\), small compatible data
generate global solutions with uniform \(H^s\) bounds for the
horizontal velocity and temperature and persistent strict stratification.
The estimates retain the nonlinear diagnostic mean without an
additional initial temperature derivative.  The terminal temperature
is determined by a normalized periodic rearrangement of the initial
label.  Its explicit inverse-map formula yields a quadratic expansion
with a cubic remainder.  An exact energy-difference identity
characterizes the equilibrium as the unique minimizer of the
quadratic perturbation energy among compatible, strictly stratified
states with the same periodic label distribution.  We also establish
decay of the full temperature difference, both velocity components,
and the pressure difference in the stated lower-order Sobolev ranges.
Exact nonlinear solutions establish optimality for uniform estimates
over the small-data class.

Throughout the paper, the parameters
\(g,\beta,\gamma,\mu_x,\mu_z,L_x,L_z\) are fixed and positive.
We use the periodic Sobolev spaces \(H^q=H^q(\Omega)\) and write
\(\|\cdot\|_2=\|\cdot\|_{L^2(\Omega)}\).  Functions of \(z\) are
identified with their constant extensions in \(x\) when norms on
\(\Omega\) are used.  Set
\[
 R:=\gamma+\partial_z\theta,
 \qquad \Theta:=\theta-\langle\theta\rangle_x.
\]
For data at order \(s\), write
\(\varepsilon:=\|v_0\|_{H^s}+\|\theta_0\|_{H^s}\) and define the
viscous dissipation by
\begin{equation}\label{eq:1.2.1}
 \mathcal D_s(t):=\mu_x\|\partial_xv(t)\|_{H^s}^2
                 +\mu_z\|\partial_zv(t)\|_{H^s}^2.
\end{equation}
For integers \(n\ge0\), the full-derivative pairing and norm are
\begin{equation}\label{eq:1.2.2}
 \langle f,g\rangle_{H^n,\mathrm d}
 :=\sum_{|\alpha|\le n}\int_\Omega
      \partial^\alpha f\,\overline{\partial^\alpha g}\,\mathrm dx\,\mathrm dz,
 \qquad
 \|f\|_{H^n,\mathrm d}^2:=\langle f,f\rangle_{H^n,\mathrm d}.
\end{equation}
Here \(\partial^\alpha=\partial_x^{\alpha_1}\partial_z^{\alpha_2}\),
and the overline denotes complex conjugation.  This norm is equivalent
to the Fourier Sobolev norm, but is kept distinct in exact differentiated
energy identities.
For Fourier frequencies \(k\in\frac{2\pi}{L_x}\mathbb Z\) and
\(m\in\frac{2\pi}{L_z}\mathbb Z\), define the nonnegative multiplier
\begin{equation}\label{eq:1.2.3}
 \widehat{\mathcal Mf}(k,m):=\rho(k,m)\widehat f(k,m),
 \qquad
 \rho(k,m):=\begin{cases}
 \dfrac{N^2k^2}{(\mu_xk^2+\mu_zm^2)m^2},&m\ne0,\\[0.3em]
 0,&m=0.
 \end{cases}
\end{equation}
The hat denotes a Fourier coefficient.  For \(\alpha>0\),
\(\mathcal M^\alpha\) has symbol \(\rho^\alpha\), which vanishes
when \(k=0\) or \(m=0\).  In particular, the temperature dissipation
below does not control the horizontal mean.
The three theorems concern global stability, terminal-state selection,
and sharp decay, respectively.

\begin{theorem}[Global stability and stratification]
\label{thm:1.1}
Let \(s\ge4\) be an integer.  There exist constants
\(\varepsilon_{\mathrm g}>0\) and \(C>0\), depending only on \(s\),
the fixed physical parameters, and the periods, such that every
real-valued initial datum
\((v_0,\theta_0)\in H^s(\Omega)\times H^s(\Omega)\) satisfying
\[
 \langle v_0\rangle_z=\langle \theta_0\rangle_z=0,
 \qquad
 \gamma+\partial_z\theta_0\ge\frac\gamma2,
 \qquad
 \varepsilon\le\varepsilon_{\mathrm g}
\]
gives rise to a unique global \(H^s\)-strong solution of
\eqref{eq:1.1.1}, with the pressure normalized by
\(\langle p\rangle=0\).  This solution has
the following properties:
\begin{enumerate}
\item[(i)]
The solution satisfies
\begin{equation}\label{eq:1.2.4}
 \inf_{t\ge0}\inf_\Omega
     \bigl(\gamma+\partial_z\theta(t)\bigr)\ge\frac\gamma2,
\end{equation}
\begin{equation}\label{eq:1.2.5}
 \sup_{t\ge0}\bigl(\|v(t)\|_{H^s}^2+\|\theta(t)\|_{H^s}^2\bigr)+\int_0^\infty
    \left(\mathcal D_s(t)
          +\|\mathcal M^{\frac12}\Theta(t)\|_{H^s,\mathrm d}^2\right)
          \,\mathrm dt
 \le C\varepsilon^2,
\end{equation}
and
\begin{equation}\label{eq:1.2.6}
 \int_0^\infty
 \left(\|\nabla v(t)\|_{L^\infty}
       +\|w^0(t)\|_{W^{2,\infty}}
       +\|\langle w\rangle_z(t)\|_{W_x^{2,\infty}}\right)\,\mathrm dt
 \le C\varepsilon.
\end{equation}
In particular,
\(v,\theta\in C([0,\infty);H^s)\) and
\(v\in L^2(0,\infty;H^{s+1})\).

\item[(ii)]
There exists a unique limiting profile \(S_\infty\in H_z^s\) with
zero vertical mean such that
\begin{equation}\label{eq:1.2.7}
p_\infty:=g\beta\partial_z^{-1}S_\infty,\qquad \|S_\infty\|_{H_z^s}\le C\varepsilon,
 \qquad
 \gamma+\partial_zS_\infty(z)\ge\frac\gamma2,
 \qquad z\in\mathbb T_z,
\end{equation}
and
\begin{equation}\label{eq:1.2.8}
 \|v(t)\|_{H^s}+\|w(t)\|_{H^{s-1}}
   \longrightarrow0,
 \qquad t\to\infty.
\end{equation}
The limiting temperature admits the following explicit expression:
\begin{equation}\label{eq:1.2.9}
S_\infty
=\langle\theta_0\rangle_x-\partial_z\int_0^\infty
\langle w(t)\Theta(t)\rangle_x,\mathrm dt
\quad\text{in }H_z^{s-3}.
\end{equation}
where the time integral converges absolutely in \(H_z^{s-2}\), and
\begin{equation}\label{eq:1.2.10}
 \|S_\infty-\langle\theta_0\rangle_x\|_{H_z^{s-3}}\le C\varepsilon^2.
\end{equation}
For every \(0<\delta\le s\), we have
\begin{equation}\label{eq:1.2.11}
 \|\theta(t)-S_\infty\|_{H^{s-\delta}}
 +\|p(t)-p_\infty\|_{H^{s-\delta}}
 \longrightarrow0,
 \qquad t\to\infty.
\end{equation}
Here, \(p_\infty\in H_z^{s+1}\), and
\((0,0,S_\infty,p_\infty)\) is a compatible hydrostatic equilibrium.
\end{enumerate}
\end{theorem}

\begin{remark}
The limiting total temperature is given by
\[
 T_\infty(z)=T_s(z)+S_\infty(z)=T_0+\gamma z+S_\infty(z),
 \qquad T_\infty'(z)\ge\frac{\gamma}{2}>0.
\]
Thus $T_\infty(z)$ increases with height, but its vertical
gradient need not be constant.
Since \(S_\infty\) is periodic with zero mean, \(T_\infty\) is affine
if and only if \(S_\infty\equiv0\), in which case \(T_\infty=T_s\).
\end{remark}

We next identify the terminal profile from the initial temperature and
quantify its difference from the initial horizontal mean.

\begin{theorem}[Terminal profile and expansion]
\label{thm:1.2}
Assume the hypotheses of Theorem~\ref{thm:1.1}, and let
\(S_\infty\) and \(p_\infty\) be the limiting profiles given there.
The following assertions hold. All constants depend only on
\(s\), the fixed physical parameters, and the periods.
\begin{enumerate}
\item[(i)] \emph{Selection of the terminal profile.}
The limiting temperature is determined by the initial temperature as
follows. Define
\begin{equation}\label{eq:1.2.12}
 \Phi_0(x,z):=z+\frac{\theta_0(x,z)}{\gamma},
 \quad
 \Psi_0(x,\eta):=\Phi_0(x,\cdot)^{-1}(\eta),
 \quad
 \Psi_*(\eta):=\frac{1}{L_x}\int_0^{L_x}\Psi_0(x,\eta)\,\mathrm dx.
\end{equation}
Here the inverses are taken on \(\mathbb R\), with the lift of
\(\Phi_0\) fixed by the displayed definition, and the horizontal average
is taken at fixed \(\eta\). Then \(\Psi_*\) is a strictly increasing
\(C^1\) diffeomorphism of \(\mathbb R\), and
\begin{equation}\label{eq:1.2.13}
 S_\infty(z)=\gamma\bigl(\Psi_*^{-1}(z)-z\bigr).
\end{equation}
Differentiating this identity yields
\begin{equation}\label{eq:1.2.14}
 \gamma+\partial_zS_\infty(z)
   =\frac{\gamma}{\Psi_*'(\Psi_*^{-1}(z))},
\end{equation}
and therefore
\begin{equation}\label{eq:1.2.15}
 \min_\Omega R(0)
 \le\gamma+\partial_zS_\infty(z)
 \le\max_\Omega R(0)
 \qquad (z\in\mathbb T_z).
\end{equation}
\item[(ii)] \emph{Quadratic expansion.}
Set \(\Theta_0:=\theta_0-\langle\theta_0\rangle_x\). The quadratic
correction \(-\frac1{2\gamma}\partial_z\langle\Theta_0^2\rangle_x\)
satisfies
\begin{equation}\label{eq:1.2.16}
 \left\|\frac1{2\gamma}\partial_z\langle\Theta_0^2\rangle_x\right\|_{H_z^{s-1}}
 \le C\|\theta_0\|_{H^s}^2.
\end{equation}
It has zero vertical mean, and moreover
\begin{equation}\label{eq:1.2.17}
\begin{aligned}
 &\left\|S_\infty-\langle\theta_0\rangle_x
       +\frac1{2\gamma}\partial_z\langle\Theta_0^2\rangle_x\right\|_{H_z^{s-3}}
 +\left\|p_\infty-g\beta\partial_z^{-1}
       \left(\langle\theta_0\rangle_x
             -\frac1{2\gamma}\partial_z\langle\Theta_0^2\rangle_x\right)\right\|_{H_z^{s-2}}\\
 \le{}& C\|\theta_0\|_{H^s}^3\le C\varepsilon^3.
\end{aligned}
\end{equation}
For the linear solution with initial data
\((\Pi_{\ne0}^xv_0,\Theta_0)\) and
\(w^{\mathrm{lin}}=-\partial_z^{-1}\partial_xv^{\mathrm{lin}}\),
where \(\Pi_{\ne0}^x\) denotes the projection defined in \eqref{eee},
the same correction admits the representation
\begin{equation}\label{eq:1.2.18}
 -\frac1{2\gamma}\partial_z\langle\Theta_0^2\rangle_x=-\partial_z\int_0^\infty
       \langle w^{\mathrm{lin}}(t)\Theta^{\mathrm{lin}}(t)\rangle_x
       \,\mathrm dt
 \quad\text{in }H_z^{s-3}.
\end{equation}
The time integral converges absolutely in \(H_z^{s-2}\).
\end{enumerate}
\end{theorem}

The terminal label \(\Phi_\infty(z):=z+\gamma^{-1}S_\infty(z)\)
is the increasing rearrangement of \(\Phi_0\) modulo \(L_z\),
normalized by \(\langle S_\infty\rangle_z=0\).
More precisely, for every continuous \(L_z\)-periodic function \(f\),
we have
\[
 \langle f(\Phi_\infty)\rangle_z=\langle f(\Phi_0)\rangle.
\]
Within the class of Theorem~\ref{thm:1.1}, and for fixed
\(\gamma\) and periods, \eqref{eq:1.2.13} determines the terminal
temperature from \(\theta_0\) alone, independently of the initial
velocity and the viscosities.

The quadratic term can produce a nonzero terminal perturbation even
when the initial horizontal mean vanishes.  For example, take
\(v_0=0\) and
\(\theta_0(x,z)=\delta\cos(kx)\sin(mz)\), where
\(k=\frac{2\pi}{L_x}\), \(m=\frac{2\pi}{L_z}\), and
\(|\delta|\) is sufficiently small.  Then \eqref{eq:1.2.17} gives
\[
 \left\|S_\infty+\frac{\delta^2m}{4\gamma}\sin(2mz)\right\|_{H_z^{s-3}}
 \le C|\delta|^3,
\]
with \(C\) independent of \(\delta\).  Thus
\(S_\infty\not\equiv0\) for sufficiently small nonzero \(\delta\).

\begin{remark}
By Proposition~\ref{prop:5.1}, the terminal equilibrium
\((0,S_\infty)\) uniquely minimizes the quadratic perturbation energy
among compatible, strictly stratified \(H^s\times H^s\) states whose
temperature labels have the same distribution modulo \(L_z\) as
\(\Phi_0\).
\end{remark}

We next give the componentwise decay rates, after a possible
reduction of the small-data threshold.

\begin{theorem}[Nonlinear decay]
\label{thm:1.3}
Let \(s\ge4\) be an integer. There exists a constant
\(\varepsilon_{\mathrm d}\in(0,\varepsilon_{\mathrm g}]\), depending
only on \(s\), the fixed physical parameters, and the periods, such that
the following holds.
Assume the hypotheses of Theorem~\ref{thm:1.1}, and suppose that
\(\varepsilon\le\varepsilon_{\mathrm d}\). Let
\((v,w,\theta,p)\) be the corresponding global strong solution, and let
\(S_\infty\) and \(p_\infty\) be its limiting profiles.
For each \(r\), let \(C_r>0\) denote a constant depending only on
\(s,r\), the fixed physical parameters, and the periods. Then the
following estimates hold.
For every \(0\le r<s\) and all \(t\ge0\),
\begin{equation}\label{eq:1.2.19}
 \|\theta(t)-S_\infty\|_{H^r}+\|\Theta(t)\|_{H^r}
 \le C_r\varepsilon(1+t)^{-\frac{s-r}{4}}.
\end{equation}
For every \(0\le r\le\max\{0,s-5\}\) and all \(t\ge0\),
\begin{equation}\label{eq:1.2.20}
 \|v(t)\|_{H^r}\le C_r\varepsilon(1+t)^{-\frac{s-r+3}{4}},
\end{equation}
\begin{equation}\label{eq:1.2.21}
 \|w(t)\|_{H^r}\le C_r\varepsilon(1+t)^{-\frac{s-r+4}{4}},
\end{equation}
\begin{equation}\label{eq:1.2.22}
 \|p(t)-p_\infty\|_{H^r}\le C_r\varepsilon(1+t)^{-\frac{s-r+1}{4}}.
\end{equation}
Moreover, for every \(r\) in the admissible range of the corresponding
estimate, the decay exponent in that estimate is optimal for uniform
estimates over the stated \(H^s\) small-data class. This optimality also
holds for the two temperature norms in \eqref{eq:1.2.19} separately.
\end{theorem}
\begin{table}[H]
\setlength{\abovecaptionskip}{5pt}
\setlength{\belowcaptionskip}{3pt}
\noindent
Table~\ref{tab:1.1} summarizes selected two-dimensional small-data
asymptotic stability results without scalar diffusion.  The decay
rates concern the full scalar difference from its terminal profile.
\par
\centering
\caption{Stability and decay without scalar diffusion.}
\label{tab:1.1}
\small
\setlength{\tabcolsep}{3.5pt}
\renewcommand{\arraystretch}{1.05}
\begin{tabularx}{\linewidth}{@{}
 >{\raggedright\arraybackslash}p{0.23\linewidth}
 >{\raggedright\arraybackslash}p{0.24\linewidth}
 >{\centering\arraybackslash}p{0.15\linewidth}
 >{\raggedright\arraybackslash}X@{}}
\toprule
\textbf{Model and setting} &
\textbf{Data and control} &
\textbf{Scalar \(L^2\) rate} &
\textbf{Terminal profile}\\
\midrule
PE (this work)\newline periodic box
&
\(H^s\), \(s\ge4\)\newline uniform \(H^s\)
&
\((1+t)^{-\frac{s}{4}}\)\({}^{\dagger}\)
&
Normalized periodic rearrangement\newline
(explicit inverse map)\\
\midrule
BE \cite{JangKim2023}\newline stress-free strip
&
\(H^\ell\), \(\ell\ge5\)\newline uniform \(H^\ell\)
&
\((1+t)^{-\frac{\ell}{4}}\)\({}^{\ddagger}\)
&
Stratified limit\newline (time-integrated flux)\\
\addlinespace[0.3em]
BE \cite{Park2024}\newline no-slip channel
&
\(H^4\cap H_0^2\)\newline uniform \(H^4\)
&
\((1+t)^{-1}\)
&
Decreasing vertical rearrangement\\
\addlinespace[0.3em]
BE \cite{DongSun2022}\newline infinite slip strip
&
\(H^m\times H^{m+1}\), \(m\ge33\)\newline
time-weighted control
&
\((1+t)^{-\frac{1}{4}}\)
&
Prescribed affine background\\
\midrule
Stokes--transport \cite{DalibardGuillodLeblond2025}\newline
no-slip channel
&
\(H^6\cap H_0^2\)\newline uniform \(H^4\)
&
\((1+t)^{-1}\)
&
Decreasing vertical rearrangement\\
\addlinespace[0.3em]
Stokes--transport \cite{Park2025Stokes}\newline
no-slip channel
&
\(H^4\cap H_0^2\)\newline uniform \(H^4\)
&
\((1+t)^{-1}\)
&
Decreasing vertical rearrangement\\
\bottomrule
\end{tabularx}
\par\smallskip
\begin{minipage}{\linewidth}
\footnotesize
Regularity concerns perturbations: horizontal velocity and temperature
for PE, full velocity and scalar for BE, and scalar alone for
Stokes--transport.  In \cite{DongSun2022}, it concerns vorticity and
scalar, also small in \(W^{5,1}\times W^{8,1}\), with scalar decay in
\(H^4\).  Domains are horizontally periodic except
\(\mathbb R\times(0,1)\) in \cite{DongSun2022}.
Data indices are integers; the respective incompressibility,
mean, and boundary conditions are understood.
\(\dagger\): lower bounds from exact nonlinear solutions.
\(\ddagger\): linearized lower bounds.
\end{minipage}
\end{table}

At \(s=4\), the scalar \(L^2\) rate coincides with that obtained in
\cite{Park2024,DalibardGuillodLeblond2025,Park2025Stokes}.
For the PE, achieving this rate requires global
control at the initial Sobolev regularity despite the derivative loss
in the diagnostically determined vertical velocity. The periodic
inverse-map formula further determines the equilibrium directly from
the initial temperature and quantifies its deviation from the initial
horizontal mean. Thus, the present contribution combines global
control of the hydrostatic system, quantitative selection of the
terminal state, and sharp componentwise relaxation.

The hydrostatic constraint creates a derivative loss without an
independent vertical-velocity energy.  In particular, the highest
derivative of the compatible mean cannot be estimated separately
from \(H^s\times H^s\) data.  We use the stratification-weighted
cancellation and horizontal compensation of the local theory
\cite{PaperI}.  Testing the differentiated temperature equation
against \(g\beta R^{-1}\partial^\alpha\theta\), together with the
differentiated momentum energy identity, cancels the principal
hydrostatic coupling.  Horizontal compensation places the remaining
mean derivative on a viscous velocity factor.  A velocity--temperature
cross term supplies the indirect dissipation represented by
\(\mathcal M\), giving a same-order estimate that retains the full
diagnostic velocity.

The passage from local control to global stability requires this
energy estimate and the diagnostic time integral \eqref{eq:1.2.6}
to close simultaneously.  At \(s=4\), \(R^{-1}\in H^3(\Omega)\),
whereas
\[
 H^3(\Omega)\not\hookrightarrow W^{2,\infty}(\Omega).
\]
Mixed-norm pairings and weighted products with \(H^3\) coefficients
replace the missing uniform bounds on second derivatives.
We jointly estimate the high-order energy, spectral decay, and
diagnostic velocity integral on finite intervals of the local
lifespan.  The fast response estimates and weighted interpolation
make the velocity coefficient integrable, with constants independent
of the interval length.  This closes the fourth-order bootstrap and
propagates higher integer regularity without an additional initial
temperature derivative.

The temperature mean is not controlled by the indirect dissipation.
Let \(\Psi(t,x,\eta)\) denote the inverse of
\(z\mapsto z+\gamma^{-1}\theta(t,x,z)\).
The full compatible velocity admits a periodic stream function,
which gives, at fixed \(\eta\),
\[
 \langle\Psi(t,\cdot,\eta)\rangle_x
 =\langle\Psi_0(\cdot,\eta)\rangle_x
 =\Psi_*(\eta).
\]
The mean constraint fixes the lift.  Consequently, the initial
temperature determines the terminal profile through a normalized
periodic rearrangement.  This identity yields the quadratic
correction with a cubic remainder and the constrained energy
characterization.  Comparison of the inverse maps also gives
\[
\|\theta-S_\infty\|_2\le C\|\Theta\|_2,
\]
allowing fluctuation
decay to control the full temperature difference.

For fixed nonzero horizontal frequency, the slow linear decay rate
is of order \(|m|^{-4}\) as the vertical frequency \(|m|\) tends to
infinity, as in the viscous Boussinesq comparison \cite{JangKim2023}.
Thus viscosity does not produce a uniform spectral gap even for
\(\Theta\).  Spectral and forced response estimates retain the
anisotropic multipliers that determine the different velocity rates,
while the pressure is recovered from its diagnostic formula.
Exact nonlinear plane-wave solutions, for which the transport terms
vanish, provide matching lower bounds along increasing vertical
frequencies and corresponding times tending to infinity.
Hence none of the stated exponents can be increased in estimates
uniform over the small-data class.

The paper is organized as follows.  Section~\ref{sec:2} gives the
functional setting, diagnostic formulation, and common estimates.
Section~\ref{sec:3} establishes the linear theory, spectral bounds,
and forced response estimates.  Section~\ref{sec:4} proves global
stability and asymptotic stratification, together with the
integer regularity obstruction for the diagnostic time integral.
Section~\ref{sec:5} identifies the terminal profile, derives its
quadratic expansion, and proves its energy-minimizing property.
Section~\ref{sec:6} establishes the sharp decay estimates.

%============================================================================================================================================================
% End of Section 1
%============================================================================================================================================================

%============================================================================================================================================================
% Section 2: Preliminaries
%============================================================================================================================================================

\section{Preliminaries}\label{sec:2}
We fix the functional setting, derive the diagnostic formulation,
and recall the local solution theory.  We then collect the multiplier
bounds and time estimates shared by the linear and nonlinear proofs.
The estimates specific to global stability and terminal-state selection
are proved in the corresponding sections.

\subsection{Notation and elementary estimates}\label{subsec:2.1}
For \(f=f(x,z,t)\), define
\begin{align*}
  &\langle f\rangle_z(x,t):= \frac{1}{L_z}\int_0^{L_z}f(x,z,t)\,\mathrm dz,\\
  &\langle f\rangle_x(z,t):= \frac{1}{L_x}\int_0^{L_x}f(x,z,t)\,\mathrm dx,\\
  &\langle f\rangle(t):= \frac{1}{L_xL_z}\int_0^{L_x}\int_0^{L_z}
     f(x,z,t)\,\mathrm dz\,\mathrm dx.
\end{align*}

We then define the projections \(\Pi_{\ne0}^z\),
\(\Pi_0^x\) , \(\Pi_{\ne0}^x\) and \(\mathcal Q\) by
\begin{align}\label{eee}
\Pi_{\ne0}^zf=f-\langle f\rangle_z,\quad
\Pi_0^xf=\langle f\rangle_x,\quad
\Pi_{\ne0}^xf=f-\langle f\rangle_x, \quad
\mathcal Q=\Pi_{\ne0}^x\Pi_{\ne0}^z.
\end{align}
Set
\[
 \mathcal K=\frac{2\pi}{L_x}\mathbb Z\times\frac{2\pi}{L_z}\mathbb Z,
 \quad \mathcal K_* =\{(k,m)\in\mathcal K:m\ne0\},
 \quad \mathcal K_*^x=\{(k,m)\in\mathcal K:k,m\ne0\}.
\]
With \(\langle k,m\rangle=(1+k^2+m^2)^{\frac12}\), our Fourier
normalization and Sobolev norm are
\begin{equation*}
  \widehat f(k,m,t)
  :=
  \frac{1}{L_xL_z}
  \int_0^{L_x}\!\int_0^{L_z}
  f(x,z,t)e^{-i(kx+mz)}\,\mathrm{d}z\,\mathrm{d}x,
\end{equation*}

\[
 f(x,z)=\sum_{(k,m)\in\mathcal K}\widehat f(k,m)e^{i(kx+mz)},
 \qquad \|f\|_2^2=L_xL_z\sum_{\mathcal K}|\widehat f|^2
\]
and
\begin{equation*}
  \|f\|_{H^q}^2
  :=L_xL_z\sum_{(k,m)\in\mathcal K}
  \langle k,m\rangle^{2q}|\widehat f(k,m)|^2,
\end{equation*}
where $ q\in\mathbb R$.
The Fourier inner product is linear in its first argument.
For every integer \(n\ge0\), the full-derivative norm and pairing
defined in \eqref{eq:1.2.2} satisfy
\begin{equation}\label{eq:2.1.1}
  c_n\|f\|_{H^n}^2
  \le \|f\|_{H^n,\mathrm d}^2
  \le C_n\|f\|_{H^n}^2
\end{equation}
and
\begin{equation}\label{eq:2.1.2}
  |\langle f,g\rangle_{H^n,\mathrm d}|
  \le C_n\|f\|_{H^{n-1}}\|g\|_{H^{n+1}}.
\end{equation}
These estimates follow from the Fourier representation and
Cauchy--Schwarz.

The last inequality defines the full-derivative pairing on
\(H^{n-1}\times H^{n+1}\).  Exact differentiated energies use this
pairing, whereas interpolation uses the Fourier norm.
Write \(H_x^q=H^q(\mathbb T_x)\) and \(H_z^q=H^q(\mathbb T_z)\).
Functions of one variable are extended constantly when norms on
\(\Omega\) are used.  The anisotropic space
\(H_z^2(\Omega)=L_x^2H_z^2\) has norm
\begin{equation*}
  \|f\|_{H_z^2(\Omega)}^2
  :=L_xL_z\sum_{(k,m)\in\mathcal K}
        (1+m^2)^2|\widehat f(k,m)|^2.
\end{equation*}

On a fixed interval \([0,T]\), we write \(C_tX=C([0,T];X)\),
\(L_t^pX=L^p(0,T;X)\), and \(W_t^{1,p}X=W^{1,p}(0,T;X)\)
for a Banach space \(X\).  The notation \(H_t^1X\) means \(W_t^{1,2}X\).
We set \(\|U\|_{H^q\times H^q}=\|v\|_{H^q}+\|\theta\|_{H^q}\)
for \(U=(v,\theta)\).  Constants \(C,c>0\), including their indexed
variants, depend only on the fixed physical parameters, periods,
and stated regularity indices unless specified otherwise.  They may
change from line to line.  We write \(f\lesssim g\) for \(f\le Cg\).

The operator \(\partial_z^{-1}\) has symbol \((im)^{-1}\) for
\(m\ne0\) and zero for \(m=0\).  Thus
\(\partial_z\partial_z^{-1}=\partial_z^{-1}\partial_z=\Pi_{\ne0}^z\).
For \(q\ge0\),
\begin{equation}\label{eq:2.1.3}
 \|\partial_z^{-1}f\|_{H^q}\le C\|f\|_{H^q},
 \qquad
 \|\partial_z^{-1}h\|_{H_z^{q+1}}\le C\|h\|_{H_z^q}.
\end{equation}
In the second bound \(h=h(z)\), and the inverse removes its constant
mode.  For the dissipation \(\mathcal D_s\) defined in
\eqref{eq:1.2.1}, zero vertical mean and periodic
Poincar\'e give
\begin{equation}\label{eq:2.1.4}
 \mathcal D_s(t)\ge c\|v(t)\|_{H^{s+1}}^2.
\end{equation}

For integer \(s\ge r\ge3\) and \(0\le q\le s\), the Sobolev
product estimate \cite{KatoPonce1988,BahouriCheminDanchin2011} is
\begin{equation}\label{eq:2.1.5}
  \|fg\|_{H^q}
  \le C\bigl(\|f\|_{H^r}\|g\|_{H^q}
             +\|g\|_{H^r}\|f\|_{H^q}\bigr).
\end{equation}
Averaging is bounded in the corresponding Sobolev spaces, and all
four projections are contractions on \(H^q(\Omega)\).

For \(q\ge0\), define the Wiener norm by
\[
 \|f\|_{\mathcal W^q}:=\sum_{(k,m)\in\mathcal K}
        \langle k,m\rangle^q|\widehat f(k,m)|.
\]
The one-dimensional norms \(\mathcal W_x^q\) and \(\mathcal W_z^q\)
use the corresponding normalized Fourier coefficients and are
preserved by constant extension.  Averaging is contractive from
\(\mathcal W^q\) to \(\mathcal W_x^q\) or \(\mathcal W_z^q\),
as appropriate.  Each of the four projections above is also
contractive on \(\mathcal W^q\), since it restricts the Fourier sum.

For \(q,r\ge0\), the weighted convolution inequality gives
\[
 \|fg\|_{\mathcal W^q}
 \le C_q\bigl(\|f\|_{\mathcal W^q}\|g\|_{\mathcal W^0}
             +\|f\|_{\mathcal W^0}\|g\|_{\mathcal W^q}\bigr).
\]
The corresponding Sobolev product bound is
\begin{equation}\label{eq:2.1.6}
 \|fg\|_{H^r}
 \le C_r\bigl(\|f\|_{H^r}\|g\|_{\mathcal W^0}
             +\|f\|_2\|g\|_{\mathcal W^r}\bigr).
\end{equation}
Both estimates hold whenever their right-hand sides are finite and
follow from \(\langle\xi+\eta\rangle^q\le
C_q(\langle\xi\rangle^q+\langle\eta\rangle^q)\) and Young's
convolution inequality.  In particular, \eqref{eq:2.1.6} bounds
multiplication by \(\varphi\in\mathcal W^r\) on \(H^r\) with
operator norm at most \(C_r\|\varphi\|_{\mathcal W^r}\).
The same bound holds from \(H^r\) to \(H_x^r\) after vertical
averaging.

The Wiener norms also control uniform derivatives.  For \(q\ge0\)
and integers \(n\ge0\),
\begin{equation}\label{eq:2.1.7}
 \|f\|_{\mathcal W^q}\le C_q\|f\|_{H^{q+2}},
 \qquad
 \|f\|_{W^{n,\infty}}\le C_n\|f\|_{\mathcal W^n}.
\end{equation}
The first inequality follows from Cauchy--Schwarz and
\(\sum_{\mathcal K}\langle k,m\rangle^{-4}<\infty\), and the
second from absolute convergence of the differentiated Fourier
series.  The product estimates also hold in one dimension, where
one additional Sobolev derivative suffices for the Wiener embedding.
In particular,
\begin{equation}\label{eq:2.1.8}
 \|h\|_{\mathcal W_x^2}\le C\|h\|_{H_x^3},
 \qquad \|h\|_{\mathcal W_z^2}\le C\|h\|_{H_z^3}.
\end{equation}

The mixed-norm estimates used below are
\begin{equation}\label{eq:2.1.9}
 \|fg\|_2\le\|f\|_{L_x^\infty L_z^2}\|g\|_{L_x^2L_z^\infty},
 \qquad
 \|g\|_{L_x^2L_z^\infty}\le C(\|g\|_2+\|g_z\|_2).
\end{equation}
Also \(\|f\|_{L_x^\infty L_z^2}\le C(\|f\|_2+\|f_x\|_2)\).
These follow from H\"older and the one-dimensional Sobolev embedding.

For \(|\alpha|=n\), write
\([\partial^\alpha,f]h=\partial^\alpha(fh)-f\partial^\alpha h\).
For smooth periodic \(f,g\), integer \(n\ge1\), and
\(\ell\in\{x,z\}\),
\begin{equation}\label{eq:2.1.10}
  \|[\partial^\alpha,f]\partial_\ell g\|_2
  \le C_n\left(
       \|\nabla f\|_{L^\infty}\|g\|_{H^n}
       +\|f\|_{H^n}\|\partial_\ell g\|_{L^\infty}\right).
\end{equation}

For \(n\ge2\), the internal terms satisfy
\begin{equation}\label{eq:2.1.11}
  \sum_{0<\eta<\alpha}\binom\alpha\eta
     \|(\partial^\eta f)\partial_\ell\partial^{\alpha-\eta}g\|_2
  \le C_n\left(
       \|\nabla f\|_{L^\infty}\|g\|_{H^n}
       +\|f\|_{H^{n-1}}\|g\|_{W^{2,\infty}}\right).
\end{equation}

Here \(0<\eta<\alpha\) means componentwise inclusion with
\(\eta\ne0,\alpha\).  We use the following consequence of the periodic
Gagliardo--Nirenberg inequality:
\begin{equation*}
  \|(\partial^\eta F)(\partial^\zeta G)\|_2
  \le C_m\left(
        \|F\|_{L^\infty}\|G\|_{H^m}
        +\|F\|_{H^m}\|G\|_{L^\infty}\right),
\end{equation*}

for \(|\eta|+|\zeta|=m\).  It follows from the periodic
Gagliardo--Nirenberg inequality and weighted Young inequality,
including the zero-order cases by H\"older.  For the full
commutator choose \(e_i\le\eta\) and apply this bound to
\(\partial_i f,\partial_\ell g\) with total remaining order \(n-1\).
For the internal terms choose also \(e_j\le\alpha-\eta\) and write
\[
 (\partial^\eta f)\partial_\ell\partial^{\alpha-\eta}g
 =(\partial^{\eta-e_i}\partial_i f)
   (\partial^{\alpha-\eta-e_j}\partial_j\partial_\ell g).
\]
The remaining order is \(n-2\), which proves \eqref{eq:2.1.11}.
These conventions also fix the diagnostic reconstruction used next.

\subsection{Diagnostic formulation and local solutions}\label{subsec:2.2}
Incompressibility and the vertically averaged temperature equation give
\(w^0=-\partial_z^{-1}\partial_xv\) and
\(\gamma\langle w\rangle_z=-\partial_x\langle v\theta\rangle_z\).
Consequently,
\begin{equation}\label{eq:2.2.1}
  w=
  -\partial_z^{-1}\partial_xv
  -\frac{1}{\gamma}\partial_x\langle v\theta\rangle_z.
\end{equation}

The hydrostatic relation and the averaged momentum equation give
\(p=g\beta\partial_z^{-1}\theta+\langle p\rangle_z\) and
\(\partial_x\langle p\rangle_z=-\partial_x\langle v^2\rangle_z\).
The normalization \(\langle p\rangle=0\) therefore yields
\begin{equation}\label{eq:2.2.2}
  p
  =
  g\beta\partial_z^{-1}\theta
  -\langle v^2\rangle_z
  +\langle v^2\rangle.
\end{equation}

We always use the full \(w\) in the nonlinearities
\begin{equation*}
 \mathcal N_1=-(v\partial_xv+w\partial_zv),
 \qquad \mathcal N_2=-(v\partial_x\theta+w\partial_z\theta).
\end{equation*}
They contain both quadratic and cubic terms.  The equivalent projected
system is
\begin{equation}\label{eq:2.2.3}
\left\{
\begin{aligned}
  &\partial_tv
-\mu_x\partial_{xx}v
-\mu_z\partial_{zz}v
+g\beta\partial_x\partial_z^{-1}\theta
=\Pi_{\ne0}^z\mathcal N_1(v,\theta),\\
  &\partial_t\theta
-\gamma\partial_z^{-1}\partial_xv
=\Pi_{\ne0}^z\mathcal N_2(v,\theta),\\
  &(v,\theta)|_{t=0}
=(v_0,\theta_0),
\end{aligned}
\right.
\end{equation}

Indeed, its vertical averages preserve the zero-mean constraints, and
\(\langle\mathcal N_1\rangle_z=-\partial_x\langle v^2\rangle_z\),
\(\langle\mathcal N_2\rangle_z=-\partial_x\langle v\theta\rangle_z\).
The diagnostic formulas then recover all four equations in
\eqref{eq:1.1.1}.  No independent datum for \(w\) or \(p\) is needed.
A useful alternative form of the mean velocity is
\begin{equation}\label{eq:2.2.4}
 \langle w\rangle_z=-\frac1\gamma
       \langle w^0\partial_z\theta+v\partial_x\Theta\rangle_z.
\end{equation}
This follows by integrating \(\partial_xv=-\partial_zw^0\) in the
product \(\partial_x\langle v\theta\rangle_z\), and using
\(\langle v\rangle_z=0\).

The local theory is taken from \cite{PaperI}, whose equations and
pressure normalization agree with ours.  Its vertical average is
written here as \(\langle\cdot\rangle_z\).  For every integer \(s\ge4\)
and real datum \(U_0=(v_0,\theta_0)\in H^s\times H^s\) with zero
vertical means and \(\gamma+\partial_z\theta_0\ge\kappa_0>0\), the
Hadamard theorem in \cite[Subsection~2.2]{PaperI} yields a unique
maximal solution on \([0,T_{\max}^{(s)}(U_0))\).
On every compact subinterval \([0,T]\) of this lifespan, the solution satisfies
\begin{equation}\label{eq:2.2.5}
\begin{alignedat}{2}
 &v\in C([0,T];H^s)\cap L^2(0,T;H^{s+1}),
 &\qquad &\theta\in C([0,T];H^s),\\
 &\partial_tv\in L^2(0,T;H^{s-1}),
 &\qquad &\partial_t\theta\in C([0,T];H^{s-1}).
\end{alignedat}
\end{equation}
The equations hold almost everywhere, the mean constraints persist,
and \(R\) has a positive lower bound on \([0,T]\times\Omega\).
The diagnostic variables satisfy \(w\in C_tH^{s-1}\) and
\(p\in C_tH^s\).  Maximality is understood in this strictly
stratified strong-solution class.

The quantitative local bounds in \cite[Subsection~4.1]{PaperI} give a
common existence time, norm and velocity-dissipation bounds, and a
positive stratification lower bound when the initial \(H^s\times H^s\)
norm and the reciprocal stratification are bounded.  The continuation
criterion in \cite[Subsection~4.2]{PaperI}, with
\(T_{\max}=T_{\max}^{(s)}(U_0)\), is
\begin{equation}\label{eq:2.2.6}
 T_{\max}<\infty\quad\Longrightarrow\quad
 \limsup_{t\to T_{\max}^-}
 \left(\|v(t)\|_{H^s}+\|\theta(t)\|_{H^s}
                       +\|R(t)^{-1}\|_{L^\infty}\right)=\infty.
\end{equation}

We also use \cite[Subsections~5.1 and~5.3]{PaperI}.  If compatible
data \(U_{0,n}\to U_0\) in \(H^s\times H^s\), they are strictly
stratified for all large \(n\).  Their solutions exist on every fixed
\([0,T]\subset[0,T_{\max}^{(s)}(U_0))\) for all large \(n\) and
converge there in \(C_t(H^s\times H^s)\).  In particular,
\(T_{\max}^{(s)}(U_0)\le\liminf_nT_{\max}^{(s)}(U_{0,n})\).
For integers \(\sigma\ge s\), solutions from the same
\(H^\sigma\times H^\sigma\) datum agree on their common lifespan.
If an \(H^s\)-strong solution obeys
\(\sup_{t\le T}\|U(t)\|_{H^s\times H^s}\le M\) and \(R\ge\kappa>0\),
then its higher initial regularity persists on \([0,T]\), with
\[
 \sup_{t\le T}\|U(t)\|_{H^\sigma\times H^\sigma}^2
 +\int_0^T\|v(t)\|_{H^{\sigma+1}}^2\,\mathrm dt
 \le Ce^{CT}\|U_0\|_{H^\sigma\times H^\sigma}^2.
\]
Here \(C\) depends only on \(s,\sigma,M,\kappa\) and the fixed
parameters, not on the higher-order norms.  This finite-time bound
is not asserted to be uniform as \(T\to\infty\).

The diagnostic formulas and \eqref{eq:2.1.5} imply, for \(s\ge4\),
\begin{equation}\label{eq:2.2.7}
\begin{aligned}
 &\|w^0\|_{H^{s-1}}\le C\|v\|_{H^s},\\
 &\|w^0\|_{H^s}\le C\|v\|_{H^{s+1}},\\
 &\|w\|_{H^{s-1}}\le C(1+\|\theta\|_{H^s})\|v\|_{H^s},\\
 &\|p\|_{H^s}\le C(\|\theta\|_{H^s}+\|v\|_{H^s}^2).
\end{aligned}
\end{equation}
On bounded subsets of \(H^s\times H^s\),
\begin{equation}\label{eq:2.2.8}
 \|w_1-w_2\|_{H^{s-1}}+\|p_1-p_2\|_{H^s}
 +\sum_{i=1}^2\|\mathcal N_i(U_1)-\mathcal N_i(U_2)\|_{H^{s-1}}
 \le C_M\|U_1-U_2\|_{H^s\times H^s}.
\end{equation}
For example, write \(v_1\theta_1-v_2\theta_2
=(v_1-v_2)\theta_1+v_2(\theta_1-\theta_2)\) and apply the
product estimate.  The other terms use the algebra property of \(H^{s-1}\).

For every compatible strong solution and \(0\le t_1\le t_2\le T\),
\begin{equation}\label{eq:2.2.9}
 \|v(t_2)\|_2^2+\frac{g\beta}{\gamma}\|\theta(t_2)\|_2^2
 +2\int_{t_1}^{t_2}\mathcal D_0(\tau)\,\mathrm d\tau
 =\|v(t_1)\|_2^2+\frac{g\beta}{\gamma}\|\theta(t_1)\|_2^2.
\end{equation}
Here \(\mathcal D_0\) is \eqref{eq:1.2.1} at order zero.
The squared norms are absolutely continuous by
\eqref{eq:2.2.5}.  Incompressibility cancels transport,
and \(\int p_xv=-g\beta\int\theta w\) cancels the temperature coupling.
Equation~\eqref{eq:2.2.1} also yields \(\langle w\rangle_x=0\).
Thus \(\langle w\theta\rangle_x=\langle w\Theta\rangle_x\), and
\begin{equation}\label{eq:2.2.10}
 \partial_t\langle\theta\rangle_x=-\partial_z\langle w\Theta\rangle_x,
 \qquad
 \int_0^{L_z}\langle\theta(t)\rangle_x(z)\,\mathrm dz=0.
\end{equation}

The closure separates the viscous evolution from the hydrostatic
coupling.  We next record the corresponding multiplier bounds and
the time identities used to estimate this coupling.

\subsection{Hydrostatic multipliers and time estimates}
On the zero-vertical-mean subspace, define
\begin{equation*}
  \mathcal L:=-\mu_x\partial_{xx}-\mu_z\partial_{zz},
  \qquad
  \mathcal T:=\partial_x\partial_z^{-1}.
\end{equation*}
Write \(a(k,m):=\mu_xk^2+\mu_zm^2\) for the symbol of
\(\mathcal L\).  The multiplier \(\mathcal M\) and its symbol
\(\rho\) are defined in \eqref{eq:1.2.3}.
Comparison of the symbols gives
\(\mathcal M=N^2\mathcal L^{-1}\mathcal T^2\).

The operators \(\mathcal L,\mathcal T\) are self-adjoint on their
maximal Fourier-multiplier domains.  The operator \(\mathcal L\) has domain
\(H^2\cap\{\langle f\rangle_z=0\}\) and is strictly positive.
The domain of \(\mathcal T\) is determined by
\(\sum_{\mathcal K_*}k^2m^{-2}|\widehat f|^2<\infty\).
Powers of \(\mathcal L\) have symbol \(a^\alpha\) on \(m\ne0\).
Inverse powers, \(\mathcal T\), and their bounded compositions are
extended by zero at \(m=0\).  A bounded composition always means
its Fourier extension, even when its individual factors do not act
on all of \(L^2\).
The symbol in \eqref{eq:1.2.3} is bounded,
so \(\mathcal M\) is a bounded nonnegative operator on \(L^2\).
Its zero-mode convention gives \(\mathcal M\Pi_0^x=0\) and
\(\mathcal M^\alpha f=\mathcal M^\alpha\mathcal Qf\) for \(\alpha>0\).
These operators commute with derivatives and projections.

On \(\mathcal K_*\), \(c\langle k,m\rangle^2\le a\le C\langle k,m\rangle^2\).
For \(k,m\ne0\),
\begin{equation}\label{eq:2.3.1}
  \frac{c_0}{(1+m^2)^2}
  \le
  \rho(k,m)
  \le
  C_0,
\end{equation}
\begin{equation}\label{eq:2.3.2}
  \frac{N^2k^2}{a(k,m)^2m^2}
  \le C\rho(k,m),
  \qquad
  \frac{|k|}{a(k,m)|m|}
  \le C\rho(k,m)^{\frac{1}{2}},
  \qquad
  \frac{k^2}{a(k,m)m^2}=\frac{1}{N^2}\rho(k,m).
\end{equation}
Indeed, \(\rho=N^2(\mu_xm^2+\mu_zm^4k^{-2})^{-1}\), and nonzero
lattice frequencies are bounded away from zero.  For \(q\ge0\),
zero-vertical-mean \(f\), and finite right-hand sides,
\begin{equation}\label{eq:2.3.3}
  \|\mathcal L^{-1}f\|_{H^{q+2}}
  +\|\mathcal L^{-1}\mathcal T f\|_{H^{q+1}}
  \le C\|f\|_{H^q},
\end{equation}
\begin{equation}\label{eq:2.3.4}
  c\|f\|_{H^{q+1}}^2
  \le\|\mathcal L^{\frac{1}{2}}f\|_{H^q}^2
  \le C\|f\|_{H^{q+1}}^2.
\end{equation}
\begin{equation}\label{eq:2.3.5}
  \|\mathcal L^{-1}\mathcal T f\|_{H^{q+1}}
  \le C\|\mathcal M^{\frac{1}{2}}f\|_{H^q},
  \qquad
  \|\mathcal T f\|_{H^q}
  \le C\|\mathcal L^{\frac{1}{2}}f\|_{H^q}.
\end{equation}
For \(q\ge2\) and \(\mathcal Qf=f\),
\begin{equation}\label{eq:2.3.6}
  \|f\|_{H^{q-2}}
  +\|\partial_xf\|_{H^{q-2}}
  +\|\mathcal T f\|_{H^{q-1}}
  \le C\|\mathcal M^{\frac{1}{2}}f\|_{H^q}.
\end{equation}
In particular, if \(n\ge2\) is an integer, \(\alpha_1\ge1\), and
\(\alpha_1+\alpha_2\le n-1\),
\begin{equation}\label{eq:2.3.7}
  \|\partial_x^{\alpha_1}\partial_z^{\alpha_2}f\|_2
  \le C_n\|\mathcal M^{\frac{1}{2}}f\|_{H^n}.
\end{equation}
The coupling satisfies, for \(f\in H^{q+1}\) and \(g\in H^q\),
\begin{equation}\label{eq:2.3.8}
  |\langle\mathcal T f,g\rangle_{H^q}|
  \le\frac1N
  \|\mathcal L^{\frac{1}{2}}f\|_{H^q}
  \|\mathcal M^{\frac{1}{2}}g\|_{H^q}.
\end{equation}
These bounds follow from Parseval and the symbol comparisons
\[
 \frac{\langle k,m\rangle^2k^2}{a^2m^2}\le C\rho,
 \quad \frac{1+k^2}{\rho\langle k,m\rangle^4}\le C,
 \quad \frac{k^2}{m^2\rho\langle k,m\rangle^2}\le C,
 \quad \left|\frac{k}{m}\right|=\frac{a^{\frac12}\rho^{\frac12}}N.
\]
The comparisons containing \(\rho^{-1}\) are used only on
\(\mathcal K_*^x\).  For integer orders the same estimates hold
in the full-derivative norms and pairings, by summing over the
multi-indices.  In \eqref{eq:2.3.8} this preserves the factor \(N^{-1}\).

The preceding bounds determine the dual spaces for the energy
pairings.  We record their time-integrated form.  For \(n\ge0\), \(\ell\in\{0,1\}\), and
\[
 f\in C([0,T];H^n)\cap L^2(0,T;H^{n+\ell}),
 \qquad f_t\in L^2(0,T;H^{n-\ell}),
\]
the squared norm is absolutely continuous and satisfies
\begin{equation}\label{eq:2.3.9}
 \|f(t_2)\|_{H^n}^2-\|f(t_1)\|_{H^n}^2
 =2\operatorname{Re}\int_{t_1}^{t_2}
       \langle\partial_tf(t),f(t)\rangle_{H^n}\,\mathrm dt.
\end{equation}
The Fourier pairing extends to \(H^{n-\ell}\times H^{n+\ell}\).
Fourier truncation converges in the two indicated time spaces,
so the pairings converge in \(L^1\), proving the formula.
At integer orders the same argument applies to the full-derivative pairing.
If these hypotheses hold on \([0,\infty)\), the squared norm and its
derivative are integrable.  Its limit is therefore zero, and
\begin{equation}\label{eq:2.3.10}
 \|f(t)\|_{H^n}^2
 \le
 2\left(\int_t^\infty
     \|\partial_\tau f(\tau)\|_{H^{n-\ell}}^2\,\mathrm d\tau\right)^{\frac12}
  \left(\int_t^\infty
     \|f(\tau)\|_{H^{n+\ell}}^2\,\mathrm d\tau\right)^{\frac12}.
\end{equation}

For the coupled equations, let \(u,f\) have zero vertical means and
satisfy, in the sense of distributions,
\begin{equation}\label{eq:2.3.11}
 u_t+\mathcal Lu+g\beta\mathcal Tf=F_1,
 \qquad f_t-\gamma\mathcal Tu=F_2.
\end{equation}
For \(q\ge0\), set
\(X_q=\operatorname{Re}\langle\mathcal L^{-1}\mathcal Tu,f\rangle_{H^q}\).
Then
\begin{equation}\label{eq:2.3.12}
\begin{aligned}
 X_q'+\frac1\gamma\|\mathcal M^{\frac12}f\|_{H^q}^2
 ={}&-\operatorname{Re}\langle\mathcal Tu,f\rangle_{H^q}
     +\gamma\|\mathcal L^{-\frac12}\mathcal Tu\|_{H^q}^2\\
 &+\operatorname{Re}\langle F_1,\mathcal L^{-1}\mathcal Tf\rangle_{H^q}
  +\operatorname{Re}\langle F_2,\mathcal L^{-1}\mathcal Tu\rangle_{H^q}.
\end{aligned}
\end{equation}
This follows by differentiation and self-adjointness.  It remains
valid, with \(X_q\in W^{1,1}(0,T)\), if
\(u,f\in C([0,T];H^q)\), \(u\in L^2(0,T;H^{q+1})\), and
\(F_i\in L^2(0,T;H^{q-1})\): the multipliers in
\eqref{eq:2.3.3} map the other factors into \(L^2H^{q+1}\), and
Fourier truncation passes every pairing in \(L^1\).
The corresponding full-derivative identity holds at integer orders.
The estimate \eqref{eq:2.3.3}, initially stated for nonnegative
orders, also holds for every real order by the same symbol comparison.
For smooth pairs one also has
\begin{equation}\label{eq:2.3.13}
 \frac12\frac{\mathrm d}{\mathrm dt}
   \left(\|u\|_{H^q}^2+\frac{g\beta}{\gamma}\|f\|_{H^q}^2\right)
 +\|\mathcal L^{\frac12}u\|_{H^q}^2
 =\operatorname{Re}\langle F_1,u\rangle_{H^q}
  +\frac{g\beta}{\gamma}\operatorname{Re}\langle F_2,f\rangle_{H^q}.
\end{equation}
Unlike the cross identity, extension of this formula requires both
source pairings to be defined.  In the nonlinear spectral applications
below, the unknowns have the additional regularity needed for
\eqref{eq:2.3.9}.  For the homogeneous linear system, the energy
identity follows by Fourier truncation, as proved in Section~\ref{sec:3}.

Two elementary time estimates will be used in the bootstrap arguments.
Let \(\omega\ge0\).
If \(H\ge0\) and
\(\int_0^\infty H(u)(1+u)^\omega\,\mathrm du<\infty\), then
\begin{equation}\label{eq:2.3.14}
 \sup_{0\le t\le T}(1+t)^\omega
       \int_0^tH(t-\tau)|F(\tau)|\,\mathrm d\tau
 \le \left(\int_0^\infty H(u)(1+u)^\omega\,\mathrm du\right)
       \sup_{0\le\tau\le T}(1+\tau)^\omega|F(\tau)|.
\end{equation}
Indeed, \(1+t\le(1+t-\tau)(1+\tau)\).  This applies to both
exponential kernels and \((1+u^{-\frac12})e^{-cu}\).

For the scalar comparison, let \(C_0,F_0\ge0\), and let \(y\ge0\) be absolutely continuous
on \([0,T]\), with \(y(0)\le C_0\varepsilon^2\).  Suppose that
\begin{equation}\label{eq:2.3.15}
 y'+c\varepsilon^{-\frac2\alpha}y^{1+\frac1\alpha}
 \le F_0\varepsilon^4(1+t)^{-\beta},
 \qquad \alpha,c,\varepsilon>0,\quad \beta\ge\alpha+1.
\end{equation}
Choose \(K\ge C_0\) so that \(cK^{1+\frac1\alpha}-\alpha K\ge1\).
If \(F_0\varepsilon^2\le1\), comparison with
\(K\varepsilon^2(1+t)^{-\alpha}\) gives
\begin{equation}\label{eq:2.3.16}
 y(t)\le K\varepsilon^2(1+t)^{-\alpha}.
\end{equation}
The choice of \(K\) is independent of \(F_0\).

%============================================================================================================================================================
% End of Section 2
%============================================================================================================================================================

%=====================================================================
% Section 3. Linear analysis
%=====================================================================
\section{Linear analysis}\label{sec:3}
We establish well-posedness for the linearized system and derive
uniform bounds for its spectral kernels.  These bounds yield Sobolev
convergence and algebraic decay.  The forced response estimates in
Subsection~\ref{subsec:3.2} are used in the nonlinear analysis of
Section~\ref{sec:4}.

\subsection{Linear well-posedness and spectral estimates}\label{subsec:3.1}
The homogeneous linearized system is
\begin{equation}\label{eq:3.1.1}
\left\{
\begin{aligned}
 &\partial_tv+\mathcal Lv+g\beta\mathcal T\theta=0,\\
 &\partial_t\theta-\gamma\mathcal Tv=0,\\
 &(v,\theta)|_{t=0}=(v_0,\theta_0).
\end{aligned}
\right.
\end{equation}
The linear diagnostic variables are \(w=-\mathcal Tv\) and
\(p=g\beta\partial_z^{-1}\theta\), with \(\langle p\rangle_z=0\).
For \((k,m)\in\mathcal K_*\), the Fourier matrix is
\[
 \mathcal A(k,m)=\begin{pmatrix}
 -a&-\dfrac{g\beta k}{m}\\[0.4em]
 \dfrac{\gamma k}{m}&0
 \end{pmatrix},
 \qquad \widehat U=\binom{\widehat v}{\widehat\theta}.
\]
The modal energy identity yields global well-posedness and the
following energy equality.

\begin{lemma}\label{lem:3.1}
Let \(s\ge0\), and let
\((v_0,\theta_0)\in H^s(\Omega)\times H^s(\Omega)\) satisfy
\(\langle v_0\rangle_z=\langle\theta_0\rangle_z=0\).
Then \eqref{eq:3.1.1} admits a unique global solution satisfying,
for every \(T>0\),
\[
 v\in C([0,T];H^s)\cap L^2(0,T;H^{s+1}),
 \qquad \theta\in C([0,T];H^s),
\]
\[
 \partial_tv\in L^2(0,T;H^{s-1}),
 \qquad \partial_t\theta\in L^2(0,T;H^s).
\]
The zero vertical means are preserved.  For
\(0\le t_1\le t_2\le T\),
\begin{equation}\label{eq:3.1.2}
 \|v(t_2)\|_{H^s}^2+\frac{g\beta}{\gamma}\|\theta(t_2)\|_{H^s}^2
 +2\int_{t_1}^{t_2}\mathcal D_s(\tau)\,\mathrm d\tau
 =\|v(t_1)\|_{H^s}^2+\frac{g\beta}{\gamma}\|\theta(t_1)\|_{H^s}^2.
\end{equation}
Here \(\mathcal D_s\) is defined by \eqref{eq:1.2.1} along
this linear solution and equals
\(\|\mathcal L^{\frac12}v\|_{H^s}^2\).
The diagnostic variables are uniquely determined under the stated
normalization and satisfy
\[
 w\in C([0,T];H^{s-1})\cap L^2(0,T;H^s),
 \qquad p\in C([0,T];H^s).
\]
The maps \((v_0,\theta_0)\mapsto(v,\theta)\) and
\((v_0,\theta_0)\mapsto v\) are continuous into
\(C([0,T];H^s\times H^s)\) and \(L^2(0,T;H^{s+1})\), respectively.
\end{lemma}
\begin{proof}
Let \(P_J\) be the Fourier projection onto \(|(k,m)|\le J\).
Define \(U_J\) by
\[
 \widehat U_J(k,m,t)=e^{\mathcal A(k,m)t}\widehat{P_JU_0}(k,m)\quad(m\ne0),
 \qquad \widehat U_J(k,0,t)=0.
\]
These finite sums preserve the zero vertical means and, for real
initial data, the symmetry
\(\widehat U_J(-k,-m,t)=\overline{\widehat U_J(k,m,t)}\).
The coupling terms cancel in each mode:
\[
 \frac{\mathrm d}{\mathrm dt}
 \left(|\widehat v_J|^2+\frac{g\beta}{\gamma}|\widehat\theta_J|^2\right)
 +2a|\widehat v_J|^2=0.
\]
After multiplication by \(L_xL_z\langle k,m\rangle^{2s}\) and
summation, \eqref{eq:2.3.4} applied to differences yields
\[
\begin{aligned}
 &\|v_J-v_K\|_{C_tH^s}^2+\|\theta_J-\theta_K\|_{C_tH^s}^2
   +\|v_J-v_K\|_{L_t^2H^{s+1}}^2\\
 \le{}&C\left(\|(P_J-P_K)v_0\|_{H^s}^2
                  +\|(P_J-P_K)\theta_0\|_{H^s}^2\right)
 \longrightarrow0.
\end{aligned}
\]
Thus, for each \(T>0\), there are \(v,\theta\) such that
\[
 \|v_J-v\|_{C_tH^s}+\|\theta_J-\theta\|_{C_tH^s}
 +\|v_J-v\|_{L_t^2H^{s+1}}\longrightarrow0.
\]
Applying the multiplier bounds to the difference equations, we obtain
\[
 \|\partial_t(v_J-v_K)\|_{L_t^2H^{s-1}}
 \le C\bigl(\|v_J-v_K\|_{L_t^2H^{s+1}}
           +T^{\frac12}\|\theta_J-\theta_K\|_{C_tH^s}\bigr)
\]
and
\[
 \|\partial_t(\theta_J-\theta_K)\|_{L_t^2H^s}
 \le C\|v_J-v_K\|_{L_t^2H^{s+1}}.
\]
Consequently the time derivatives converge in the stated spaces,
and the limit satisfies \eqref{eq:3.1.1} in distributions with
initial value \(U_0\).  The diagnostic formulas imply
\[
 \|w\|_{C_tH^{s-1}}+\|p\|_{C_tH^s}
 \le C\bigl(\|v\|_{C_tH^s}+\|\theta\|_{C_tH^s}\bigr),
 \qquad \|w\|_{L_t^2H^s}\le C\|v\|_{L_t^2H^{s+1}}.
\]
All multiplier bounds remain valid at \(s=0\), with \(H^{-1}\)
in the indicated spaces.  Strong convergence in \(C_tH^s\) and
\(L_t^2H^{s+1}\) passes the finite-sum energy identity to
\eqref{eq:3.1.2}.  Applying the same modal identity to the
difference of two solutions proves uniqueness and continuous
dependence.  Since \(T>0\) is arbitrary, uniqueness identifies the
solutions on overlapping intervals and yields the global solution.
\end{proof}

For the decay estimates, we now compute the matrix exponential.
The characteristic equation and its roots are
\[
 \lambda^2+a\lambda+N^2\frac{k^2}{m^2}=0,
 \qquad
 \lambda_\pm=-\frac a2\pm\frac12\sqrt{a^2-4N^2\frac{k^2}{m^2}}.
\]
For negative discriminant, choose the square root so that
\(\lambda_+=-\frac a2+i\omega\), with \(\omega>0\).  Define
\begin{equation}\label{eq:3.1.3}
 K_1(t):=\frac{e^{\lambda_+t}-e^{\lambda_-t}}{\lambda_+-\lambda_-},
 \quad K_2(t):=\frac{\lambda_+e^{\lambda_+t}-\lambda_-e^{\lambda_-t}}
                      {\lambda_+-\lambda_-},
 \quad K_3(t):=\frac{\lambda_+e^{\lambda_-t}-\lambda_-e^{\lambda_+t}}
                      {\lambda_+-\lambda_-}.
\end{equation}
For the forced system \eqref{eq:2.3.11}, with
\((u,f)=(v,\theta)\) and zero-vertical-mean sources, set
\(\widehat F=(\widehat F_1,\widehat F_2)^{\mathsf T}\).
Variation of constants gives
\begin{equation}\label{eq:3.1.4}
 \widehat U(t)=e^{\mathcal At}\widehat U_0
 +\int_0^t e^{\mathcal A(t-\tau)}\widehat F(\tau)\,\mathrm d\tau,
 \qquad
 e^{\mathcal At}=\begin{pmatrix}
 K_2&-\dfrac{g\beta k}{m}K_1\\[0.4em]
 \dfrac{\gamma k}{m}K_1&K_3
 \end{pmatrix}.
\end{equation}
The kernels satisfy
\begin{equation}\label{eq:3.1.5}
 K_2=\lambda_+K_1+e^{\lambda_-t},
 \qquad K_3=-\lambda_-K_1+e^{\lambda_-t}.
\end{equation}
At \(k=0,m\ne0\), they reduce to
\[
 K_1(t)=\frac{1-e^{-at}}a,
 \qquad K_2(t)=e^{-at},\qquad K_3(t)=1.
\]
Thus the horizontal temperature mean is stationary in the homogeneous
system.  With forcing, it instead satisfies
\(\partial_t\langle\theta\rangle_x=\langle F_2\rangle_x\).
At a repeated root, the continuous extensions are
\[
 K_1(t)=te^{-\frac{at}{2}},\qquad
 K_2(t)=\left(1-\frac{at}{2}\right)e^{-\frac{at}{2}},\qquad
 K_3(t)=\left(1+\frac{at}{2}\right)e^{-\frac{at}{2}}.
\]
In particular, \eqref{eq:3.1.4} remains valid without root separation.

The homogeneous solution operator is uniformly bounded on the
zero-vertical-mean subspace of \(H^s\times H^s\), by
\eqref{eq:3.1.2}.  For zero initial data and
\(F_1,F_2\in L^1(0,T;H^s)\) with zero vertical means, the Duhamel
solution therefore satisfies
\begin{equation}\label{eq:3.1.6}
 \sup_{0\le t\le T}\bigl(\|v_F(t)\|_{H^s}+\|\theta_F(t)\|_{H^s}\bigr)
 \le C\int_0^T\bigl(\|F_1(t)\|_{H^s}+\|F_2(t)\|_{H^s}\bigr)\,\mathrm dt.
\end{equation}
This follows from Minkowski's inequality and \eqref{eq:3.1.2}.
Approximation in \(L_t^1H^s\), together with strong continuity of the
homogeneous evolution, yields \((v_F,\theta_F)\in C_t(H^s\times H^s)\).
The constant is independent of \(T\), and no time derivative of the
sources is required.

To obtain bounds uniform near repeated roots, we separate the
uniformly damped modes from the region with distinct slow and fast
branches.  The spectral kernels satisfy the following estimates.
\begin{lemma}\label{lem:3.2}
Fix \(\delta\in(0,1)\), and set
\[
 D=1-\frac{4\rho}a,\qquad
 \Sigma_1=\{(k,m)\in\mathcal K_*:D\le\delta^2\},\qquad
 \Sigma_2=\{(k,m)\in\mathcal K_*:D>\delta^2\}.
\]
Write \(\Sigma_j^x=\Sigma_j\cap\mathcal K_*^x\).  The following
bounds hold uniformly for \(t\ge0\).
\begin{enumerate}
\item[\textup{(i)}] On \(\Sigma_1\),
\[
 |K_1(t)|\le te^{-\frac{1-\delta}{2}at},\qquad
 |K_2(t)|+|K_3(t)|\le C_\delta(1+at)e^{-\frac{1-\delta}{2}at}.
\]
\item[\textup{(ii)}] On \(\Sigma_2\),
\begin{equation}\label{eq:3.1.7}
 \rho\le-\lambda_+\le\frac{2\rho}{1+\delta},\qquad
 \frac a2\le-\lambda_-\le a,\qquad
 \lambda_+-\lambda_-\ge\delta a,
\end{equation}
\begin{equation}\label{eq:3.1.8}
\begin{aligned}
 |K_1(t)|\le{}&C_\delta a^{-1}(e^{-\rho t}+e^{-\frac a2t}),\\
 |K_2(t)|\le{}&C_\delta\left(\frac\rho a e^{-\rho t}+e^{-\frac a2t}\right),\\
 |K_3(t)|\le{}&C_\delta\left(e^{-\rho t}+\frac\rho a e^{-\frac a2t}\right).
\end{aligned}
\end{equation}
Multiplication of the \(K_1\) bound by \(\left|\frac{k}{m}\right|\)
or \(\frac{k^2}{m^2}\) gives the corresponding coupling bounds.
\end{enumerate}
\end{lemma}
\begin{proof}
For real roots in \(\Sigma_1\),
\[
 K_1(t)=\int_0^te^{\lambda_+(t-\tau)}e^{\lambda_-\tau}\,\mathrm d\tau,
 \qquad \lambda_-\le\lambda_+\le-\frac{1-\delta}{2}a.
\]
This proves the bound for \(K_1\), and \eqref{eq:3.1.5} gives the
other two.  For complex roots,
\[
 K_1=e^{-\frac{at}{2}}\frac{\sin(\omega t)}\omega,
 \qquad
 K_{2,3}=e^{-\frac{at}{2}}
       \left(\cos(\omega t)\mp\frac{a\sin(\omega t)}{2\omega}\right).
\]
The bound \(|\sin(\omega t)|\le\omega t\) proves part~\textup{(i)},
including the repeated-root limit.  On \(\Sigma_2\),
\(\delta<\sqrt D\le1\), and
\[
 \lambda_+=-\frac{2\rho}{1+\sqrt D},\qquad
 \lambda_-=-\frac a2(1+\sqrt D),\qquad
 \lambda_+-\lambda_-=a\sqrt D.
\]
These identities prove \eqref{eq:3.1.7}.  Substitution into
\eqref{eq:3.1.3} gives \eqref{eq:3.1.8}, with the factor
\(\frac\rho a\) retained in the slow part of \(K_2\).
\end{proof}

For the rest of the paper, fix this partition with \(\delta=\frac12\).
On \(\Sigma_1^x\),
\[
 \frac3{16}\le\frac\rho a
 =\frac{N^2k^2}{a^2m^2}
 \le\frac{N^2}{\mu_xam^2}\le\frac Ca.
\]
Hence \(\Sigma_1^x\) is finite on the fixed lattice.  On this set,
\begin{equation}\label{eq:3.1.9}
 a(k,m)\le C,\qquad |K_1(t)|+|K_2(t)|+|K_3(t)|\le Ce^{-ct}.
\end{equation}
The same bound holds for the coupling and diagnostic multipliers
on this finite set, including nonreal and repeated roots.  In contrast,
for fixed \(k\ne0\),
\[
 \lim_{|m|\to\infty}m^4\lambda_+(k,m)=-\frac{N^2k^2}{\mu_z}.
\]
Thus no uniform spectral gap is available on the horizontally
oscillatory subspace.  The slow factors in \eqref{eq:3.1.8}
will instead yield algebraic decay under the regularity assumptions
below.

\subsection{Linear stability and forced response}\label{subsec:3.2}
The preceding spectral estimates yield the following stability and
algebraic decay results for the homogeneous linearized system.
The limiting temperature is the stationary horizontal mean
\(\langle\theta_0\rangle_x\).

\begin{proposition}\label{prop:3.1}
Let \((v,\theta)\) solve \eqref{eq:3.1.1} with zero-vertical-mean
initial data, and set \(w=-\mathcal Tv\).  The following statements hold.
\begin{enumerate}
\item[\textup{(i)}] Let \(s\ge0\), \(v_0\in H^{s+1}\), and
\(\theta_0\in H^s\).  Then
\begin{equation}\label{eq:3.2.1}
  \|v(t)\|_{H^s}
  +\|w(t)\|_{H^s}
  +\|\theta(t)-\langle\theta_0\rangle_x\|_{H^s}
  \longrightarrow0,
  \qquad t\to\infty,
\end{equation}
\begin{equation}\label{eq:3.2.2}
 2\int_0^\infty\mathcal D_s(t)\,\mathrm dt
=
 \|v_0\|_{H^s}^2
 +\frac{g\beta}{\gamma}\|\theta_0\|_{H^s}^2
 -\frac{g\beta}{\gamma}\|\langle\theta_0\rangle_x\|_{H^s}^2.
\end{equation}

\item[\textup{(ii)}] Suppose \(v_0,\theta_0\in L^2\),
\(w_0=-\mathcal Tv_0\in L^2\), and
\(\Pi_{\ne0}^x\theta_0\in H_z^2(\Omega)\).  Then, for \(t\ge0\),
\begin{equation*}
 \|v(t)\|_{L^2}
 \le Ce^{-ct}\|\Pi_0^x v_0\|_{L^2}
     +C(1+t)^{-1}\|\Pi_{\ne0}^x v_0\|_{L^2}
     +C(1+t)^{-\frac{1}{2}}\|\Pi_{\ne0}^x\theta_0\|_{L^2},
\end{equation*}
\begin{equation*}
 \|w(t)\|_{L^2}
 \le C(1+t)^{-1}
 \left(\|\Pi_{\ne0}^x w_0\|_{L^2}
       +\|\Pi_{\ne0}^x\theta_0\|_{L^2}\right),
\end{equation*}
\begin{equation*}
 \|\Pi_{\ne0}^x\theta(t)\|_{L^2}
 \le C(1+t)^{-\frac{1}{2}}
 \left(\|\Pi_{\ne0}^x v_0\|_{L^2}
       +\|\Pi_{\ne0}^x\theta_0\|_{H_z^2(\Omega)}\right).
\end{equation*}

The constants in these decay bounds depend only on the fixed
parameters and periods.
\end{enumerate}
\end{proposition}

\begin{remark}
The additional regularity of \(v_0\) in part~\textup{(i)} ensures
\(w_0=-\mathcal Tv_0\in H^s\).  It is not required for the
\(H^s\times H^s\) well-posedness established in
Lemma~\ref{lem:3.1}.
\end{remark}

\begin{proof}
We first prove part~\textup{(i)}.  For \(m\ne0\),
\[
 \widehat v(0,m,t)=e^{-\mu_zm^2t}\widehat v_0(0,m),
 \quad \widehat w(0,m,t)=0,
 \quad \widehat\theta(0,m,t)=\widehat\theta_0(0,m).
\]
For \(k,m\ne0\), both eigenvalues have negative real part, so each
Fourier coefficient tends to zero.  The representation
\eqref{eq:3.1.4} also gives
\[
 \widehat w=K_2\widehat w_0
           +g\beta\frac{k^2}{m^2}K_1\widehat\theta_0,
 \qquad \widehat w_0=-\frac{k}{m}\widehat v_0.
\]
Lemma~\ref{lem:3.2}, \eqref{eq:3.1.9}, and \eqref{eq:2.3.2}
therefore imply
\[
 |\widehat v|+|\widehat{\theta-\langle\theta_0\rangle_x}|
 \le C\bigl(|\widehat v_0|+|\widehat\theta_0|\bigr),
 \qquad
 |\widehat w|
 \le C\bigl(|\widehat w_0|+|\widehat\theta_0|\bigr).
\]
Since \(w_0\in H^s\), the squares of these majorants are summable
with weight \(L_xL_z\langle k,m\rangle^{2s}\).
Dominated convergence proves \eqref{eq:3.2.1}.  Taking
\(t_1=0\) and \(t_2\to\infty\) in \eqref{eq:3.1.2} then yields
\eqref{eq:3.2.2}.

We next prove part~\textup{(ii)}.  The multipliers on \(\Sigma_1^x\) are exponentially damped by \eqref{eq:3.1.9}.
On \(\Sigma_2^x\), \eqref{eq:3.1.8} and \eqref{eq:2.3.2} yield
\[
 |K_2|+\left|\frac{k^2}{m^2}K_1\right|
 \le C\rho e^{-\rho t}+Ce^{-ct}
\]
and
\[
 \left|\frac{k}{m}K_1\right|
 \le C\rho^{\frac12}e^{-\rho t}+Ce^{-ct},
 \qquad
 |K_3|\le Ce^{-\rho t}+Ce^{-ct}.
\]
Using \(\rho^\alpha e^{-\rho t}\le C_\alpha(1+t)^{-\alpha}\) for \(0<\rho\le C_0\) and \(\alpha>0\), we obtain, on \(\mathcal K_*^x\),
\[
 |\widehat v(t)|
 \le C(1+t)^{-1}|\widehat v_0|
     +C(1+t)^{-\frac12}|\widehat\theta_0|
\]
and
\[
 |\widehat w(t)|
 \le C(1+t)^{-1}\bigl(|\widehat w_0|+|\widehat\theta_0|\bigr).
\]
Parseval's identity and the separate heat evolution of
\(\Pi_0^xv\) give the two velocity estimates.

For the temperature fluctuation, the velocity input is controlled
by \(\frac{k}{m}K_1\).  By \eqref{eq:2.3.1}, its slow temperature
input satisfies
\[
 e^{-\rho t}=\rho^{-\frac12}\rho^{\frac12}e^{-\rho t}
 \le C(1+m^2)(1+t)^{-\frac12}.
\]
Thus two vertical derivatives of \(\Pi_{\ne0}^x\theta_0\) suffice:
\[
\begin{aligned}
 \|\Pi_{\ne0}^x\theta(t)\|_2^2
 \le{}&C(1+t)^{-1}L_xL_z\sum_{\mathcal K_*^x}
       \bigl(|\widehat v_0|^2+(1+m^2)^2|\widehat\theta_0|^2\bigr)\\
 ={}&C(1+t)^{-1}\bigl(\|\Pi_{\ne0}^xv_0\|_2^2
                +\|\Pi_{\ne0}^x\theta_0\|_{H_z^2(\Omega)}^2\bigr).
\end{aligned}
\]
Taking square roots completes the proof.
\end{proof}

The same modal majorants also give uniform convergence when
\(\|v_0\|_{\mathcal W^0}+\|w_0\|_{\mathcal W^0}
+\|\theta_0\|_{\mathcal W^0}<\infty\), with the same compatibility
conditions:
\[
 \|v(t)\|_\infty+\|w(t)\|_\infty
 +\|\theta(t)-\langle\theta_0\rangle_x\|_\infty\longrightarrow0.
\]
This follows from the same coefficientwise limits and dominated
convergence in the absolute Fourier sums.

We now estimate the forced system.  The slow spectral contribution
can be expressed through the current temperature, leaving an
exponentially damped equation for the remaining component.  The
following bounds will be applied on finite intervals of the nonlinear
lifespan.
\begin{lemma}\label{lem:3.3}
Let \(s\ge4\) be an integer and \(T>0\).  Suppose that
\(F_i\in C([0,T];H^{s-1})\), \(i=1,2\), have zero vertical
means.  Let \((v,\theta)\in C([0,T];H^s\times H^s)\) be a
distributional solution of \eqref{eq:2.3.11}, with
\((u,f)=(v,\theta)\), zero vertical means, and initial data
\((v_0,\theta_0)\).
Set \(\Theta=\Pi_{\ne0}^x\theta\), \(w^0=-\mathcal Tv\), and
\(\varepsilon=\|v_0\|_{H^s}+\|\theta_0\|_{H^s}\).
For \(0\le r\le s-1\) and \(0\le t\le T\),
\begin{equation}\label{eq:3.2.3}
 \|w^0(t)\|_{H^r}
 \le C_r\varepsilon e^{-ct}+C_r\|\mathcal M\Theta(t)\|_{H^r}
 +C_r\int_0^tK_{\mathrm f}(t-\tau)
                    \sum_{i=1}^2\|F_i(\tau)\|_{H^r}\,\mathrm d\tau,
\end{equation}
\begin{equation}\label{eq:3.2.4}
 \|v(t)\|_{H^r}
 \le C_r\|\mathcal M^{\frac34}\Theta(t)\|_{H^r}+C_r\varepsilon e^{-ct}
 +C_r\int_0^t e^{-c(t-\tau)}
                    \sum_{i=1}^2\|F_i(\tau)\|_{H^r}\,\mathrm d\tau,
\end{equation}
where \(K_{\mathrm f}(u)=(1+u^{-\frac12})e^{-cu}\) for \(u>0\).
The full horizontal velocity also satisfies
\begin{equation}\label{eq:3.2.5}
 \|v(t)\|_{\mathcal W^0}
 \le C\varepsilon e^{-ct}
 +C\|\mathcal M\Theta(t)\|_2^{\frac34}\|\Theta(t)\|_{H^4}^{\frac14}
 +C\int_0^tK_{\mathrm f}(t-\tau)
                    \sum_{i=1}^2\|F_i(\tau)\|_2\,\mathrm d\tau.
\end{equation}
The constants are independent of \(T\), the solution, and the
source terms \(F_1,F_2\).
\end{lemma}

\begin{proof}
We use the partition fixed in Subsection~\ref{subsec:3.1}.
On \(\Sigma_2^x\), \eqref{eq:3.1.7} gives
\(\lambda_-\le-\frac a2\) and \(\rho\le|\lambda_+|\le C\rho\).
Define
\begin{equation}\label{eq:3.2.6}
 h_+(k,m):=\frac{m\lambda_+(k,m)}{\gamma k},
 \qquad
 \widehat v_{\mathrm f}:=\widehat v-h_+\widehat\Theta
 \quad\text{on }\Sigma_2^x.
\end{equation}
The characteristic polynomial yields
\[
 -a-\frac{\gamma k}{m}h_+=\lambda_-,
 \qquad \lambda_-h_+=\frac{g\beta k}{m}.
\]
The assumed regularity makes each Fourier coefficient absolutely
continuous.  Since \(\widehat\Theta=\widehat\theta\) for \(k\ne0\),
the forced equations yield
\begin{equation}\label{eq:3.2.7}
 \partial_t\widehat v_{\mathrm f}
 =\lambda_-\widehat v_{\mathrm f}+\widehat F_1-h_+\widehat F_2.
\end{equation}
Hence
\[
 \widehat v_{\mathrm f}(t)=e^{\lambda_-t}\widehat v_{\mathrm f}(0)
 +\int_0^te^{\lambda_-(t-\tau)}
     \bigl(\widehat F_1(\tau)-h_+\widehat F_2(\tau)\bigr)\,\mathrm d\tau,
\]
and the original variables satisfy
\begin{equation}\label{eq:3.2.8}
 \widehat v=h_+\widehat\Theta+\widehat v_{\mathrm f},
 \qquad
 \widehat{w^0}=-\frac{\lambda_+}{\gamma}\widehat\Theta
                           -\frac{k}{m}\widehat v_{\mathrm f}.
\end{equation}
The relevant multipliers satisfy
\begin{equation}\label{eq:3.2.9}
 |h_+|\le C\frac{|k|}{a|m|}\le C\rho^{\frac34},
 \qquad \left|\frac{k}{m}h_+\right|\le C\rho,
 \qquad \left|\frac{k}{m}\right|\le C\sqrt a.
\end{equation}
For the comparison with \(\rho^{\frac34}\),
\[
 \frac{|k|}{a|m|}\rho^{-\frac34}
 =\frac{|m|^{\frac12}}{N^{\frac32}a^{\frac14}|k|^{\frac12}}
 \le C,
\]
using \(a\ge\mu_zm^2\) and the lower bound on nonzero \(|k|\).
The other bounds follow from the definition of \(\rho\) and the
lower bound on nonzero \(|m|\).  Combining \eqref{eq:3.2.9} with
\(\lambda_-\le-\frac a2\), we obtain, for \(u>0\),
\begin{equation}\label{eq:3.2.10}
 \left|\frac{k}{m}e^{\lambda_-u}\right|
 \le Cu^{-\frac12}e^{-cu},
 \qquad
 \left|\frac{k}{m}h_+e^{\lambda_-u}\right|\le Ce^{-cu}.
\end{equation}
At time zero,
\[
 \left(L_xL_z\sum_{\Sigma_2^x}\langle k,m\rangle^{2r}
       \left|\frac{k}{m}\widehat v_{\mathrm f}(k,m,0)\right|^2\right)^{\frac12}
 \le C_r\bigl(\|v_0\|_{H^{r+1}}+\|\Theta_0\|_{H^r}\bigr)
 \le C_r\varepsilon.
\]
The initial fast component of \(v\) is likewise bounded by
\(C_r\varepsilon\), since \(h_+\) is bounded.  Neither initial
term has a short-time singularity.  Inserting
\eqref{eq:3.2.9}--\eqref{eq:3.2.10} in the preceding integral
formula and using Minkowski's inequality proves the
\(\Sigma_2^x\) contributions to \eqref{eq:3.2.3}--\eqref{eq:3.2.4}.
The finite region \(\Sigma_1^x\) contributes exponentially damped
initial and source terms by \eqref{eq:3.1.9}.

The horizontal zero mode of \(w^0\) vanishes.  The horizontal mean
of \(v\) satisfies
\begin{equation}\label{eq:3.2.11}
\left\{
\begin{aligned}
 &\partial_t\Pi_0^xv-\mu_z\partial_{zz}\Pi_0^xv=\Pi_0^xF_1,\\
 &\Pi_0^xv|_{t=0}=\Pi_0^xv_0.
\end{aligned}
\right.
\end{equation}
Since the datum and source have zero vertical mean,
\[
 \|\Pi_0^xv(t)\|_{H^r}
 \le Ce^{-ct}\|\Pi_0^xv_0\|_{H^r}
 +C\int_0^te^{-c(t-\tau)}\|\Pi_0^xF_1(\tau)\|_{H^r}\,\mathrm d\tau.
\]
This completes \eqref{eq:3.2.3}--\eqref{eq:3.2.4}.
The possibly time-dependent horizontal mean of \(\theta\) is
annihilated by \(\mathcal T\) and every positive power of
\(\mathcal M\), and therefore does not enter the estimates.

For the Wiener bound, let \(\mathcal H_+\) have symbol \(h_+\)
on \(\Sigma_2^x\) and zero elsewhere.  For \(\mathcal Qf=f\),
Cauchy--Schwarz and Fourier interpolation give
\begin{equation}\label{eq:3.2.12}
 \|\mathcal H_+f\|_{\mathcal W^0}
 \le C\|\mathcal M^{\frac34}f\|_{H^1}
 \le C\|\mathcal Mf\|_2^{\frac34}\|f\|_{H^4}^{\frac14}.
\end{equation}
For the first inequality, \eqref{eq:3.2.9} and ellipticity give
\[
 \sum_{\Sigma_2^x}
 \frac{|h_+(k,m)|^2}{\rho(k,m)^{\frac32}\langle k,m\rangle^2}
 \le C\sum_{k,m\ne0}\frac{|m|}{|k|\langle k,m\rangle^3}
 \le C\sum_{k\ne0}\frac1{k^2}<\infty.
\]
The last bound follows by summing in \(m\) separately over
\(|m|\ge|k|\) and \(|m|<|k|\).  For the second inequality,
H\"older's inequality yields
\[
 \|\mathcal M^{\frac34}f\|_{H^1}^2
 =L_xL_z\sum_{\mathcal K_*^x}
      \bigl(\rho^2|\widehat f|^2\bigr)^{\frac34}
      \bigl(\langle k,m\rangle^8|\widehat f|^2\bigr)^{\frac14}
 \le\|\mathcal Mf\|_2^{\frac32}\|f\|_{H^4}^{\frac12}.
\]

For \(u>0\), the heat sum satisfies
\[
 \left(\sum_{\mathcal K_*^x}e^{-a(k,m)u}\right)^{\frac12}
 \le C(1+u^{-\frac12})e^{-cu}=CK_{\mathrm f}(u).
\]
Since \(\lambda_-\le-\frac a2\) and \(h_+\) is bounded,
Cauchy--Schwarz controls the fast source terms from \(L^2\) to
\(\mathcal W^0\) by this kernel.  Equation~\eqref{eq:3.1.9}
handles the finite region.  The horizontal zero mode is controlled
by the one-dimensional heat sum, whose singularity
\(u^{-\frac14}\) is also bounded by \(C(1+u^{-\frac12})\).
The initial contribution has the direct bound
\[
 \sum_{\Sigma_2^x}|\widehat v_{\mathrm f}(k,m,0)|
 +\|\Pi_0^xv_0\|_{\mathcal W^0}
 \le C\bigl(\|v_0\|_{\mathcal W^0}+\|\Theta_0\|_{\mathcal W^0}\bigr)
 \le C\varepsilon
\]
by \eqref{eq:2.1.7}.  The remaining finite set satisfies the same bound.
Together with \eqref{eq:3.2.12}, these estimates prove
\eqref{eq:3.2.5}.

Finally, let \(P_J\) retain the Fourier modes \(|(k,m)|\le J\),
and apply the estimates to the truncated solution and sources.  For \(0\le r\le s-1\),
\[
 \sup_{t\le T}\int_0^tK_{\mathrm f}(t-\tau)
       \|(I-P_J)F_i(\tau)\|_{H^r}\,\mathrm d\tau
 \le\|K_{\mathrm f}\|_{L^1(0,\infty)}
       \|(I-P_J)F_i\|_{C_tH^{s-1}}\longrightarrow0.
\]
The same argument applies to the exponential kernel.  Convergence
of the solution in \(C_tH^s\) gives convergence of \(w^0\) in
\(C_tH^{s-1}\) and of \(v\) in \(C_t\mathcal W^0\).
Passing to the limit proves the estimates for all \(t\in[0,T]\).
\end{proof}

%=====================================================================
% Section 4. Global stability and asymptotic stratification
%=====================================================================
\section{Global stability and asymptotic stratification}\label{sec:4}
We establish the nonlinear energy estimates in
Subsection~\ref{subsec:4.1} and use them in Subsection~\ref{subsec:4.2}
to prove Theorem~\ref{thm:1.1}.  We then show that the diagnostic
time-integral estimate used in the proof cannot hold uniformly
with only an \(H^3\) bound on the initial data.

\subsection{Nonlinear energy estimates}\label{subsec:4.1}
Let \(s\ge4\) be an integer, and let \([0,T]\) be a compact subinterval
of the lifespan of a compatible \(H^s\)-strong solution.  With
\(\mathcal D_s\) as in \eqref{eq:1.2.1}, set
\[
 E_1(t):=\|v(t)\|_{H^s,\mathrm d}^2+\|\theta(t)\|_{H^s,\mathrm d}^2,
 \qquad
 X(t):=\langle\mathcal L^{-1}\mathcal Tv(t),\Theta(t)\rangle_{H^s,\mathrm d}.
\]
\begin{equation}\label{eq:4.1}
 E_*(T):=\sup_{0\le t\le T}E_1(t)+\int_0^T\mathcal D_s(t)\,\mathrm dt.
\end{equation}
\begin{equation*}
 \mathcal C_v(t):=\|\nabla v(t)\|_{L^\infty}
   +\|w^0(t)\|_{W^{2,\infty}}
   +\|\langle w\rangle_z(t)\|_{W_x^{2,\infty}}.
\end{equation*}

Then \(\|w\|_{W^{2,\infty}}\le C\mathcal C_v\).  The identity
\(\mathcal Q\theta=\Theta\) allows us to apply
\eqref{eq:2.3.6}--\eqref{eq:2.3.8} to the temperature fluctuation.
At \(s=4\), we use
\begin{equation}\label{eq:4.2}
 e(t):=E_1(t)^{\frac12},
 \qquad V(t):=\|v(t)\|_{H^5},
 \qquad A_4(t):=\|\mathcal M^{\frac12}\Theta(t)\|_{H^4,\mathrm d},
\end{equation}
so \(V^2\) is equivalent to \(\mathcal D_4\).

The velocity--temperature cross term controls the temperature
fluctuation without thermal diffusion.  A related energy argument for
the Boussinesq system appears in \cite[Proposition~4.3]{JangKim2023}.
On each interval \([0,T]\) with \(\sup_{t\le T}E_1(t)\le1\),
the cross term \(X\) belongs to \(W^{1,1}(0,T)\) and satisfies,
for almost every \(t\in(0,T)\),
\begin{equation}\label{eq:4.3}
 X'(t)+\frac{1}{2\gamma}
 \|\mathcal M^{\frac12}\Theta(t)\|_{H^s,\mathrm d}^2
 \le C\mathcal D_s(t).
\end{equation}
Consequently,
\begin{equation}\label{eq:4.4}
 \int_0^T\|\mathcal M^{\frac12}\Theta(t)\|_{H^s,\mathrm d}^2\,\mathrm dt
 \le C E_*(T).
\end{equation}
The constants in \eqref{eq:4.3}--\eqref{eq:4.4} are independent
of \(T\) and of the solution under the stated bound.  To prove
\eqref{eq:4.3}, we use the diagnostic formulas and the algebra
property of \(H^{s-1}\) to obtain
\begin{equation}\label{eq:4.5}
 \|\mathcal N_1(t)\|_{H^{s-1}}+\|\mathcal N_2(t)\|_{H^{s-1}}
 \le CE_1(t)^{\frac12}\mathcal D_s(t)^{\frac12},
 \qquad \|w(t)\|_{H^{s-1}}\le C\mathcal D_s(t)^{\frac12}.
\end{equation}
Thus \(\mathcal N_i\in L_t^2H^{s-1}\).  The projected equations
satisfy \eqref{eq:2.3.11} with \((u,f)=(v,\Theta)\),
\(F_1=\Pi_{\ne0}^z\mathcal N_1\), and \(F_2=\mathcal Q\mathcal N_2\).
The full-derivative version of \eqref{eq:2.3.12} applies at order
\(s\).  In particular, \(X\in W^{1,1}\).  By \eqref{eq:2.1.2},
\eqref{eq:2.3.5}, and \eqref{eq:2.3.8}, we have
\[
 |\langle\mathcal Tv,\Theta\rangle_{H^s,\mathrm d}|
 +|\langle\mathcal Q\mathcal N_1,\mathcal L^{-1}\mathcal T\Theta\rangle_{H^s,\mathrm d}|
 \le C(1+E_1^{\frac12})\mathcal D_s^{\frac12}
          \|\mathcal M^{\frac12}\Theta\|_{H^s,\mathrm d},
\]
and
\[
 \gamma\|\mathcal L^{-\frac12}\mathcal Tv\|_{H^s,\mathrm d}^2
 +|\langle\mathcal Q\mathcal N_2,\mathcal L^{-1}\mathcal Tv\rangle_{H^s,\mathrm d}|
 \le C(1+E_1^{\frac12})\mathcal D_s.
\]
Young's inequality proves \eqref{eq:4.3}.  The estimate
\begin{equation}\label{eq:4.6}
 |X(t)|\le C\|v(t)\|_{H^s,\mathrm d}\|\Theta(t)\|_{H^s,\mathrm d}
 \le CE_1(t)
\end{equation}
then gives \eqref{eq:4.4} by time integration.

The highest-order temperature terms require the following weighted
embeddings.  The fourth-order case uses mixed
norms in place of second-order uniform bounds.
For \(\mathcal Qf=f\) and integers \(s\ge5\),
\[
 \|f\|_{\mathcal W^2}\le C\|\mathcal M^{\frac12}f\|_{H^5},
 \qquad
 \|f\|_{W^{2,\infty}}\le C_s\|\mathcal M^{\frac12}f\|_{H^s,\mathrm d},
\]
whenever the respective right-hand sides are finite.
At order four, one has
\begin{equation}\label{eq:4.7}
 \|f\|_{\mathcal W^1}+\|f\|_{H^2}
 +\|\partial_xf\|_{H^2}+\|\mathcal Tf\|_{H^3}
 \le C\|\mathcal M^{\frac12}f\|_{H^4},
\end{equation}
and, for \(|\alpha|\le2\),
\begin{equation}\label{eq:4.8}
 \|\partial^\alpha f\|_{L_x^\infty L_z^2}
 \le C\|\mathcal M^{\frac12}f\|_{H^4}.
\end{equation}

The Wiener estimates follow from the convergent coefficient sum
\[
 \sum_{\mathcal K_*^x}\frac{\langle k,m\rangle^{-6}}{\rho}
 \le C\sum_{k\ne0}\frac1{k^2}
         \sum_m\frac{m^2}{(1+k^2+m^2)^2}
 \le C\sum_{k\ne0}\frac1{k^2(1+k^2)^{\frac12}}<\infty.
\]
Cauchy--Schwarz gives \(H^5\)-control of \(\mathcal W^2\) and
\(H^4\)-control of \(\mathcal W^1\).  Equations
\eqref{eq:2.1.7}, \eqref{eq:2.1.1}, and \eqref{eq:2.3.6}
give the remaining Sobolev and uniform bounds.
For \eqref{eq:4.8}, apply the Hilbert-valued
one-dimensional embedding to \(\partial^\alpha f\).  Its additional
horizontal derivative is controlled by \(\|\partial_xf\|_{H^2}\)
in \eqref{eq:4.7}.

To justify the energy identities below, we approximate strong solutions
by smooth solutions of the original system, with convergence also in
the velocity dissipation norm.  This approximation follows from the local theory and the
one-sided comparison estimates in \cite{PaperI}.

\begin{lemma}\label{lem:4.1}
Let \(s\ge4\) be an integer, and let \((v,\theta)\) be a compatible,
strictly stratified \(H^s\)-strong solution on a compact subinterval
\([0,T]\) of its maximal lifespan.  There is a sequence of solutions \((v_n,\theta_n)\) of the original
system, with smooth compatible initial data and arbitrarily high
spatial Sobolev regularity on \([0,T]\), such that
\begin{equation}\label{eq:4.9}
\begin{aligned}
 v_n\longrightarrow{}&v
 &&\text{in }C([0,T];H^s)\cap L^2(0,T;H^{s+1}),\\
 \theta_n\longrightarrow{}&\theta
 &&\text{in }C([0,T];H^s).
\end{aligned}
\end{equation}
The stratification bounds are uniform on \([0,T]\times\Omega\):
\begin{equation}\label{eq:4.10}
 R_n:=\gamma+\partial_z\theta_n\ge\kappa_T,
 \qquad R\ge\kappa_T,
\end{equation}
for some \(\kappa_T>0\) independent of \(n\).  Moreover,
\begin{equation}\label{eq:4.11}
\begin{aligned}
 w_n^0\longrightarrow{}&w^0
 &&\text{in }C([0,T];H^{s-1})\cap L^2(0,T;H^s),\\
 \langle w_n\rangle_z\longrightarrow{}&\langle w\rangle_z
 &&\text{in }C([0,T];H_x^{s-1}),\\
 p_n\longrightarrow{}&p
 &&\text{in }C([0,T];H^s).
\end{aligned}
\end{equation}
\end{lemma}

\begin{proof}
Take \(U_{0,n}=J_nU_0\), where \(J_n\) is the Fourier smoothing
operator used in \cite[Subsection~5.1]{PaperI}, and let
\(U_n=(v_n,\theta_n)\) be the corresponding solutions of the original
system.  The Hadamard theorem and higher-order persistence in
\cite[Subsections~2.2 and~5.1]{PaperI} give, for all sufficiently
large \(n\), solutions on \([0,T]\) with arbitrarily high spatial
Sobolev regularity and \(U_n\to U\) in \(C_t(H^s\times H^s)\).
The smoothing preserves real-valuedness and the mean constraints.
Sobolev convergence preserves strict stratification for all large
\(n\), and \eqref{eq:4.10} follows after discarding finitely many terms.

The one-sided comparison and smoothing estimates in
\cite[Subsections~5.2--5.3]{PaperI}, with the velocity dissipation
term retained, imply
\[
\begin{aligned}
 &\|U_n-U\|_{C_t(H^s\times H^s)}
   +\|v_n-v\|_{L_t^2H^{s+1}}\\
 \le{}&C_T\left(
   \|(I-J_n)U_0\|_{H^s\times H^s}
   +n\|(I-J_n)U_0\|_{H^{s-1}\times H^{s-1}}
   \right)\longrightarrow0.
\end{aligned}
\]
Here \(C_T\) may depend on \(T\), the reference solution bounds,
and the positive stratification lower bound, but not on \(n\).
The Fourier tail estimates for \(J_n\) make the right-hand side
tend to zero.  This proves
\eqref{eq:4.9}.  The diagnostic bounds
\eqref{eq:2.2.7}--\eqref{eq:2.2.8}, together with
\(w_n^0-w^0=-\mathcal T(v_n-v)\), give \eqref{eq:4.11}.
\end{proof}

We now estimate the high-order energy.  For \(R>0\), define
\begin{equation*}
 E_2=\|v\|_{H^s,\mathrm d}^2+\frac{g\beta}{\gamma}\|\theta\|_2^2
       +g\beta\sum_{1\le|\alpha|\le s}
                       \int_\Omega R^{-1}|\partial^\alpha\theta|^2.
\end{equation*}
The horizontal compensation for the diagnostic mean is established
in \cite[Section~3]{PaperI}.  Combined with the indirect temperature
dissipation, it yields the following energy estimate with coefficient
\(\mathcal C_v\).  Its time integrability will be proved in
Subsection~\ref{subsec:4.2}.

\begin{proposition}\label{prop:4.1}
Let \(s\ge4\) be an integer.  There are
\(\lambda_E>0\), \(\delta_2\in(0,1]\), and \(c,C>0\) such that
every compatible \(H^s\)-strong solution on \([0,T]\) with
\begin{equation}\label{eq:4.12}
 \sup_{t\le T}E_1(t)\le\frac{\delta_2^2}{4}
\end{equation}
satisfies \(R\ge\frac\gamma2\).  The modified energy
\(E_3(t):=E_2(t)+\lambda_E X(t)\) belongs to \(W^{1,1}(0,T)\) and obeys
\begin{equation}\label{eq:4.13}
  cE_1(t)\le E_3(t)\le CE_1(t),
\end{equation}
and, for almost every \(t\in(0,T)\),
\begin{equation}\label{eq:4.14}
  \frac{\mathrm d}{\mathrm dt}E_3(t)
  +c\bigl(\mathcal D_s(t)+\|\mathcal M^{\frac12}\Theta(t)\|_{H^s,\mathrm d}^2\bigr)
  \le C\mathcal C_v(t)E_3(t).
\end{equation}
The constants are independent of \(T\) and of higher solution norms.
\end{proposition}

\begin{proof}
We first work with smooth solutions under the bounds
\(E_1^{\frac12}\le1\) and
\(\|\partial_z\theta\|_{L^\infty}\le\frac\gamma2\).
The choice of \(\delta_2\) below ensures both conditions for the
solution and its sufficiently close smooth approximations.
Write \(S=\langle\theta\rangle_x\) and \(q=R^{-1}\).
Sobolev embedding, averaging, and the composition estimate give
\begin{equation}\label{eq:4.15}
 \frac\gamma2\le R\le\frac{3\gamma}2,\qquad
 \frac\gamma2\le\gamma+S'\le\frac{3\gamma}2,\qquad
 \|q\|_{W^{1,\infty}}\le C.
\end{equation}
The diagnostic bounds \eqref{eq:2.2.7},
\eqref{eq:2.1.4}, and norm equivalence imply
\begin{equation}\label{eq:4.16}
\begin{aligned}
 &\|v\|_{W^{2,\infty}}+\|\theta\|_{W^{2,\infty}}+\|S\|_{H_z^s} \le CE_1^{\frac12},\\
 &\|v\|_{H^{s+1}}+\|w^0\|_{H^s}+\|w\|_{H^{s-1}} \le C\mathcal D_s^{\frac12},\\
 &\|\langle w\rangle_z\|_{H_x^{s-1}} \le CE_1^{\frac12}\mathcal D_s^{\frac12}.
\end{aligned}
\end{equation}
When \(s\ge5\), we also use
\begin{equation}\label{eq:4.17}
 \|q\|_{W^{2,\infty}}\le C,\qquad
 \|\theta\|_{W^{3,\infty}}\le CE_1^{\frac12},\qquad
 \|\Theta\|_{W^{2,\infty}}\le C\|\mathcal M^{\frac12}\Theta\|_{H^s,\mathrm d},
\end{equation}
where the last bound is the weighted embedding established above.
At \(s=4\), we use \eqref{eq:4.7}--\eqref{eq:4.8} instead.  With
\eqref{eq:4.2}, the coefficient bounds are
\begin{equation*}
 \|q_{zz}\|_{L_x^2L_z^\infty}\le Ce,\qquad
 \|S''\|_{L_z^\infty}+\|S'''\|_{L_z^\infty}+\|S''''\|_{L_z^2}\le Ce.
\end{equation*}
Indeed, \(q_{zz}=2R^{-3}\theta_{zz}^2-R^{-2}\theta_{zzz}\), and
\[
\|\theta_{zzz}\|_{L_x^2L_z^\infty}\le C(\|\theta_{zzz}\|_2+\|\theta_{zzzz}\|_2).
\]
No second-order uniform bound on \(q\) is used in this case.

All constants are uniform under \eqref{eq:4.15} and the preliminary
smallness bound.  In particular, \(c_EE_1\le E_2\le C_EE_1\) for
fixed \(c_E,C_E>0\).
With \(D_t:=\partial_t+v\partial_x+w\partial_z\), the temperature
equation gives
\[
D_tR=R\partial_xv-(\partial_zv)\partial_x\Theta.
\]
Hence
\begin{equation}\label{eq:4.18}
  D_t(R^{-1})=-R^{-1}\partial_xv+R^{-2}(\partial_zv)\partial_x\Theta,
  \qquad
  \|D_t(R^{-1})\|_{L^\infty}\le C\mathcal C_v.
\end{equation}

We derive the weighted identity before estimating its remainders.
For \(1\le|\alpha|\le s\), the equations take the form
\begin{equation*}
\left\{
\begin{aligned}
  &D_t(\partial^\alpha v)+\mathcal L \partial^\alpha v+\partial_x\partial^\alpha p
  =-[\partial^\alpha,v]\partial_xv
       -[\partial^\alpha,w]\partial_zv,\\
  &D_t(\partial^\alpha\theta)+R\partial^\alpha w
  =-[\partial^\alpha,v]\partial_x\Theta
       -\sum_{0<\eta<\alpha}\binom\alpha\eta
        (\partial^\eta w)\partial_z\partial^{\alpha-\eta}\theta.
\end{aligned}
\right.
\end{equation*}
The term \(R\partial^\alpha w\) combines the highest-order vertical
transport term with the background coupling.  Pair the equations with
\(\partial^\alpha v\) and \(g\beta R^{-1}\partial^\alpha\theta\),
respectively.  Periodicity, incompressibility, and hydrostatic balance
imply
\begin{equation*}
  \int_\Omega \partial^\alpha v\partial_x\partial^\alpha p=-g\beta\int_\Omega\partial^\alpha\theta \partial^\alpha w.
\end{equation*}
This cancels the principal temperature coupling, including the
compatible vertical mean.  Summing the weighted identities over
\(1\le|\alpha|\le s\) and adding the zero-order energy identity
\eqref{eq:2.2.9}, we obtain
\begin{equation}\label{eq:4.19}
\begin{aligned}
  \frac12\frac{\mathrm d}{\mathrm dt}E_2
       +\|\mathcal L^{\frac12}v\|_{H^s,\mathrm d}^2
  ={}&-\sum_{1\le|\alpha|\le s}\int_\Omega
       \partial^\alpha v\bigl([\partial^\alpha,v]\partial_xv
                     +[\partial^\alpha,w]\partial_zv\bigr)\\
     &-g\beta\sum_{1\le|\alpha|\le s}\int_\Omega
       R^{-1}\partial^\alpha\theta[\partial^\alpha,v]\partial_x\Theta\\
     &-g\beta\sum_{1\le|\alpha|\le s}
       \sum_{0<\eta<\alpha}\binom\alpha\eta
       \int_\Omega R^{-1}\partial^\alpha\theta(\partial^\eta w)
                        \partial_z\partial^{\alpha-\eta}\theta\\
     &+\frac{g\beta}{2}\sum_{1\le|\alpha|\le s}
       \int_\Omega D_t(R^{-1})|\partial^\alpha\theta|^2.
\end{aligned}
\end{equation}
The full-derivative dissipation is comparable to \(\mathcal D_s\)
in \eqref{eq:1.2.1}, by \eqref{eq:2.1.1} and
\eqref{eq:2.3.4}.

For the horizontal temperature commutator, \eqref{eq:2.1.10},
\eqref{eq:4.16}, and \eqref{eq:2.3.6} imply
\begin{equation}\label{eq:4.20}
\begin{aligned}
 \sum_{1\le|\alpha|\le s}
      \left|\int_\Omega R^{-1}\partial^\alpha\theta
                   [\partial^\alpha,v]\partial_x\Theta\right|
  \le{}&CE_1^{\frac12}\left(
       \|\nabla v\|_{L^\infty}\|\Theta\|_{H^s}
       +\|v\|_{H^s}\|\partial_x\Theta\|_{L^\infty}\right)\\
  \le{}&C\mathcal C_vE_1
       +CE_1^{\frac12}\mathcal D_s^{\frac12}
         \|\mathcal M^{\frac12}\Theta\|_{H^s,\mathrm d}.
\end{aligned}
\end{equation}
We next split the last temperature factor in the vertical sum as
\(\theta=\Theta+S\).  For the fluctuation part and \(s\ge5\),
\eqref{eq:2.1.11} gives
\[
\begin{aligned}
 \sum_{1\le|\alpha|\le s}\sum_{0<\eta<\alpha}\binom\alpha\eta
 \left|\int_\Omega q\partial^\alpha\theta
              (\partial^\eta w)\partial_z\partial^{\alpha-\eta}\Theta\right|
 \le{}&CE_1^{\frac12}\left(\|\nabla w\|_\infty\|\Theta\|_{H^s}
                   +\|w\|_{H^{s-1}}\|\Theta\|_{W^{2,\infty}}\right)\\
 \le{}&C\mathcal C_vE_1+CE_1^{\frac12}\mathcal D_s^{\frac12}
                    \|\mathcal M^{\frac12}\Theta\|_{H^s,\mathrm d}.
\end{aligned}
\]
At \(s=4\), terms with \(|\eta|\le2\) instead use
\(\|\partial^\eta w\|_\infty\le C\mathcal C_v\) and two
\(L^2\) temperature factors.  If \(|\eta|=3\), necessarily
\(|\alpha|=4\).  For \(\eta_2\ge1\), incompressibility and
\(H^5\hookrightarrow W^{3,\infty}\) yield
\[
 \|\partial^\eta w\|_\infty\le CV,
 \qquad \|\partial_z\partial^{\alpha-\eta}\Theta\|_2\le CA_4.
\]
For \(\eta=(3,0)\), the last factor is \(\Theta_{xz}\) or
\(\Theta_{zz}\).  We estimate the velocity factors using the
diagnostic formula:
\[
 \|\partial_x^3w^0\|_{L_x^2L_z^\infty}
 \le C\|\partial_x^4v\|_2\le CV,
 \qquad
 \|\partial_x^3\langle w\rangle_z\|_{L_x^2}\le CeV.
\]
Estimate \eqref{eq:4.8} places the last factor in
\(L_x^\infty L_z^2\), with norm at most \(CA_4\).
Thus the fluctuation contribution is bounded by
\(C\mathcal C_ve^2+CeVA_4\) also at order four.

For the mean-temperature contribution, write \(\alpha=(a,b)\) and
\(n=a+b\le s\).  Only \(\alpha-\eta=(0,j)\) occurs, so the
terms, with coefficients \(\binom bj\), are
\[
 I_{a,b,j}:=\int_\Omega q(\partial_x^a\partial_z^b\theta)
                     (\partial_x^a\partial_z^{b-j}w)S^{(j+1)},
 \qquad 1\le j\le b,\quad n-j\ge1.
\]
If \(n-j\le2\), then \(j+1\le n\le s\) and
\[
 |I_{a,b,j}|\le C\|\partial_x^a\partial_z^{b-j}w\|_\infty
              \|\partial_x^a\partial_z^b\theta\|_2
              \|S^{(j+1)}\|_{L^2(\Omega)}
 \le C\mathcal C_vE_1.
\]
Suppose \(n-j\ge3\).  When \(b=j\), one has \(a\ge3\)
and \(b+3\le s\).  Moving one vertical derivative from the
fluctuation gives
\[
 I_{a,b,b}=-\int_\Omega\partial_x^a\partial_z^{b-1}\Theta\,
                    \partial_z\bigl(qS^{(b+1)}\partial_x^aw\bigr).
\]
The bounds
\[
 \|qS^{(b+1)}\|_{W^{1,\infty}}\le CE_1^{\frac12},
 \qquad
 \|\partial_x^aw\|_2+\|\partial_x^a\partial_zw\|_2
 \le C\mathcal D_s^{\frac12}
\]
and \eqref{eq:2.3.7} control this term by
\(CE_1^{\frac12}\mathcal D_s^{\frac12}
\|\mathcal M^{\frac12}\Theta\|_{H^s,\mathrm d}\).
At \(s=4\), this includes the configuration \((a,b,j)=(3,1,1)\).

The remaining terms have \(b-j\ge1\).  At \(s=4\), they satisfy
\(n=4\), \(j=1\), and \((a,b)=(2,2),(1,3),(0,4)\).
Set \(Q_2=qS''\).  Then
\(\|Q_2\|_\infty+\|(Q_2)_z\|_\infty\le Ce\).
For \((a,b)=(2,2)\) or \((1,3)\), incompressibility gives
\[
 I_{a,b,1}=\int_\Omega\partial_x^a\partial_z^{b-1}\Theta
       \bigl((Q_2)_z\partial_x^{a+1}\partial_z^{b-2}v
                    +Q_2\partial_x^{a+1}\partial_z^{b-1}v\bigr).
\]
The temperature factor has order three and contains a horizontal
derivative, so \eqref{eq:2.3.7} bounds its \(L^2\) norm by \(CA_4\).
Each velocity factor has \(L^2\) norm at most \(CV\).
Hence \(|I_{a,b,1}|\le CeVA_4\).

For \((a,b)=(0,4)\), split
\(\theta_{zzzz}=\Theta_{zzzz}+S''''\).  Since
\(w_{zzz}=-v_{xzz}\), two vertical integrations by parts give
\begin{equation}\label{eq:4.21}
 \int_\Omega q\Theta_{zzzz}w_{zzz}S''
 =-\int_\Omega\Theta_{zz}
       \bigl((Q_2)_{zz}v_{xzz}+2(Q_2)_zv_{xzzz}+Q_2v_{xzzzz}\bigr).
\end{equation}
The last two terms are bounded by \(CeVA_4\).  In the first,
\((Q_2)_{zz}=qS''''+2q_zS'''+q_{zz}S''\).  The first two
coefficients are treated by
\[
 \left|\int_\Omega\Theta_{zz}qS''''v_{xzz}\right|
 \le C\|\Theta_{zz}\|_2\|S''''\|_{L^2(\Omega)}\|v_{xzz}\|_\infty
 \le CeVA_4,
 \qquad \|q_zS'''\|_\infty\le Ce.
\]
For the last coefficient, \eqref{eq:2.1.9} and \eqref{eq:4.8} yield
\begin{equation}\label{eq:4.22}
 \left|\int_\Omega\Theta_{zz}q_{zz}S''v_{xzz}\right|
 \le\|S''\|_\infty\|\Theta_{zz}\|_{L_x^\infty L_z^2}
       \|q_{zz}\|_{L_x^2L_z^\infty}\|v_{xzz}\|_2
 \le Ce^2VA_4.
\end{equation}
Neither \(S''''\) nor \(q_{zz}\) is placed in \(L^\infty\).
For the mean test factor, put \(q_0=(\gamma+S')^{-1}\).
The horizontal mean of \(w\), and hence of \(w_{zzz}\), vanishes.
Therefore
\[
 \int_\Omega qS''''S''w_{zzz}
 =\int_\Omega(q-q_0)S''''S''w_{zzz},
 \qquad q-q_0=-\frac{\Theta_z}{R(\gamma+S')}.
\]
Estimate \eqref{eq:4.7} gives \(\|q-q_0\|_\infty\le CA_4\).
With \(S''''\) and \(w_{zzz}\) in \(L^2\), this contribution
is at most \(Ce^2VA_4\).  Since \(e\le1\), the mean-temperature
terms at order four are bounded by
\(C\mathcal C_ve^2+CeVA_4\).

For \(s\ge5\) and \(b-j\ge1\), use
\begin{equation}\label{eq:4.23}
 \partial_x^a\partial_z^{b-j}w
 =-\partial_x^{a+1}\partial_z^{b-j-1}v.
\end{equation}
For the fluctuation in the test factor, choose
\[
 d:=\begin{cases}
       \max\{0,n-(s-1)\},&a\ge1,\\
       \max\{0,n-(s-2)\},&a=0.
     \end{cases}
\]
Then \(0\le d\le2\) and \(d\le b\).  Transferring these
\(d\) vertical derivatives gives
\begin{equation}\label{eq:4.24}
 \int_\Omega q(\partial_x^a\partial_z^b\Theta)
                (\partial_x^a\partial_z^{b-j}w)S^{(j+1)}
 ={}(-1)^{d+1}\sum_{\ell=0}^d\binom d\ell
      \int_\Omega\partial_x^a\partial_z^{b-d}\Theta\,
           \partial_z^\ell\bigl(qS^{(j+1)}\bigr)
           \partial_x^{a+1}\partial_z^{b-j-1+d-\ell}v.
\end{equation}
The temperature factor has order at most \(s-1\) with a horizontal
derivative, or at most \(s-2\) without one.  Its \(L^2\) norm is
therefore bounded by \(C\|\mathcal M^{\frac12}\Theta\|_{H^s,\mathrm d}\).
The velocity factor has order \(n-j+d-\ell\le s+1\).  Whenever
\(j+\ell+2\le s\),
\[
 \|\partial_z^\ell(qS^{(j+1)})\|_\infty
 \le C\|q\|_{W^{2,\infty}}\|S\|_{H_z^{j+\ell+2}}
 \le CE_1^{\frac12}.
\]
Thus all such terms have the required bound.  The sole exception is
\(a=0\), \(n=s\), \(j=s-3\), and \(d=\ell=2\): indeed,
\(j\le n-3\) and \(d\le1\) unless \(a=0,n=s\).
For this exception,
\[
 \partial_{zz}(qS^{(s-2)})=qS^{(s)}+2q_zS^{(s-1)}+q_{zz}S^{(s-2)}.
\]
The last two terms are bounded in \(L^\infty\) by
\(CE_1^{\frac12}\).  The highest mean derivative is instead
estimated in \(L^2\):
\begin{equation}\label{eq:4.25}
 \left|\int_\Omega q\partial_z^{s-2}\Theta\,S^{(s)}v_{xzz}\right|
 \le{}C\|\partial_z^{s-2}\Theta\|_2\|S^{(s)}\|_{L^2(\Omega)}\|v_{xzz}\|_\infty
 \le{}CE_1^{\frac12}\mathcal D_s^{\frac12}
                 \|\mathcal M^{\frac12}\Theta\|_{H^s,\mathrm d}.
\end{equation}
Here \(\|v_{xzz}\|_\infty\le C\|v\|_{H^{s+1}}\).
In particular, no term of \eqref{eq:4.24} uses more than
\(s\) temperature derivatives.

Only \(a=0\) leaves a mean test factor.  Since
\(\langle w\rangle_x=0\), its contribution equals
\[
 \int_\Omega qS^{(n)}S^{(j+1)}\partial_z^{n-j}w
 =-\int_\Omega\frac{\Theta_z}{R(\gamma+S')}
                   S^{(n)}S^{(j+1)}\partial_z^{n-j}w.
\]
Using \eqref{eq:4.17}, \eqref{eq:4.23}, and \(j+2\le s\), its absolute
value is at most
\[
 C\|\mathcal M^{\frac12}\Theta\|_{H^s,\mathrm d}
   \|S\|_{H_z^s}\|S\|_{W_z^{j+1,\infty}}\|v\|_{H^{s+1}}
 \le CE_1\mathcal D_s^{\frac12}\|\mathcal M^{\frac12}\Theta\|_{H^s,\mathrm d}.
\]
Summing the mean-temperature terms, using \(E_1\le1\), gives
\begin{equation}\label{eq:4.26}
 \sum_{1\le|\alpha|\le s}\sum_{0<\eta<\alpha}\binom\alpha\eta
 \left|\int_\Omega q\partial^\alpha\theta(\partial^\eta w)
                  \partial_z\partial^{\alpha-\eta}S\right|
 \le{}C\mathcal C_vE_1+CE_1^{\frac12}\mathcal D_s^{\frac12}
                         \|\mathcal M^{\frac12}\Theta\|_{H^s,\mathrm d}.
\end{equation}
The fluctuation estimates obtained above likewise give
\begin{equation}\label{eq:4.27}
 \sum_{1\le|\alpha|\le s}\sum_{0<\eta<\alpha}\binom\alpha\eta
 \left|\int_\Omega q\partial^\alpha\theta(\partial^\eta w)
                  \partial_z\partial^{\alpha-\eta}\Theta\right|
 \le{}C\mathcal C_vE_1+CE_1^{\frac12}\mathcal D_s^{\frac12}
                         \|\mathcal M^{\frac12}\Theta\|_{H^s,\mathrm d}.
\end{equation}

It remains to estimate the momentum commutators.  The horizontal
term satisfies, by \eqref{eq:2.1.10},
\[
  \sum_{1\le|\alpha|\le s}
  \left|\int_\Omega \partial^\alpha v[\partial^\alpha,v]\partial_xv\right|
  \le C\|\nabla v\|_{L^\infty}\|v\|_{H^s}^2
  \le C\mathcal C_vE_1.
\]
For the vertical commutator, separate the highest-order terms:
\begin{equation*}
  [\partial^\alpha,w]\partial_zv
  =(\partial^\alpha w^0)\partial_zv
       +(\partial^\alpha\langle w\rangle_z)\partial_zv
     +\sum_{0<\eta<\alpha}\binom\alpha\eta
       (\partial^\eta w)\partial_z\partial^{\alpha-\eta}v.
\end{equation*}
By \eqref{eq:2.1.11} and \eqref{eq:4.16},
\begin{equation*}
\begin{aligned}
  &\sum_{1\le|\alpha|\le s}\sum_{0<\eta<\alpha}\binom\alpha\eta
     \left|\int_\Omega \partial^\alpha v(\partial^\eta w)
                   \partial_z\partial^{\alpha-\eta}v\right|\\
  \le{}&C\|\nabla w\|_{L^\infty}\|v\|_{H^s}^2
       +C\|v\|_{H^s}\|w\|_{H^{s-1}}\|v\|_{W^{2,\infty}}\\
  \le{}&C\mathcal C_vE_1+CE_1^{\frac12}\mathcal D_s.
\end{aligned}
\end{equation*}
The contribution of \(w^0\) at the highest order satisfies
\[
  \sum_{1\le|\alpha|\le s}
   \left|\int_\Omega \partial^\alpha v(\partial^\alpha w^0)\partial_zv\right|
  \le C\|v\|_{H^s}\|w^0\|_{H^s}\|\partial_zv\|_{L^\infty}
  \le CE_1^{\frac12}\mathcal D_s.
\]
If \(\alpha_2\ge1\), then
\(\partial^\alpha\langle w\rangle_z=0\).  For
\(\alpha=(a,0)\) with \(1\le a\le s-1\), we have
\[
  \left|\int_\Omega
       (\partial_x^av)(\partial_x^a\langle w\rangle_z)\partial_zv\right|
  \le C\|v\|_{H^s}\|\langle w\rangle_z\|_{H_x^{s-1}}
         \|\partial_zv\|_{L^\infty}
  \le CE_1\mathcal D_s.
\]
For the remaining horizontal endpoint, set
\[
 B_s:=\int_\Omega(\partial_x^s\langle w\rangle_z)\partial_zv\,\partial_x^sv,
 \qquad G_s:=\langle(\partial_zv)(\partial_x^sv)\rangle_z.
\]
The diagnostic formula \eqref{eq:2.2.1} and integration by parts in
\(x\) give
\begin{equation}\label{eq:4.28}
  B_s=-\frac{L_z}{\gamma}\int_{\mathbb T_x}
          (\partial_x^{s+1}\langle v\theta\rangle_z)G_s\,\mathrm dx
  =\frac{L_z}{\gamma}\int_{\mathbb T_x}
          (\partial_x^s\langle v\theta\rangle_z)\partial_xG_s\,\mathrm dx.
\end{equation}
The averaging and product estimates imply
\begin{equation*}
  \|\partial_x^s\langle v\theta\rangle_z\|_{L_x^2}
  \le C\|v\theta\|_{H^s}
  \le C\|v\|_{H^s}\|\theta\|_{H^s}
  \le CE_1^{\frac12}\mathcal D_s^{\frac12}.
\end{equation*}
Since
\[
  \partial_xG_s
  =\langle(\partial_x\partial_zv)(\partial_x^sv)
            +(\partial_zv)(\partial_x^{s+1}v)\rangle_z,
\]
Jensen's and H\"older's inequalities give
\begin{equation}\label{eq:4.29}
  \|\partial_xG_s\|_{L_x^2}
  \le C\bigl(
       \|\partial_x\partial_zv\|_{L^\infty}\|\partial_x^sv\|_2
       +\|\partial_zv\|_{L^\infty}\|\partial_x^{s+1}v\|_2
       \bigr)
  \le CE_1^{\frac12}\mathcal D_s^{\frac12}.
\end{equation}
Thus \eqref{eq:4.28} yields
\begin{equation*}
  |B_s|
  \le\frac{L_z}{\gamma}
       \|\partial_x^s\langle v\theta\rangle_z\|_{L_x^2}
       \|\partial_xG_s\|_{L_x^2}
  \le CE_1\mathcal D_s.
\end{equation*}
This estimate uses at most \(s\) derivatives of \(\theta\) and
\(s+1\) derivatives of \(v\).  Summing the momentum estimates and
using \(E_1\le1\), we conclude that
\begin{equation}\label{eq:4.30}
  \sum_{1\le|\alpha|\le s}
    \left|\int_\Omega \partial^\alpha v
       \bigl([\partial^\alpha,v]\partial_xv
            +[\partial^\alpha,w]\partial_zv\bigr)\right|
  \le C\mathcal C_vE_1+CE_1^{\frac12}\mathcal D_s.
\end{equation}

Combining \eqref{eq:4.18}, \eqref{eq:4.20},
\eqref{eq:4.26}--\eqref{eq:4.27}, and \eqref{eq:4.30}
in \eqref{eq:4.19}, we obtain
\begin{equation}\label{eq:4.31}
  \frac{\mathrm d}{\mathrm dt}E_2+c_D\mathcal D_s
  \le C_H\mathcal C_vE_2
       +C_HE_1^{\frac12}
          \bigl(\mathcal D_s+\|\mathcal M^{\frac12}\Theta\|_{H^s,\mathrm d}^2\bigr).
\end{equation}

Let \(C_X\) be the constant in \eqref{eq:4.6}, and \(C_D\) the
constant on the right of \eqref{eq:4.3}.  First choose
\(\lambda_E\in(0,1]\), then \(\delta_2\in(0,1]\), so that
\begin{equation*}
 \lambda_EC_X\le\frac{c_E}{2},\qquad
 \lambda_EC_D\le\frac{c_D}{4},\qquad
 C_s\delta_2\le\frac\gamma2,\qquad
 C_H\delta_2\le\min\left\{\frac{c_D}{4},\frac{\lambda_E}{4\gamma}\right\},
\end{equation*}
where \(\|\theta_z\|_\infty\le C_sE_1^{\frac{1}{2}}\).
For smooth solutions with \(E_1^{\frac{1}{2}}\le\delta_2\), these choices
give \eqref{eq:4.13}.  Adding \(\lambda_E\) times
\eqref{eq:4.3} to \eqref{eq:4.31} leaves positive coefficients
\(\frac{c_D}{2}\) and \(\frac{\lambda_E}{4\gamma}\) on the two dissipations.
This proves \eqref{eq:4.14} in the smooth class.

To justify the estimate at strong-solution regularity, apply Lemma~\ref{lem:4.1}.
The margin in \eqref{eq:4.12} ensures \(E_{1,n}^{\frac{1}{2}}\le\delta_2\) for large \(n\).
On \([0,T]\), the preceding convergences imply
\[
\begin{alignedat}{2}
 q_n\to{}&q
 &\qquad&\text{in }C_tH^{s-1}\cap C_tW^{1,\infty},\\
 \mathcal M^{\frac12}\Theta_n\to{}&\mathcal M^{\frac12}\Theta
 &\qquad&\text{in }C_tH^s.
\end{alignedat}
\]
Moreover, \(E_{j,n}\to E_j\) uniformly on \([0,T]\) for
\(j=1,2,3\), and
\[
 \mathcal D_{s,n}\to\mathcal D_s,
 \qquad
 \mathcal C_{v,n}\to\mathcal C_v
 \quad\text{in }L_t^1.
\]
For the last convergence, use \(w_n^0\to w^0\) in \(L_t^2H^s\)
and \(\langle w_n\rangle_z\to\langle w\rangle_z\) in
\(C_tH_x^{s-1}\), followed by Sobolev embedding and H\"older in time.
No \(C_tW^{2,\infty}\) convergence of \(w_n^0\) is needed.

We pass to the integrated identity \eqref{eq:4.19} after
performing the integrations by parts above.  For its highest-order
mean-velocity term, \eqref{eq:4.28} gives the two convergences
\[
 \partial_x^s\langle v_n\theta_n\rangle_z
       \to\partial_x^s\langle v\theta\rangle_z
            \quad\text{in }L_t^2L_x^2,
 \qquad \partial_xG_{s,n}\to\partial_xG_s
            \quad\text{in }L_t^2L_x^2.
\]
The product estimates and \eqref{eq:4.29} give these convergences.
By the compensated identity \eqref{eq:4.28},
\[
\begin{aligned}
 \|B_{s,n}-B_s\|_{L_t^1}
 \le{}&C\|\partial_x^s\langle v_n\theta_n-v\theta\rangle_z\|_{L_t^2L_x^2}
          \|\partial_xG_{s,n}\|_{L_t^2L_x^2}\\
 &+C\|\partial_x^s\langle v\theta\rangle_z\|_{L_t^2L_x^2}
          \|\partial_x(G_{s,n}-G_s)\|_{L_t^2L_x^2}\longrightarrow0.
\end{aligned}
\]
For the exceptional temperature pairings at \(s=4\), we use
\[
\begin{alignedat}{2}
 \Theta_{n,zz}\to{}&\Theta_{zz}
 &\qquad&\text{in }C_t(L_x^\infty L_z^2),\\
 q_{n,zz}\to{}&q_{zz}
 &\qquad&\text{in }C_t(L_x^2L_z^\infty),\\
 S_n^{(4)}\to{}&S^{(4)}
 &\qquad&\text{in }C_tL_z^2,\\
 v_{n,xzz}\to{}&v_{xzz}
 &\qquad&\text{in }L_t^2L^\infty.
\end{alignedat}
\]
The first convergence uses the horizontal derivative in
\eqref{eq:4.8}.  The second follows from the
formula for \(q_{zz}\) and \eqref{eq:2.1.9}.
Together with convergence of the fifth velocity derivatives in
\(L_t^2L^2\), these limits pass \eqref{eq:4.21} and
\eqref{eq:4.22} in \(L_t^1\).
The remaining fourth-order configurations pass to the limit using
the pairings established above.
For \(s\ge5\), one has \(q_n\to q\) in \(C_tW^{2,\infty}\).
The ordinary coefficients in \eqref{eq:4.24} satisfy
\[
 \partial_z^\ell(q_nS_n^{(j+1)})
 \longrightarrow\partial_z^\ell(qS^{(j+1)})
 \quad\text{in }C_tL^\infty,
 \qquad j+\ell+2\le s,\quad 0\le\ell\le2.
\]
Indeed, \(S_n\to S\) in \(C_tH_z^s\), so every mean derivative
of order at most \(s-1\) converges uniformly by the one-dimensional
Sobolev embedding.  These are the coefficient norms used in
\eqref{eq:4.24}.  For the exceptional term \eqref{eq:4.25},
use \(S_n^{(s)}\to S^{(s)}\) in \(C_tL_z^2\) and
\(v_{n,xzz}\to v_{xzz}\) in \(L_t^2L^\infty\).

All remaining commutator and weight terms converge in \(L_t^1\)
by the product bounds and subtraction of individual factors.
The limiting identity has an \(L^1\) right-hand side, so
\(E_2\in W^{1,1}(0,T)\).  The cross identity gives the same
regularity for \(X\), hence for \(E_3\).  Passing to the limit in the
integrated inequality between arbitrary endpoints proves
\eqref{eq:4.14} at strong-solution regularity.
\end{proof}

To estimate the spectral energy, we need product bounds that retain
the hydrostatic weight in the vertical transport and diagnostic mean.
At the fourth-order endpoint, an \(H^3\) coefficient replaces the
Wiener coefficient used at higher orders.
\begin{lemma}\label{lem:4.2}
\begin{enumerate}
\item[\textup{(i)}] For \(r\ge0\), \(\varphi\in\mathcal W^{r+2}\), and
\(f\in H^r\) with \(\mathcal Qf=f\),
\[
 \|\mathcal M^{\frac12}\mathcal Q(\varphi f)\|_{H^r}
 +\|\langle\varphi f\rangle_z\|_{H_x^r}
 \le C_r\|\varphi\|_{\mathcal W^{r+2}}
          \|\mathcal M^{\frac12}f\|_{H^r}.
\]
The constants are bounded when \(r\) ranges over a fixed compact
subset of \([0,\infty)\).
\item[\textup{(ii)}] For \(\varphi\in H^3\) and \(f\in L^2\) with
\(\mathcal Qf=f\),
\begin{equation}\label{eq:4.32}
 \|\mathcal M^{\frac12}\mathcal Q(\varphi f)\|_2
 +\|\langle\varphi f\rangle_z\|_{L_x^2}
 \le C\|\varphi\|_{H^3}\|\mathcal M^{\frac12}f\|_2.
\end{equation}
The constant depends only on the fixed parameters and periods.
\end{enumerate}
\end{lemma}
\begin{proof}
For part~\textup{(i)}, let
\(\xi=(k,m),\eta=(\ell,n)\in\mathcal K_*^x\).  We use
\begin{equation*}
 \frac{\rho(\xi)^{\frac12}}{\rho(\eta)^{\frac12}}
 \le C\langle\xi-\eta\rangle^2.
\end{equation*}
If \(|\xi-\eta|\ge\frac{|\eta|}{2}\), this follows from
\eqref{eq:2.3.1}.  Otherwise \(|\xi|\) and \(|\eta|\) are comparable,
and
\[
 \frac{\rho(\xi)^{\frac12}}{\rho(\eta)^{\frac12}}
 =\frac{|k||n|}{|\ell||m|}\left(\frac{a(\eta)}{a(\xi)}\right)^{\frac12}
 \le C(1+|k-\ell|)(1+|m-n|).
\]
Therefore
\[
 \langle\xi\rangle^r\rho(\xi)^{\frac12}|\widehat{\varphi f}(\xi)|
 \le C_r\sum_{\eta\in\mathcal K_*^x}
 \langle\xi-\eta\rangle^{r+2}|\widehat\varphi(\xi-\eta)|
 \langle\eta\rangle^r\rho(\eta)^{\frac12}|\widehat f(\eta)|.
\]
Young's inequality yields the first norm.  For output \(\xi=(k,0)\),
\[
 \frac{\langle k\rangle^r}{\langle\eta\rangle^r\rho(\eta)^{\frac12}}
 \le C_r\langle\xi-\eta\rangle^r(1+n^2)
 \le C_r\langle\xi-\eta\rangle^{r+2},
\]
which gives the vertical-average norm by the same convolution bound.
Fourier truncation extends the calculation to the stated spaces.
The weight constants are bounded on compact intervals of \(r\).

For part~\textup{(ii)}, first take trigonometric polynomials.  Set
\(h=\mathcal L^{-\frac12}\mathcal Tf\), and define
\(\partial_x^{-1}\) by the symbol \((ik)^{-1}\) for \(k\ne0\)
and zero for \(k=0\).  On the range of \(\mathcal Q\),
\[
 \|h\|_2=\frac1N\|\mathcal M^{\frac12}f\|_2,
 \qquad
 f=\partial_z\partial_x^{-1}\mathcal L^{\frac12}h.
\]
The identity
\(\mathcal L^{\frac12}
=-\mu_x\partial_x\mathcal L^{-\frac12}\partial_x
 -\mu_z\partial_z\mathcal L^{-\frac12}\partial_z\)
gives
\begin{equation*}
 f=\partial_zF+\partial_{zz}G,
 \qquad F=-\mu_x\mathcal L^{-\frac12}\partial_xh,
 \qquad G=-\mu_z\partial_x^{-1}\mathcal L^{-\frac12}\partial_zh.
\end{equation*}
The symbols \(\frac{|k|}{\sqrt a}\),
\(\frac{|m|}{|k|\sqrt a}\), and \(\frac{|m|}{\sqrt a}\)
are bounded on \(\mathcal K_*^x\).  Hence
\[
 \|F\|_2+\|G\|_2+\|\partial_xG\|_2
 +\|G\|_{L_x^\infty L_z^2}\le C\|h\|_2.
\]
The product has the exact expansion
\begin{equation}\label{eq:4.33}
 \varphi f=\partial_z(\varphi F)-\varphi_zF
       +\partial_{zz}(\varphi G)-2\partial_z(\varphi_zG)
       +\varphi_{zz}G.
\end{equation}
Write \(\mathcal B=\mathcal L^{-\frac12}\mathcal T\), with its
symbol extended by zero at \(m=0\).  Then
\[
 \mathcal B\mathcal Q=\mathcal B,
 \qquad
 \mathcal B\partial_z=\mathcal L^{-\frac12}\partial_x\Pi_{\ne0}^z,
 \qquad
 \mathcal B\partial_{zz}=\mathcal L^{-\frac12}\partial_z\partial_x.
\]
These identities retain the zero-mode projection.  In particular,
\[
 \mathcal B\partial_{zz}(\varphi G)
 =\mathcal L^{-\frac12}\partial_z
       \bigl(\varphi_xG+\varphi\partial_xG\bigr).
\]
The boundedness of \(\mathcal B\),
\(\mathcal L^{-\frac12}\partial_x\), and
\(\mathcal L^{-\frac12}\partial_z\) on the indicated subspaces
controls all terms except \(\varphi_{zz}G\) by
\(C\|\varphi\|_{W^{1,\infty}}\|h\|_2\).  For the last term,
\eqref{eq:2.1.9} gives
\[
 \|\varphi_{zz}G\|_2
 \le\|\varphi_{zz}\|_{L_x^2L_z^\infty}
       \|G\|_{L_x^\infty L_z^2}
 \le C\|\varphi\|_{H^3}\|h\|_2.
\]
Since \(H^3(\Omega)\hookrightarrow W^{1,\infty}(\Omega)\)
and \(\|\mathcal M^{\frac12}u\|_2=N\|\mathcal Bu\|_2\),
this proves the first term in \eqref{eq:4.32}.
Taking the vertical average of \eqref{eq:4.33}
yields
\[
 \langle\varphi f\rangle_z
 =-\langle\varphi_zF\rangle_z+\langle\varphi_{zz}G\rangle_z.
\]
The same two product bounds and the boundedness of averaging from
\(L^2(\Omega)\) to \(L_x^2\) prove the second term.  Finally,
approximate \(\varphi\) in \(H^3\) and \(f\) in \(L^2\) by
Fourier truncations.  The products converge in \(L^2\), so the
bounded multipliers and averaging operators pass to the limit.
\end{proof}

We use this lemma to absorb the vertical transport into the spectral
dissipation, leaving the horizontal transport as the forcing.
For \(\lambda_{\mathrm M}>0\), define
\begin{equation*}
 E_{5,r}=\|\mathcal Mv\|_{H^r}^2
 +\frac{g\beta}{\gamma}\|\mathcal M\Theta\|_{H^r}^2
 +\lambda_{\mathrm M}\operatorname{Re}
       \langle\mathcal L^{-1}\mathcal T\mathcal Mv,\mathcal M\Theta\rangle_{H^r}.
\end{equation*}
\begin{proposition}\label{prop:4.2}
Let \(s\ge4\) be an integer.  There exist
\(\delta_{\mathrm M}\in(0,1]\), \(\lambda_{\mathrm M}>0\), and
\(c,C>0\) such that every compatible \(H^s\)-strong solution on
\([0,T]\), with \(\sup_{t\le T}E_1(t)^{\frac{1}{2}}\le\delta_{\mathrm M}\),
has the following properties.  For
\(0\le r\le\max\{0,s-5\}\), \(E_{5,r}\in W^{1,1}(0,T)\),
\begin{equation}\label{eq:4.34}
 c\bigl(\|\mathcal Mv\|_{H^r}^2+\|\mathcal M\Theta\|_{H^r}^2\bigr)
 \le E_{5,r}
 \le C\bigl(\|\mathcal Mv\|_{H^r}^2+\|\mathcal M\Theta\|_{H^r}^2\bigr),
\end{equation}
and, for almost every \(t\in(0,T)\),
\begin{equation}\label{eq:4.35}
 \frac{\mathrm d}{\mathrm dt}E_{5,r}
 +c\bigl(\|\mathcal L^{\frac12}\mathcal Mv\|_{H^r}^2
             +\|\mathcal M^{\frac32}\Theta\|_{H^r}^2\bigr)
 \le C\bigl(\|v\partial_xv\|_{H^r}
                  +\|v\partial_x\Theta\|_{H^r}\bigr)^2.
\end{equation}

The constants are independent of \(T\) and the solution, and are
uniform over the stated range of \(r\) for fixed \(s\).
\end{proposition}
\begin{proof}
For \(s\ge5\) and \(0\le r\le s-5\), \eqref{eq:2.1.7} gives
\[
 \|v_z\|_{\mathcal W^{r+2}}+\|\theta_z\|_{\mathcal W^{r+2}}
 \le C\bigl(\|v\|_{H^{r+5}}+\|\theta\|_{H^{r+5}}\bigr)
 \le CE_1^{\frac12}.
\]
At \(s=4,r=0\), we instead use
\(\|v_z\|_{H^3}+\|\theta_z\|_{H^3}\le CE_1^{\frac12}\).
Thus Lemma~\ref{lem:4.2}\textup{(i)} applies when \(s\ge5\), and
part~\textup{(ii)} applies when \(s=4,r=0\), with
\(f=w^0\) supported on \(\mathcal K_*^x\).
Applying the weighted product estimate to the mean identity
\eqref{eq:2.2.4} yields
\begin{equation*}
 \|\langle w\rangle_z\|_{H_x^r}
 \le CE_1^{\frac12}\|\mathcal M^{\frac12}w^0\|_{H^r}
       +C\|v\partial_x\Theta\|_{H^r}.
\end{equation*}
The ordinary product estimate gives
\begin{equation}\label{eq:4.36}
 \|\langle w\rangle_z\|_{H_x^r}
 \le CE_1^{\frac12}\|w^0\|_{H^r}+C\|v\partial_x\Theta\|_{H^r}.
\end{equation}

For \(f=v\) or \(\theta\), the same weighted bound controls
\(\mathcal M^{\frac{1}{2}}\mathcal Q(w^0f_z)\).
The remaining product obeys
\(\|\mathcal M^{\frac{1}{2}}\mathcal Q(\langle w\rangle_z f_z)\|_{H^r}
 \le CE_1^{\frac{1}{2}}\|\langle w\rangle_z\|_{H_x^r}\),
using \(\mathcal W^r\) coefficients for \(s\ge5\) and bounded
coefficients at \(s=4\).  Hence, for \(E_1\le1\),
\begin{equation}\label{eq:4.37}
 \sum_{i=1}^2\|\mathcal M^{\frac12}\mathcal Q\mathcal N_i\|_{H^r}
 \le CE_1^{\frac12}\|\mathcal M^{\frac12}w^0\|_{H^r}
      +C\bigl(\|v\partial_xv\|_{H^r}+\|v\partial_x\Theta\|_{H^r}\bigr).
\end{equation}
\begin{equation}\label{eq:4.38}
 \sum_{i=1}^2\|\mathcal N_i\|_{H^r}
 \le CE_1^{\frac12}\|w^0\|_{H^r}
      +C\bigl(\|v\partial_xv\|_{H^r}+\|v\partial_x\Theta\|_{H^r}\bigr).
\end{equation}

Both estimates retain \(w\partial_z\langle\theta\rangle_x\) and
require no decay of the mean temperature.

Apply \eqref{eq:2.3.13} and \eqref{eq:2.3.12} to
\((u,f)=(\mathcal Mv,\mathcal M\Theta)\), with sources
\(\mathcal M\mathcal Q\mathcal N_i\), in the Fourier \(H^r\) pairing.
Since \(r\le s-1\) throughout the stated range, the strong-solution
regularity and boundedness of \(\mathcal M\) place both unknowns
in \(H_t^1H^r\).  The velocity also has the spatial derivatives
required for the viscous pairing.
The linear cross terms are bounded by
\[
 C\|\mathcal L^{\frac{1}{2}}\mathcal Mv\|_{H^r}
          \|\mathcal M^{\frac{3}{2}}\Theta\|_{H^r}
       +C\|\mathcal L^{\frac{1}{2}}\mathcal Mv\|_{H^r}^2.
\]
Transfer half a power of \(\mathcal M\) in each source pairing, as in
\[
 \langle\mathcal M\mathcal Q\mathcal N_2,\mathcal M\Theta\rangle_{H^r}
 =\langle\mathcal M^{\frac{1}{2}}\mathcal Q\mathcal N_2,
                         \mathcal M^{\frac{3}{2}}\Theta\rangle_{H^r}.
\]
We estimate the remaining source terms using the boundedness of
\(\mathcal L^{-1}\mathcal T\) and
\(\mathcal M^{\frac12}\mathcal L^{-\frac12}\).
For the direct source pairing, we have
\[
 |\langle\mathcal M\mathcal Q\mathcal N_1,\mathcal Mv\rangle_{H^r}|
 \le C\|\mathcal M^{\frac12}\mathcal Q\mathcal N_1\|_{H^r}
       \|\mathcal L^{\frac12}\mathcal Mv\|_{H^r}.
\]
The two cross terms satisfy
\[
 |\langle\mathcal L^{-1}\mathcal T\mathcal M\mathcal Q\mathcal N_1,
          \mathcal M\Theta\rangle_{H^r}|
 \le C\|\mathcal M^{\frac12}\mathcal Q\mathcal N_1\|_{H^r}
       \|\mathcal M^{\frac32}\Theta\|_{H^r},
\]
and
\[
 |\langle\mathcal L^{-1}\mathcal T\mathcal Mv,
          \mathcal M\mathcal Q\mathcal N_2\rangle_{H^r}|
 \le C\|\mathcal M^{\frac12}\mathcal Q\mathcal N_2\|_{H^r}
       \|\mathcal L^{\frac12}\mathcal Mv\|_{H^r}.
\]
First choose \(\lambda_{\mathrm M}\) small enough for energy
equivalence and linear absorption.  Young's inequality then yields
\begin{equation*}
 \frac{\mathrm d}{\mathrm dt}E_{5,r}
 +c\bigl(\|\mathcal L^{\frac12}\mathcal Mv\|_{H^r}^2
              +\|\mathcal M^{\frac32}\Theta\|_{H^r}^2\bigr)
 \le C\sum_{i=1}^2
       \|\mathcal M^{\frac12}\mathcal Q\mathcal N_i\|_{H^r}^2.
\end{equation*}

The exact relation
\begin{equation*}
 \|\mathcal L^{\frac12}\mathcal Mv\|_{H^r}^2
 =N^2\|\mathcal M^{\frac12}w^0\|_{H^r}^2,
\end{equation*}
follows from \(a\rho^2=N^2\rho \frac{k^2}{m^2}\).  Both sides vanish at
\(k=0\).  Substitute \eqref{eq:4.37} and choose
\(C\delta_{\mathrm M}^2\le \frac{cN^2}{2}\).  The vertical transport terms
are absorbed, proving \eqref{eq:4.35}.
The weighted product constants are bounded on \([0,s-5]\) for
\(s\ge5\).  At \(s=4\), only \(r=0\) occurs.  This gives the
stated uniformity and completes the proof.
\end{proof}

\subsection{Proof of Theorem~\ref{thm:1.1}}\label{subsec:4.2}
The high-order energy, spectral decay, and diagnostic time integral
are estimated jointly on finite intervals of the local lifespan.
These bounds yield global continuation and convergence without
assuming the decay estimates of Theorem~\ref{thm:1.3}.

\begin{proof}[Proof of Theorem~\ref{thm:1.1}]
We first work at order four and set
\[
\varepsilon=\|v_0\|_{H^4}+\|\theta_0\|_{H^4}.
\]
The case \(\varepsilon=0\) is immediate, so assume \(\varepsilon>0\)
and fix \(0<T<T_{\max}^{(4)}\).
By \eqref{eq:2.2.5} and \eqref{eq:2.2.8}, the sources
\(F_i=\Pi_{\ne0}^z\mathcal N_i\) belong to \(C([0,T];H^3)\)
and have zero vertical means.  Thus Lemma~\ref{lem:3.3} applies
to \eqref{eq:2.2.3}, with source norms bounded by those of
\(\mathcal N_i\) in both Sobolev and Wiener spaces.

For the horizontal zero mode, incompressibility gives
\[
 \mathcal N_1=-\partial_x(v^2)-\partial_z(wv),\qquad
 \langle\mathcal N_1\rangle=0,\qquad
 \Pi_0^xF_1=\Pi_0^x\mathcal N_1.
\]
Hence \eqref{eq:3.2.11} has source \(\Pi_0^x\mathcal N_1\)
with zero vertical mean.  We estimate \(\langle w\rangle_z\)
separately by its compatibility identity and evaluate all energies
at order four until the higher-order argument below.  Define
\begin{equation}\label{eq:4.39}
 J(T):=\int_0^T\left(\|w^0(t)\|_{\mathcal W^2}
   +\|\langle w\rangle_z(t)\|_{\mathcal W_x^2}
   +\|\partial_z\Pi_0^xv(t)\|_{\mathcal W_z^1}\right)\,\mathrm dt.
\end{equation}

For sufficiently small fixed \(\delta_{\mathrm I}\in(0,1]\),
we claim that \(\sup_{t\le T}e(t)\le\delta_{\mathrm I}\) implies
\begin{equation}\label{eq:4.40}
 J(T)\le C\left(\varepsilon+E_*(T)
              +\int_0^T\|\mathcal M\Theta(t)\|_{\mathcal W^2}\,\mathrm dt\right).
\end{equation}
We also prove
\begin{equation}\label{eq:4.41}
 \int_0^T\mathcal C_v(t)\,\mathrm dt\le CJ(T).
\end{equation}
The constants are independent of \(T\) and higher solution norms.
The local regularity \(w^0\in L_t^2H^4\),
\(\langle w\rangle_z\in C_tH_x^3\), and \(v\in C_tH^4\),
together with \eqref{eq:2.1.7}--\eqref{eq:2.1.8}, ensures
\(J(T)<\infty\).

For \(w^0\), apply \eqref{eq:3.2.7} with these sources and
integrate the kernels before summing over frequencies.
On \(\Sigma_2^x\),
\begin{equation*}
\begin{aligned}
 \langle k,m\rangle^2\int_0^\infty
       \left|\frac{k}{m}e^{\lambda_-u}\right|\,\mathrm du
 \le{}&C\frac{|k|}{|m|}\le C\langle k,m\rangle,\\
 \langle k,m\rangle^2\int_0^\infty
       \left|\frac{k}{m}h_+e^{\lambda_-u}\right|\,\mathrm du
 \le{}&C\frac{\langle k,m\rangle^2\rho}{a}\le C.
\end{aligned}
\end{equation*}
The initial fast contribution is bounded by
\(C(\|v_0\|_{\mathcal W^1}+\|\Theta_0\|_{\mathcal W^0})\),
and \eqref{eq:3.1.9} controls the finite region \(\Sigma_1^x\).
Keeping the slow contribution in \eqref{eq:3.2.8} as
\(\mathcal M\Theta(t)\), we obtain by Tonelli's theorem
\begin{equation}\label{eq:4.42}
\int_0^T\|w^0\|_{\mathcal W^2}\,\mathrm dt \le{}C\varepsilon+C\int_0^T\|\mathcal M\Theta\|_{\mathcal W^2}\,\mathrm dt +C\int_0^T\bigl(\|\mathcal N_1\|_{\mathcal W^1} +\|\mathcal N_2\|_{\mathcal W^0}\bigr)\,\mathrm dt.
\end{equation}
The compatible vertical mean satisfies
\begin{equation}\label{eq:4.43}
 \|\langle w\rangle_z\|_{\mathcal W_x^2}
 \le Ce\bigl(\|w^0\|_{\mathcal W^2}+\|v\|_{\mathcal W^0}\bigr)
          +CVA_4.
\end{equation}
Since \(\langle w\rangle=0\), the output horizontal weight may
be taken as \(|K|^2\).  Write \(K=k+\ell\) and use
\(|K|^2\le2k^2+2\ell^2\) in each product in
\eqref{eq:2.2.4}.
In \(\langle w^0\theta_z\rangle_z\), the part with weight \(k^2\) is bounded by
\[
C\|w^0\|_{\mathcal W^2}\|\theta\|_{\mathcal W^1} \le Ce\|w^0\|_{\mathcal W^2}.
\]
For the part with weight \(\ell^2\), use
\(|m|\,|\widehat{w^0}(k,m)|=|k|\,|\widehat v(k,m)|\).
It is bounded by
\[
 C\sum_{m\ne0}\left(\sum_k|k|\,|\widehat v(k,m)|\right)
                  \left(\sum_{\ell\ne0}\ell^2
                         |\widehat\Theta(\ell,-m)|\right)
 \le C\|v\|_{H^2}\|\partial_x^3\Theta\|_2\le CVA_4.
\]
Here Cauchy--Schwarz is applied first in the horizontal frequencies
and then in \(m\).  The last bound uses \eqref{eq:4.7}.
In \(\langle v\Theta_x\rangle_z\), the part with weight
\(k^2\) is bounded by
\[
C\|v\|_{\mathcal W^2}\|\Theta_x\|_{\mathcal W^0}\le CVA_4.
\]
For the other part,
\[
 \sup_m\sum_{\ell\ne0}|\ell|^3|\widehat\Theta(\ell,m)|
 \le C\sup_m\left(\sum_\ell\ell^8
                       |\widehat\Theta(\ell,m)|^2\right)^{\frac12}
 \le C\|\theta\|_{H^4}\le Ce.
\]
Multiplication by \(\|v\|_{\mathcal W^0}\) completes
\eqref{eq:4.43}.

The diagnostic relation and the lower bounds on nonzero lattice
frequencies imply
\begin{equation}\label{eq:4.44}
 \|v\|_{\mathcal W^0}
 \le C\bigl(\|w^0\|_{\mathcal W^1}
               +\|\partial_z\Pi_0^xv\|_{\mathcal W_z^0}\bigr).
\end{equation}
Since \(e\le1\), we use the diagnostic bound
\[
\|w\|_{\mathcal W^1}+\|v\|_{\mathcal W^2}\le CV
\]
and the product estimate to obtain
\begin{equation*}
 \|\mathcal N_1\|_{\mathcal W^1}\le CV^2.
\end{equation*}
Integration of the zero-mean heat kernels in \eqref{eq:3.2.11}
then yields
\begin{equation}\label{eq:4.45}
 \int_0^T\|\partial_z\Pi_0^xv\|_{\mathcal W_z^1}\,\mathrm dt
 \le C\varepsilon+C\int_0^T\|\Pi_0^x\mathcal N_1\|_{\mathcal W_z^0}\,\mathrm dt
 \le C\bigl(\varepsilon+E_*(T)\bigr).
\end{equation}
Indeed, for \(m\ne0\),
\(\langle m\rangle|m|\int_0^\infty e^{-\mu_zm^2u}\,\mathrm du
\le C\).
Also \(\|\theta\|_{\mathcal W^1}\le Ce\), so
\eqref{eq:4.44} implies
\begin{equation*}
 \|\mathcal N_2\|_{\mathcal W^0}
 \le Ce\bigl(\|w^0\|_{\mathcal W^2}
        +\|\langle w\rangle_z\|_{\mathcal W_x^2}
        +\|\partial_z\Pi_0^xv\|_{\mathcal W_z^0}\bigr).
\end{equation*}
Finally, it follows from the Cauchy--Schwarz inequality and
\eqref{eq:4.4} that
\[
 \int_0^T VA_4\,\mathrm dt
 \le\left(\int_0^T V^2\,\mathrm dt\right)^{\frac12}
    \left(\int_0^T A_4^2\,\mathrm dt\right)^{\frac12}
 \le CE_*(T).
\]
Adding \eqref{eq:4.42}, the time integral of \eqref{eq:4.43},
and \eqref{eq:4.45}, we obtain
\[
 J(T)\le C\left(\varepsilon+E_*(T)
                +\int_0^T\|\mathcal M\Theta\|_{\mathcal W^2}\,\mathrm dt\right)
            +C\delta_{\mathrm I}J(T).
\]
Absorbing the last term for small \(\delta_{\mathrm I}\) proves
\eqref{eq:4.40}.  To prove \eqref{eq:4.41}, use
\(\widehat v=-\frac m k\widehat{w^0}\) on \(\mathcal K_*^x\):
\[
 \|\nabla\Pi_{\ne0}^xv\|_\infty
 \le\sum_{\mathcal K_*^x}
       \left(|m|+\frac{m^2}{|k|}\right)|\widehat{w^0}(k,m)|
 \le C\|w^0\|_{\mathcal W^2}.
\]
The remaining terms in \(\mathcal C_v\) are controlled by the
integrand of \(J\), proving \eqref{eq:4.41}.

We relate the diagnostic time integral to the spectral energy by
two interpolation estimates.  For \(\mathcal Qf=f\), we have
\begin{equation}\label{eq:4.46}
 \|\partial_xf\|_2
 \le C\|\mathcal Mf\|_2^{\frac12}\|f\|_{H^4}^{\frac12},
\end{equation}
and
\begin{equation}\label{eq:4.47}
 \|\mathcal Mf\|_{\mathcal W^2}
 \le C\|\mathcal Mf\|_2^{\frac7{20}}
       \|\mathcal M^{\frac12}f\|_{H^4}^{\frac{13}{20}}.
\end{equation}
Estimate~\eqref{eq:4.46} follows by Fourier interpolation from \(k^2\le C\rho\langle k,m\rangle^4\).
To prove \eqref{eq:4.47}, set \(b=\frac{13}{20}\) and split the sum \(\sum_{\mathcal K_*^x}\rho^b\langle k,m\rangle^{4-8b}\)
into the regions \(|k|\ge|m|\) and \(|m|>|k|\).
The first part converges since \(\rho\le Cm^{-2}\) and \(b>\frac58\), and the second since \(\rho\le Ck^2m^{-4}\) and \(b>\frac35\).
Applying Cauchy--Schwarz, followed by Fourier interpolation, we obtain
\[
 \|\mathcal Mf\|_{\mathcal W^2}
 \le C\|\mathcal M^{1-\frac b2}f\|_{H^{4b}}
 \le C\|\mathcal Mf\|_2^{1-b}
       \|\mathcal M^{\frac12}f\|_{H^4}^{b}.
\]
If, in addition, \(\|\mathcal Mf(t)\|_2\le A\varepsilon(1+t)^{-2}\)
on \([0,T]\), then \eqref{eq:4.47} implies
\begin{equation}\label{eq:4.48}
 \int_0^T\|\mathcal Mf(t)\|_{\mathcal W^2}\,\mathrm dt
 \le C(A\varepsilon)^{\frac7{20}}
 \left(\int_0^T\|\mathcal M^{\frac12}f(t)\|_{H^4}^2
          \,\mathrm dt\right)^{\frac{13}{40}}.
\end{equation}
Indeed, by H\"older's inequality,
\[
 \int_0^T(1+t)^{-\frac7{10}}
       \|\mathcal M^{\frac12}f(t)\|_{H^4}^{\frac{13}{20}}\,\mathrm dt
 \le\left(\int_0^T
       \|\mathcal M^{\frac12}f(t)\|_{H^4}^2\,\mathrm dt\right)^{\frac{13}{40}}
    \left(\int_0^\infty
       (1+t)^{-\frac{28}{27}}\,\mathrm dt\right)^{\frac{27}{40}}.
\]
Since \(\frac{28}{27}>1\), the constant is independent of \(T\).

Under the same smallness condition, if
\(E_*(T)\le K\varepsilon^2\) and
\(\|\mathcal M\Theta(t)\|_2\le A\varepsilon(1+t)^{-2}\),
\eqref{eq:4.48} and \eqref{eq:4.4} give
\begin{equation}\label{eq:4.49}
 \int_0^T\|\mathcal M\Theta(t)\|_{\mathcal W^2}\,\mathrm dt
 \le C A^{\frac7{20}}K^{\frac{13}{40}}\varepsilon.
\end{equation}
We now close the joint bootstrap on the local lifespan.
Fix the energy parameters in Propositions~\ref{prop:4.1} and
\ref{prop:4.2}, and choose
\(0<\delta\le\min\{\frac{\delta_2}{2},
\delta_{\mathrm M},\delta_{\mathrm I}\}\)
small enough that \(e\le\delta\) implies
\(\|\theta_z\|_\infty\le\frac{\gamma}{4}\).
Set \(Y=E_{5,0}\).
For constants \(K,A,B\ge1\) to be chosen below, we impose the
bootstrap assumptions
\begin{equation}\label{eq:4.50}
\left\{
\begin{aligned}
 &E_*(T)\le K\varepsilon^2,\qquad J(T)\le B\varepsilon,\\
 &\sup_{0\le t\le T}(1+t)^2Y(t)^{\frac12}\le A\varepsilon,\\
 &\sup_{0\le t\le T}(1+t)^2\|w^0(t)\|_2\le A\varepsilon,\\
 &\sup_{0\le t\le T}(1+t)^{\frac32}
       \|v(t)\|_{\mathcal W^0}\le A\varepsilon.
\end{aligned}
\right.
\end{equation}
The quantities in \eqref{eq:4.50} are continuous in \(T\).
At \(T=0\), the bounds hold with \(K=A=C_{\mathrm{in}}\) for a
fixed \(C_{\mathrm{in}}\ge1\), and \(J(0)=0\).  Thus large fixed
\(K,A,B\) and \(\sqrt K\varepsilon\le\delta\) start the strict
bootstrap on a short interval.

By Proposition~\ref{prop:4.1}, \eqref{eq:4.41}, and Gronwall's
inequality,
\begin{equation*}
 E_*(T)+\int_0^T A_4(t)^2\,\mathrm dt
 \le C_E\varepsilon^2\exp\bigl(C_EB\varepsilon\bigr).
\end{equation*}
Choose \(K\ge\max\{4C_E,2C_{\mathrm{in}},1\}\).
If \(C_EB\varepsilon\le\log2\), then the first bound in
\eqref{eq:4.50} improves to
\(E_*(T)\le\frac K2\varepsilon^2\).
Below, \(C_K,c_K>0\) may depend on this fixed \(K\), but not on
\(A,B,T\) or \(\varepsilon\).

For the spectral bound, note that
\(\partial_xv=\partial_x\mathcal Qv\).
It follows from \eqref{eq:4.46}, \eqref{eq:4.34}, and
\eqref{eq:4.50} that
\[
 \|\partial_xv(t)\|_2+\|\partial_x\Theta(t)\|_2
 \le C_K A^{\frac12}\varepsilon(1+t)^{-1}.
\]
Combining this estimate with the \(\mathcal W^0\)-bound on \(v\)
in \eqref{eq:4.50}, we obtain
\begin{equation}\label{eq:4.51}
 \|v\partial_xv\|_2+\|v\partial_x\Theta\|_2
 \le C_K A^{\frac32}\varepsilon^2(1+t)^{-\frac52}.
\end{equation}
On \(\mathcal K_*^x\), \eqref{eq:2.3.1} yields
\(\rho^{-2}\le C\langle k,m\rangle^8\).
H\"older's inequality in the Fourier sum gives
\begin{equation}\label{eq:4.52}
\|\mathcal M\Theta\|_2^2 \le{}\left(\|\mathcal M^{\frac32}\Theta\|_2^2\right)^{\frac45}
\left(L_xL_z\sum_{\mathcal K_*^x} \rho^{-2}|\widehat\Theta|^2\right)^{\frac15}
\le{}C_K\varepsilon^{\frac25} \left(\|\mathcal M^{\frac32}\Theta\|_2^2\right)^{\frac45}.
\end{equation}
For the velocity, we combine Poincar\'e's inequality with the
boundedness of \(\mathcal M\) and the bootstrap bound to obtain
\[
 \|\mathcal Mv\|_2^2
 \le C\|\mathcal Mv\|_2^{\frac25}
       \|\mathcal L^{\frac12}\mathcal Mv\|_2^{\frac85}
 \le C_K\varepsilon^{\frac25}
       \|\mathcal L^{\frac12}\mathcal Mv\|_2^{\frac85}.
\]
By H\"older's inequality, the preceding velocity estimate and
\eqref{eq:4.52}, together with the energy equivalence
\eqref{eq:4.34}, imply
\[
 \|\mathcal L^{\frac12}\mathcal Mv\|_2^2
 +\|\mathcal M^{\frac32}\Theta\|_2^2
 \ge c_K\varepsilon^{-\frac12}Y^{\frac54}.
\]
Substituting the preceding lower bound and \eqref{eq:4.51}
into \eqref{eq:4.35}, we obtain
\begin{equation*}
 Y'+c_K\varepsilon^{-\frac12}Y^{\frac54}
 \le C_KA^3\varepsilon^4(1+t)^{-5}.
\end{equation*}
Choose \(K_{\mathrm M}\ge1\), depending only on \(K\) and
fixed constants, so that
\[
 Y(0)\le K_{\mathrm M}\varepsilon^2,
 \qquad c_KK_{\mathrm M}^{\frac54}-4K_{\mathrm M}\ge1.
\]
Apply \eqref{eq:2.3.15}--\eqref{eq:2.3.16}
with \(\alpha=4\), \(\beta=5\), and \(F_0=C_KA^3\).
When \(C_KA^3\varepsilon^2\le1\), this gives
\begin{equation}\label{eq:4.53}
 Y(t)^{\frac12}\le\sqrt{K_{\mathrm M}}\varepsilon(1+t)^{-2}.
\end{equation}
This comparison does not use decay of \(\Theta\) in an unweighted norm.

For the velocity bounds, \eqref{eq:4.36}, \eqref{eq:4.38},
and \eqref{eq:4.51}, with \(s=4,r=0\), imply
\begin{equation}\label{eq:4.54}
 \|\langle w\rangle_z\|_{L_x^2}+\sum_{i=1}^2\|\mathcal N_i\|_2
 \le C_KA\varepsilon^2(1+t)^{-2}
            +C_KA^{\frac32}\varepsilon^2(1+t)^{-\frac52}.
\end{equation}
Using \eqref{eq:4.53}--\eqref{eq:4.54} in the response bounds
\eqref{eq:3.2.3} and \eqref{eq:3.2.5}, together with
\eqref{eq:2.3.14}, we obtain
\begin{equation*}
 \sup_{t\le T}(1+t)^2\|w^0(t)\|_2
 \le C_K(1+\sqrt{K_{\mathrm M}})\varepsilon
     +C_K(A+A^{\frac32})\varepsilon^2,
\end{equation*}
and
\begin{equation*}
 \sup_{t\le T}(1+t)^{\frac32}\|v(t)\|_{\mathcal W^0}
 \le C_K(1+K_{\mathrm M}^{\frac38})\varepsilon
     +C_K(A+A^{\frac32})\varepsilon^2.
\end{equation*}
Enlarge \(C_K\) to cover both inequalities, and choose
\[
 A\ge4\max\left\{C_{\mathrm{in}}+\sqrt{K_{\mathrm M}},
                  C_K(1+\sqrt{K_{\mathrm M}}),
                  C_K(1+K_{\mathrm M}^{\frac38})\right\}.
\]
If \(C_K(A+A^{\frac32})\varepsilon\le\frac A4\), these bounds and
\eqref{eq:4.53} improve all three pointwise assumptions in
\eqref{eq:4.50}, with \(A\) replaced by \(\frac A2\).

Finally, \eqref{eq:4.40} and \eqref{eq:4.49} imply
\[
 J(T)\le C\bigl(1+K\varepsilon
                  +A^{\frac7{20}}K^{\frac{13}{40}}\bigr)\varepsilon.
\]
Absorb the equivalence constant in \eqref{eq:4.34} into \(C\)
and choose \(B\ge4C(1+K+A^{\frac7{20}}K^{\frac{13}{40}})\).
For \(\varepsilon\le1\), the bootstrap bound on \(J(T)\) then
improves to \(J(T)\le\frac B2\varepsilon\).
With the energy parameters and \(K,K_{\mathrm M},A,B\) fixed
in this order, choose \(\varepsilon_4\in(0,1]\) such that
\begin{equation*}
 \sqrt K\,\varepsilon_4\le\delta,
 \qquad
 C_EB\varepsilon_4\le\log2,
\end{equation*}
and
\begin{equation*}
 C_KA^3\varepsilon_4^2\le1,
 \qquad
 C_K(A+A^{\frac32})\varepsilon_4\le\frac{A}{4}.
\end{equation*}
A first-exit argument extends these improvements to every
\(T<T_{\max}^{(4)}\).  In particular,
\begin{equation}\label{eq:4.55}
 \|\partial_z\theta(t)\|_\infty\le\frac\gamma4,
 \qquad R(t)\ge\frac{3\gamma}{4},
 \qquad \|R(t)^{-1}\|_\infty\le\frac4{3\gamma}.
\end{equation}
The uniform \(H^4\) bound and \eqref{eq:4.55} exclude finite
lifespan by \eqref{eq:2.2.6}.  Hence \(T_{\max}^{(4)}=\infty\).
Taking \(\varepsilon_{\mathrm g}(4)=\varepsilon_4\) and
letting \(T\to\infty\) proves \eqref{eq:1.2.5}--\eqref{eq:1.2.6}
at order four, together with
\begin{equation}\label{eq:4.56}
\begin{aligned}
 &\|\mathcal Mv(t)\|_2+\|\mathcal M\Theta(t)\|_2+\|w^0(t)\|_2
 \le C\varepsilon(1+t)^{-2},\\
 &\|v(t)\|_{\mathcal W^0}
 \le C\varepsilon(1+t)^{-\frac32},\\
 &\|v\partial_xv\|_2+\|v\partial_x\Theta\|_2
 \le C\varepsilon^2(1+t)^{-\frac52}.
\end{aligned}
\end{equation}
These auxiliary bounds will be used in the endpoint decay proof.

Fix an integer \(s\ge5\) and return to
\(\varepsilon=\|v_0\|_{H^s}+\|\theta_0\|_{H^s}\).
For \(\varepsilon\le\varepsilon_4\), uniqueness in
Subsection~\ref{subsec:2.2} identifies the maximal \(H^s\) solution
with the global \(H^4\) solution on its lifespan.
It therefore inherits \eqref{eq:4.55} and
\(\int_0^T\mathcal C_v(t)\,\mathrm dt\le C_4\varepsilon\)
for every \(0<T<T_{\max}^{(s)}\).

At order \(s\), Proposition~\ref{prop:4.1} and Gronwall's inequality
yield, provided \(\sup_{t\le T}E_1(t)\le\frac{\delta_2(s)^2}{4}\),
\[
 E_*(T)+\int_0^T\|\mathcal M^{\frac12}\Theta\|_{H^s,\mathrm d}^2\,\mathrm dt
 \le C_s\varepsilon^2\exp\left(C_s\int_0^T\mathcal C_v\,\mathrm dt\right)
 \le C_s\varepsilon^2 e^{C_sC_4\varepsilon}.
\]
Choose \(K_s\) sufficiently large and then
\(\varepsilon_{\mathrm g}(s)\le\varepsilon_4\) sufficiently small
that \(K_s\varepsilon_{\mathrm g}(s)^2\le\frac{\delta_2(s)^2}{4}\)
and the preceding estimate improves
\(E_*(T)\le K_s\varepsilon^2\) to
\(E_*(T)\le\frac{K_s}{2}\varepsilon^2\).
Continuity, the stratification bound \eqref{eq:4.55}, and the
continuation criterion \eqref{eq:2.2.6} then imply
\(T_{\max}^{(s)}=\infty\).
Together with the fourth-order case, this yields, for every integer
\(s\ge4\),
\begin{equation}\label{eq:4.57}
 \sup_{t\ge0}E_1(t)+\int_0^\infty
 \bigl(\mathcal D_s(t)
       +\|\mathcal M^{\frac12}\Theta(t)\|_{H^s,\mathrm d}^2\bigr)
 \,\mathrm dt\le C_s\varepsilon^2.
\end{equation}
Together with the inherited velocity time-integral bound, this proves
\eqref{eq:1.2.4}--\eqref{eq:1.2.6}.  The local theory supplies
uniqueness and strong-solution regularity.

It remains to prove convergence and identify the limiting flux.
From \eqref{eq:2.2.3}, \eqref{eq:2.2.7}, \eqref{eq:2.3.6},
and \eqref{eq:4.5}, we have, for almost every \(t\ge0\),
\begin{equation*}
 \|w\|_{H^{s-1}}+\|\partial_t\Theta\|_{H^{s-1}}
 \le C\mathcal D_s^{\frac12},
\end{equation*}
and
\begin{equation*}
 \|\partial_tv\|_{H^{s-1}}
 \le C\left(\mathcal D_s^{\frac12}
       +\|\mathcal M^{\frac12}\Theta\|_{H^s,\mathrm d}\right).
\end{equation*}
These bounds, \eqref{eq:2.3.6}, and \eqref{eq:4.57} yield
\begin{equation}\label{eq:4.58}
 \int_0^\infty\bigl(\|v\|_{H^{s+1}}^2+\|w\|_{H^{s-1}}^2+\|\Theta\|_{H^{s-2}}^2
 +\|\mathcal T\Theta\|_{H^{s-1}}^2+\|\partial_t\Theta\|_{H^{s-1}}^2
 +\|\partial_tv\|_{H^{s-1}}^2\bigr)\,\mathrm dt\le C\varepsilon^2.
\end{equation}
Applying \eqref{eq:2.3.10} with \(f=\Theta\), \(n=s-2\),
and \(\ell=0\), we obtain
\begin{equation*}
 \|\Theta(t)\|_{H^{s-2}}^2
 \le C
 \left(\int_t^\infty\mathcal D_s(\tau)\,\mathrm d\tau\right)^{\frac12}
 \left(\int_t^\infty \|\mathcal M^{\frac12}\Theta(\tau)\|_{H^s,\mathrm d}^2\,\mathrm d\tau\right)^{\frac12}
 \longrightarrow0.
\end{equation*}
Since \(H^{s-2}(\Omega)\) is an algebra for \(s\ge4\),
\eqref{eq:2.2.10} implies
\[
 \|\langle w\Theta\rangle_x\|_{H_z^{s-2}}
 +\|\partial_t\langle\theta\rangle_x\|_{H_z^{s-3}}
 \le C\mathcal D_s^{\frac12}\|\mathcal M^{\frac12}\Theta\|_{H^s,\mathrm d}.
\]
Consequently,
\begin{equation*}
\begin{aligned}
 &\int_t^\infty
   \left(\|\langle w(\tau)\Theta(\tau)\rangle_x\|_{H_z^{s-2}}
        +\|\partial_\tau \langle\theta(\tau)\rangle_x\|_{H_z^{s-3}}\right)
       \,\mathrm d\tau\\
 \le{}&
 C\left(\int_t^\infty\mathcal D_s(\tau)\,\mathrm d\tau\right)^{\frac12}
  \left(\int_t^\infty \|\mathcal M^{\frac12}\Theta(\tau)\|_{H^s,\mathrm d}^2\,\mathrm d\tau\right)^{\frac12}.
\end{aligned}
\end{equation*}
The flux is therefore absolutely Bochner integrable in
\(H_z^{s-2}\).  Define
\begin{equation}\label{eq:4.59}
 S_\infty
:=\langle\theta_0\rangle_x-\partial_z\int_0^\infty
          \langle w(\tau)\Theta(\tau)\rangle_x\,\mathrm d\tau
 \quad\text{in }H_z^{s-3}.
\end{equation}
Integrating \eqref{eq:2.2.10}, we obtain the tail representation
\[
 \langle\theta(t)\rangle_x-S_\infty
 =\partial_z\int_t^\infty
   \langle w(\tau)\Theta(\tau)\rangle_x\,\mathrm d\tau,
\]
and, by the preceding flux estimate,
\begin{equation}\label{eq:4.60}
 \|\langle\theta(t)\rangle_x-S_\infty\|_{H_z^{s-3}}
 \le C
 \left(\int_t^\infty
       \mathcal D_s(\tau)\,\mathrm d\tau\right)^{\frac12}
 \left(\int_t^\infty
       \|\mathcal M^{\frac12}\Theta(\tau)\|_{H^s,\mathrm d}^2
       \,\mathrm d\tau\right)^{\frac12}.
\end{equation}
The right-hand side tends to zero as \(t\to\infty\).

Equation~\eqref{eq:4.59} proves \eqref{eq:1.2.9}, and
\eqref{eq:4.60} at \(t=0\) gives \eqref{eq:1.2.10}.
Periodicity implies \(\langle S_\infty\rangle_z=0\).
The uniform \(H_z^s\) bound and convergence in \eqref{eq:4.60}
imply, by weak lower semicontinuity, that \(S_\infty\in H_z^s\)
and
\begin{equation*}
 \|S_\infty\|_{H_z^s}\le C\varepsilon.
\end{equation*}
Since \(\theta-S_\infty=\Theta+\langle\theta\rangle_x-S_\infty\),
the difference tends to zero in \(H^{s-3}\) and stays bounded in
\(H^s\).  For \(0<\delta<3\), interpolation yields
\begin{equation*}
 \|\theta(t)-S_\infty\|_{H^{s-\delta}}
 \le C_\delta\|\theta(t)-S_\infty\|_{H^{s-3}}^{\frac\delta3}
   \bigl(\|\theta(t)\|_{H^s}
         +\|S_\infty\|_{H_z^s}\bigr)^{1-\frac\delta3}
 \longrightarrow0.
\end{equation*}
Sobolev inclusion covers \(3\le\delta\le s\).
In particular, the mean converges in \(H_z^2\hookrightarrow C^1\),
so averaging \eqref{eq:4.55} and passing to the limit yields
\begin{equation*}
 \gamma+\partial_zS_\infty(z)
   =\lim_{t\to\infty}\bigl(\gamma+\partial_z\langle\theta(t)\rangle_x(z)\bigr)
   \ge\frac\gamma2.
\end{equation*}
This proves \eqref{eq:1.2.7}.

Finally, apply \eqref{eq:2.3.10} to \(f=v\), \(n=s\), \(\ell=1\),
using \eqref{eq:4.58}.  Then
\begin{equation*}
 \|v(t)\|_{H^s}^2
 \le C\int_t^\infty
       \left(\mathcal D_s(\tau)+
       \|\mathcal M^{\frac12}\Theta(\tau)\|_{H^s,\mathrm d}^2\right)
             \,\mathrm d\tau\longrightarrow0.
\end{equation*}
The diagnostic formulas \eqref{eq:2.2.1} and \eqref{eq:2.2.2} imply
\[
 \|w(t)\|_{H^{s-1}}\le C(1+\|\theta(t)\|_{H^s})\|v(t)\|_{H^s}
       \longrightarrow0,
\]
and, with \(p_\infty=g\beta\partial_z^{-1}S_\infty\),
\[
 \|p(t)-p_\infty\|_{H^{s-\delta}}
 \le C_\delta\left(\|\theta(t)-S_\infty\|_{H^{s-\delta}}
                           +\|v(t)\|_{H^s}^2\right)\longrightarrow0
 \qquad(0<\delta\le s).
\]
The one-dimensional antiderivative bound ensures
\(p_\infty\in H_z^{s+1}\) with zero mean, proving
\eqref{eq:1.2.8} and \eqref{eq:1.2.11}.
Since \(\partial_zp_\infty=g\beta S_\infty\), the limit is a
compatible hydrostatic equilibrium, uniquely determined by the
strong convergence and pressure normalization.
\end{proof}

The fourth-order estimate \eqref{eq:1.2.6} cannot be extended to
a uniform bound in terms of the \(H^3\) norm of the initial data.
Proposition~\ref{prop:4.3} proves this already for \(w^0\), even
within the smooth data class generating global solutions.
Consequently, \(s=4\) is the smallest integer index for this
particular uniform estimate.  This does not exclude existence or
stability in \(H^3\) by a different estimate.

\begin{proposition}\label{prop:4.3}
There is no constant \(C\), depending only on the fixed parameters
and periods, such that every smooth compatible solution of the
linearized system \eqref{eq:3.1.1} satisfies
\begin{equation}\label{eq:4.61}
 \int_0^\infty\|w^{\mathrm{lin}}(t)\|_{W^{2,\infty}}\,\mathrm dt
 \le C\bigl(\|v_0\|_{H^3}+\|\theta_0\|_{H^3}\bigr),
 \qquad w^{\mathrm{lin}}=-\mathcal Tv^{\mathrm{lin}}.
\end{equation}
There are also no constants \(C,\delta>0\) for which the analogous
nonlinear bound with \(w^0\) holds uniformly for all smooth,
compatible, strictly stratified data of \(H^3\times H^3\) size
at most \(\delta\), even when restricted to data generating the
global solutions of Theorem~\ref{thm:1.1}.
\end{proposition}

\begin{proof}
Fix a sufficiently large integer \(j_0\), and define the frequency
blocks
\[
 \Gamma_j:=\left\{\left(\frac{2\pi n}{L_x},\frac{2\pi\ell}{L_z}\right):
               2^j\le n<2^{j+1},\quad 2^j\le\ell<2^{j+1}\right\}.
\]
Each \(\Gamma_j\) contains \(2^{2j}\) frequencies, with
\(a(k,m)\) comparable to \(2^{2j}\) and \(\rho(k,m)\) comparable
to \(2^{-2j}\).  Thus, after fixing \(j_0\) large enough,
\(\Gamma_j\subset\Sigma_2^x\), the roots are real and negative,
and \(|h_+(k,m)|\le C2^{-2j}\) for \(j\ge j_0\).  For a positive integer \(J\),
define
\begin{equation}\label{eq:4.62}
\begin{aligned}
 \theta_0^{(J)}(x,z):={}&\frac1{\sqrt J}
        \sum_{j=j_0}^{j_0+J-1}2^{-4j}
           \sum_{(k,m)\in\Gamma_j}\cos(kx+mz),\\
 v_0^{(J)}(x,z):={}&\frac1{\sqrt J}
        \sum_{j=j_0}^{j_0+J-1}2^{-4j}
           \sum_{(k,m)\in\Gamma_j}h_+(k,m)\cos(kx+mz).
\end{aligned}
\end{equation}
These data are real trigonometric polynomials with zero horizontal
and vertical means.  Orthogonality and the disjoint frequency
blocks give constants independent of \(J\) such that
\begin{equation}\label{eq:4.63}
 c\le\|\theta_0^{(J)}\|_{H^3}\le C,
 \qquad \|v_0^{(J)}\|_{H^3}\le C.
\end{equation}
Indeed, the squared temperature norm of each block is comparable
to \(J^{-1}2^{-8j}2^{6j}2^{2j}=J^{-1}\).
Every mode lies in the slow eigenspace, so
\[
 w^{\mathrm{lin},(J)}(t,x,z)
 =\frac1{\sqrt J}\sum_{j=j_0}^{j_0+J-1}2^{-4j}
       \sum_{(k,m)\in\Gamma_j}
       \frac{-\lambda_+(k,m)}\gamma e^{\lambda_+(k,m)t}
                                  \cos(kx+mz).
\]
At \((x,z)=(0,0)\), all terms in
\(\partial_{xx}w^{\mathrm{lin},(J)}\) have the same sign.
Since \(\|w^{\mathrm{lin},(J)}(t)\|_{W^{2,\infty}}
\ge|\partial_{xx}w^{\mathrm{lin},(J)}(t,0,0)|\), time integration gives
\begin{equation}\label{eq:4.64}
 \int_0^\infty\|w^{\mathrm{lin},(J)}(t)\|_{W^{2,\infty}}\,\mathrm dt
 \ge\frac1{\gamma\sqrt J}\sum_{j=j_0}^{j_0+J-1}2^{-4j}
                         \sum_{(k,m)\in\Gamma_j}k^2
 \ge c\sqrt J.
\end{equation}
This contradicts \eqref{eq:4.61} and
\eqref{eq:4.63} as \(J\to\infty\).

For the nonlinear assertion, suppose that such constants
\(C,\delta>0\) exist.  Fix \(J\) and a finite \(T>0\).
Rescale the data in \eqref{eq:4.62} by a factor \(\eta\in(0,1]\), chosen so that
\[
\begin{aligned}
 \eta\bigl(\|v_0^{(J)}\|_{H^6}+\|\theta_0^{(J)}\|_{H^6}\bigr)
 \le{}&\varepsilon_{\mathrm g}(6),\\
 \eta\bigl(\|v_0^{(J)}\|_{H^3}+\|\theta_0^{(J)}\|_{H^3}\bigr)
 \le{}&\delta,
\end{aligned}
\]
and \(\eta\|\partial_z\theta_0^{(J)}\|_\infty
\le\frac{\gamma}{2}\).
For fixed \(J\), these conditions hold for all sufficiently small
\(\eta>0\).  Theorem~\ref{thm:1.1} at \(s=6\) gives a global
solution \((v^\eta,\theta^\eta)\).  Since the initial data are
trigonometric polynomials, higher-order persistence in
Subsection~\ref{subsec:2.2} makes this solution spatially smooth
on every finite interval.  Moreover,
\[
 \sup_{t\ge0}
       \bigl(\|v^\eta(t)\|_{H^6}+\|\theta^\eta(t)\|_{H^6}\bigr)
 \le C_J\eta.
\]
The constant \(C_J\) may depend on the fixed polynomial data but
is independent of sufficiently small \(\eta\).  The diagnostic
formulas and the algebra property of \(H^5\) imply, for the same
range of amplitudes,
\[
 \sup_{t\ge0}\sum_{i=1}^2
       \|\mathcal N_i(v^\eta,\theta^\eta)(t)\|_{H^5}
 \le C_J\eta^2.
\]
Apply \eqref{eq:3.1.4} with \(F_i=\Pi_{\ne0}^z\mathcal N_i\),
divide by \(\eta\), and subtract the linear solution with the
unscaled data.  The resulting difference has zero initial data.
Equation~\eqref{eq:3.1.6} at order five and boundedness of the
vertical projection yield
\[
\begin{aligned}
 &\sup_{0\le t\le T}
 \left(\left\|\frac{v^\eta(t)}\eta-v^{\mathrm{lin},(J)}(t)\right\|_{H^5}
      +\left\|\frac{\theta^\eta(t)}\eta-\theta^{\mathrm{lin},(J)}(t)\right\|_{H^5}\right)\\
 \le{}&\frac C\eta\int_0^T\sum_{i=1}^2
       \|\mathcal N_i(v^\eta,\theta^\eta)(t)\|_{H^5}\,\mathrm dt
 \le C_JT\eta\longrightarrow0.
\end{aligned}
\]
Only fixed \(J\) and finite \(T\) are used in this limit.
In particular, no differentiability of a solution map in
\(H^3\times H^3\) is assumed.
The diagnostic multiplier loses at most one derivative.  Hence
\begin{equation}\label{eq:4.65}
 \frac{w^{0,\eta}}\eta\longrightarrow w^{\mathrm{lin},(J)}
 \quad\text{in }C([0,T];H^4)
                    \hookrightarrow C([0,T];W^{2,\infty}).
\end{equation}
Apply the assumed nonlinear estimate on \([0,T]\), divide
by \(\eta\), and let \(\eta\to0\) in
\eqref{eq:4.65}.  We obtain
\[
 \int_0^T\|w^{\mathrm{lin},(J)}(t)\|_{W^{2,\infty}}\,\mathrm dt
 \le C\bigl(\|v_0^{(J)}\|_{H^3}+\|\theta_0^{(J)}\|_{H^3}\bigr).
\]
Monotone convergence as \(T\to\infty\), followed by
\eqref{eq:4.63}--\eqref{eq:4.64}, gives a contradiction for
sufficiently large \(J\).  Higher norms restrict the admissible
amplitude for each fixed \(J\), but do not enter the assumed
uniform \(H^3\) constant.
\end{proof}

%============================================================================================================================================================
% End of Section 4
%============================================================================================================================================================

%=====================================================================
% Section 5. Proof of Theorem 1.2
%=====================================================================
\section{Terminal profile and quadratic expansion}\label{sec:5}
We establish the periodic inverse-map identities used to identify
the terminal profile and derive its quadratic expansion in
Theorem~\ref{thm:1.2}.
Proposition~\ref{prop:5.1} gives a variational characterization
of the terminal equilibrium.

For a real \(\vartheta\in C^1(\Omega)\) with
\(\langle\vartheta\rangle_z=0\) and \(\gamma+\partial_z\vartheta>0\),
let \(\Phi_\vartheta(x,z)=z+\gamma^{-1}\vartheta(x,z)\) and let
\(\Psi_\vartheta(x,\eta)\) be its inverse in the vertical variable.
Both are increasing \(C^1\) lifts, with
\begin{equation}\label{eq:5.1}
 \Phi_\vartheta(x,z+L_z)=\Phi_\vartheta(x,z)+L_z,
 \quad \Psi_\vartheta(x,\eta+L_z)=\Psi_\vartheta(x,\eta)+L_z,
 \quad \partial_\eta\Psi_\vartheta
 =\frac{\gamma}{(\gamma+\vartheta_z)(x,\Psi_\vartheta)}.
\end{equation}
Their displacements are periodic in both variables.  All norms in
\(\eta\) below are taken over one period, and \(L_{x,\eta}^2\) denotes
\(L^2((0,L_x)\times(0,L_z))\).  For any
\(C^1\) lift \(G\) with \(G'>0\) and \(G(z+L_z)=G(z)+L_z\),
\begin{equation}\label{eq:5.2}
 \int_0^{L_z}(G^{-1}(\eta)-\eta)\,\mathrm d\eta
 =-\int_0^{L_z}(G(z)-z)\,\mathrm dz.
\end{equation}

The inverse displacement is periodic, so its integral over
\([0,L_z]\) equals its integral over \([G(0),G(0)+L_z]\).
The substitution \(\eta=G(z)\) then gives
\[
 \int_0^{L_z}(G^{-1}(\eta)-\eta)\,\mathrm d\eta
 =-\int_0^{L_z}(G(z)-z)G'(z)\,\mathrm dz
 =-\int_0^{L_z}(G(z)-z)\,\mathrm dz,
\]
since \(\int_0^{L_z}(G-z)(G'-1)\,\mathrm dz=0\).  In particular,
\begin{equation}\label{eq:5.3}
 \int_0^{L_z}(\Psi_\vartheta(x,\eta)-\eta)\,\mathrm d\eta
 =-\frac1\gamma\int_0^{L_z}\vartheta(x,z)\,\mathrm dz=0.
\end{equation}
Set \(A_\vartheta(\eta)=\langle\Psi_\vartheta(\cdot,\eta)\rangle_x\).
For every continuous \(L_z\)-periodic function \(f\), change of
variables gives the distribution identity
\begin{equation}\label{eq:5.4}
 \langle f(\Phi_\vartheta)\rangle
 =\frac1{L_z}\int_0^{L_z}f(\eta)A_\vartheta'(\eta)\,\mathrm d\eta.
\end{equation}
Thus \(A_\vartheta\) is an increasing degree-one diffeomorphism
with zero-mean displacement.  Its inverse \(F=A_\vartheta^{-1}\)
is the unique increasing \(C^1\) lift satisfying
\begin{equation}\label{eq:5.5}
 F(z+L_z)=F(z)+L_z,\quad \int_0^{L_z}(F(z)-z)\,\mathrm dz=0,
 \quad
 \frac1{L_z}\int_0^{L_z}f(F(z))\,\mathrm dz
       =\langle f(\Phi_\vartheta)\rangle
\end{equation}
for every continuous \(L_z\)-periodic \(f\).  Existence follows from
\eqref{eq:5.2} and \eqref{eq:5.4}.  Let \(\widetilde F\) be another
such lift.  The measure in \eqref{eq:5.4} has a continuous positive
density, so \(\widetilde F\) has no constant interval.  Equality of
the measures of vertical arcs gives, for \(\eta_1<\eta_2<\eta_1+L_z\),
\[
 \widetilde F^{-1}(\eta_2)-\widetilde F^{-1}(\eta_1)
 =A_\vartheta(\eta_2)-A_\vartheta(\eta_1).
\]
Thus \(\widetilde F^{-1}-A_\vartheta\) is constant, and
\(\widetilde F^{-1}\) is \(C^1\) with positive derivative.
By \eqref{eq:5.2}, the constant is zero, proving uniqueness.
Likewise, two compatible labels with the same periodic distribution
have averaged inverses with the same derivative and the same
zero-mean displacement, so their averaged inverses coincide.
This fixes the translation in the periodic monotone rearrangement.
For the related bounded-channel setting, see
\cite[Section~2.4]{DalibardGuillodLeblond2025} and
\cite[Section~1.3]{Park2024}.

We also need the change-of-variables identity
\begin{equation}\label{eq:5.6}
 \frac1{\gamma^2}\int_0^{L_z}|\vartheta(x,z)|^2\,\mathrm dz
 =\int_0^{L_z}|\Psi_\vartheta(x,\eta)-\eta|^2\,\mathrm d\eta.
\end{equation}
Changing variables with \(z=\Psi_\vartheta(x,\eta)\) inserts the
Jacobian \(\partial_\eta\Psi_\vartheta\).  It can be removed because
\[
 \int_0^{L_z}|\Psi_\vartheta-\eta|^2
       (\partial_\eta\Psi_\vartheta-1)\,\mathrm d\eta
 =\frac13[(\Psi_\vartheta-\eta)^3]_0^{L_z}=0.
\]
For two \(C^1\) degree-one lifts \(G_i(z)=z+g_i(z)\) with
\(0<a\le G_i'\le b\), one further has
\begin{equation}\label{eq:5.7}
 \frac{a}{b^2}\|g_1-g_2\|_{L_z^2}^2
 \le\|G_1^{-1}-G_2^{-1}\|_{L_\eta^2}^2
 \le\frac{b}{a^2}\|g_1-g_2\|_{L_z^2}^2.
\end{equation}
Indeed, the mean value theorem gives
\[
 \frac{|g_1(z)-g_2(z)|}{b}
 \le |G_2^{-1}(G_1(z))-z|
 \le \frac{|g_1(z)-g_2(z)|}{a}.
\]
Integration with \(\mathrm d\eta=G_1'(z)\,\mathrm dz\) proves
\eqref{eq:5.7}.  The same constants apply after integration in \(x\).

The next lemma gives the inverse-map expansion used twice in the proof of Theorem~\ref{thm:1.2}.

\begin{lemma}\label{lem:5.1}
Let \(n\ge3\) be an integer.  There exist \(\delta,C>0\), depending
only on \(n\) and \(L_z\), such that, for every real-valued
\(f\in H_z^n\) with \(\|f\|_{H_z^n}\le\delta\), the lift
\(\mathrm{id}+f\) is an increasing \(C^1\) diffeomorphism of
\(\mathbb R\).  Its inverse is \(\mathrm{id}+h\), where
\(h\in H_z^n\) is \(L_z\)-periodic and satisfies
\begin{equation}\label{eq:5.8}
 \|h\|_{H_z^n}\le C\|f\|_{H_z^n},\qquad
 \|h'\|_{L_z^\infty}\le\frac13.
\end{equation}
Moreover,
\begin{equation}\label{eq:5.9}
 \|h+f\|_{H_z^{n-1}}\le C\|f\|_{H_z^n}^2,\qquad
 \|h+f-ff'\|_{H_z^{n-2}}\le C\|f\|_{H_z^n}^3.
\end{equation}
\end{lemma}

\begin{proof}
Choose \(\delta\le1\) so that \(\|f'\|_{L_z^\infty}\le\frac14\).
Since \(f\) is bounded and \(1+f'\ge\frac34\), the lift
\(\mathrm{id}+f\) is an increasing \(C^1\) diffeomorphism of
\(\mathbb R\).  The lift property implies that its inverse
displacement \(h\) is \(L_z\)-periodic.  The inverse identities are
\[
 h(z)+f(z+h(z))=0,\qquad
 1+h'(z)=\frac1{1+f'(z+h(z))},
\]
so \(\|h'\|_{L_z^\infty}\le\frac13\).
Periodicity and change of variables over one period give
\[
 \|h\|_{L_z^2}^2
 =\int_0^{L_z}|f(z)|^2(1+f'(z))\,\mathrm dz
 \le\frac54\|f\|_{L_z^2}^2.
\]
For smooth \(f\) and \(2\le j\le n\), the differentiated inverse
identity has coefficient \(1+f'(z+h)\ge\frac34\) in front of
\(h^{(j)}\).  The other highest-order term is
\(f^{(j)}(z+h)(1+h')^j\).  Every remaining term contains lower
derivatives of \(f\) and \(h\), including a derivative of \(h\)
of order at least two.  Induction, using
\(\|f^{(j)}\|_{L_z^\infty}\le C\|f\|_{H_z^n}\) for \(j<n\)
and change of variables for \(f^{(n)}(z+h)\), yields
\[
 \sum_{j=1}^{n-1}\|h^{(j)}\|_{L_z^\infty}
 +\|h^{(n)}\|_{L_z^2}
 \le C\|f\|_{H_z^n}.
\]
Together with the \(L_z^2\) bound, this proves \eqref{eq:5.8}
for smooth \(f\).

For \(0\le\tau\le1\), the derivative of \(\mathrm{id}+\tau h\)
lies between \(\frac23\) and \(\frac43\).  We claim that, for every
integer \(0\le j\le n\) and every \(u\in H_z^j\),
\begin{equation}\label{eq:5.10}
 \|u\circ(\mathrm{id}+\tau h)\|_{H_z^j}
 \le C\|u\|_{H_z^j},
\end{equation}
with \(C\) depending only on \(n\) and \(L_z\), uniformly in
\(\tau\).  For smooth \(u\), change of variables and the derivative
bounds prove the cases \(j=0,1\).  For \(2\le j\le n\), the terms
containing \(u^{(j)}\) and \(h^{(j)}\), respectively, satisfy
\[
 \|(u^{(j)}\circ(\mathrm{id}+\tau h))(1+\tau h')^j\|_{L_z^2}
 \le C\|u^{(j)}\|_{L_z^2},
\]
and
\[
 \|(u'\circ(\mathrm{id}+\tau h))\tau h^{(j)}\|_{L_z^2}
 \le C\|u'\|_{L_z^\infty}\|h^{(j)}\|_{L_z^2}
 \le C\|u\|_{H_z^j}.
\]
In every intermediate term, derivatives of \(h\) have order at
most \(j-1\) and are bounded in \(L_z^\infty\), while change of
variables controls the composed derivative of \(u\) in \(L_z^2\).
This proves \eqref{eq:5.10} for smooth \(f,u\).

Applying Taylor's formula to the inverse identity, we write
\[
 h+f=-h\int_0^1f'(z+\tau h)\,\mathrm d\tau.
\]
By \eqref{eq:5.8}, \eqref{eq:5.10}, and the algebra property
of \(H_z^{n-1}\),
\[
 \|h+f\|_{H_z^{n-1}}
 \le C\|h\|_{H_z^{n-1}}\|f'\|_{H_z^{n-1}}
 \le C\|f\|_{H_z^n}^2.
\]
Expanding to second order, we also have
\[
 h+f-ff'=-(h+f)f'
 -h^2\int_0^1(1-\tau)f''(z+\tau h)\,\mathrm d\tau.
\]
Using the preceding bound and \eqref{eq:5.10} in the algebra
\(H_z^{n-2}\), we obtain
\[
 \|h+f-ff'\|_{H_z^{n-2}}
 \le C\|h+f\|_{H_z^{n-1}}\|f\|_{H_z^n}
     +C\|h\|_{H_z^n}^2\|f\|_{H_z^n}
 \le C\|f\|_{H_z^n}^3.
\]
This proves \eqref{eq:5.9}, including \(n=3\), since \(H_z^1\)
is an algebra.

For general \(f\in H_z^n\), choose real smooth periodic
mollifications \(f_\ell\to f\) in \(H_z^n\) with
\(\|f_\ell\|_{H_z^n}\le\|f\|_{H_z^n}\), and let \(h_\ell\)
be their inverse displacements.  The uniform derivative lower
bound and the inverse identities imply \(h_\ell\to h\) in \(C^1\),
while the uniform Sobolev bound gives
\(h_\ell\rightharpoonup h\) in \(H_z^n\).
Since \(f_\ell f_\ell'\to ff'\) in \(H_z^{n-1}\), weak lower
semicontinuity and \(C^1\) convergence pass
\eqref{eq:5.8}--\eqref{eq:5.9} to the limit.
For smooth \(u\) and fixed \(\tau\),
\(u\circ(\mathrm{id}+\tau h_\ell)\) converges uniformly to
\(u\circ(\mathrm{id}+\tau h)\).  Its uniform \(H_z^j\) bound and
weak lower semicontinuity prove \eqref{eq:5.10} for the limiting
map, with the same constant for all \(\tau\in[0,1]\).
Density and the \(L_z^2\) change-of-variables bound extend the
estimate to every \(u\in H_z^j\).
\end{proof}

We now apply these identities to the transported temperature label. The normalized lift is kept fixed throughout the argument.

\begin{proof}[Proof of Theorem~\ref{thm:1.2}]
Let \((v,w,\theta,p)\) be the global solution of
Theorem~\ref{thm:1.1}, and set
\[
 \Phi(t,x,z)=z+\frac{\theta(t,x,z)}{\gamma},\qquad
 \Psi(t,x,\eta)=\Phi(t,x,\cdot)^{-1}(\eta).
\]
The initial maps are those in \eqref{eq:1.2.12}.
On every compact time interval, \eqref{eq:2.2.5} and Sobolev
embedding imply \(\theta\in C_tC^2\) and \(\theta_t\in C_tC^1\),
while \eqref{eq:1.2.4} ensures \(\Phi_z\ge\frac12\).
Thus \(\Psi\) is \(C^1\) in \((t,x,\eta)\), with the lift and
normalization properties \eqref{eq:5.1} and \eqref{eq:5.3}.
The periodic stream function for the full velocity is
\[
 \psi=\partial_z^{-1}v+\frac1\gamma\langle v\theta\rangle_z,
 \qquad (\partial_z\psi,-\partial_x\psi)=(v,w).
\]
The transport equation for \(\Phi\) and the inverse identity give,
at fixed \(\eta\),
\[
 \Psi_t=w(t,x,\Psi)-v(t,x,\Psi)\Psi_x
       =-\partial_x[\psi(t,x,\Psi(t,x,\eta))].
\]
Integrating in \(x\), we obtain
\begin{equation}\label{eq:5.11}
 \langle\Psi(t,\cdot,\eta)\rangle_x
 =\langle\Psi_0(\cdot,\eta)\rangle_x
 =\Psi_*(\eta),\qquad t\ge0,\quad\eta\in\mathbb R.
\end{equation}
For every continuous \(L_z\)-periodic function \(f\), the distribution
identity \eqref{eq:5.4} therefore reads
\[
 \langle f(\Phi(t))\rangle
 =\frac1{L_z}\int_0^{L_z}f(\eta)\Psi_*'(\eta)\,\mathrm d\eta
 =\langle f(\Phi_0)\rangle.
\]

We next identify the terminal profile.  Differentiation of the
initial inverse map gives
\begin{equation}\label{eq:5.12}
 \Psi_*'(\eta)=\frac1{L_x}\int_0^{L_x}
              \frac{\gamma}{R(0,x,\Psi_0(x,\eta))}\,\mathrm dx,
 \qquad
 \frac{\gamma}{\max_\Omega R(0)}\le\Psi_*'\le
              \frac{\gamma}{\min_\Omega R(0)}.
\end{equation}
Together with \eqref{eq:5.1}--\eqref{eq:5.3}, these bounds show
that \(\Psi_*\) and its inverse are increasing normalized lifts.
Let \(\Phi_\infty(z)=z+\gamma^{-1}S_\infty(z)\) and
\(\Psi_\infty=\Phi_\infty^{-1}\).
By \eqref{eq:1.2.11}, \(\theta(t)\to S_\infty\) in
\(H^3(\Omega)\hookrightarrow C^1(\Omega)\).
The derivative lower bounds in \eqref{eq:1.2.4} and
\eqref{eq:1.2.7} imply
\begin{equation*}
 \|\Psi(t)-\Psi_\infty\|_{L^\infty_{x,\eta}}
 \le\frac2\gamma\|\theta(t)-S_\infty\|_{L^\infty(\Omega)}
 \longrightarrow0.
\end{equation*}
Passing to the limit in \eqref{eq:5.11} identifies the fixed lifts
\(\Psi_\infty=\Psi_*\).  Inversion proves \eqref{eq:1.2.13},
and \eqref{eq:5.5} supplies the rearrangement characterization.
Using \eqref{eq:5.12}, we find
\[
 \gamma+S_\infty'(z)
 =\left[\frac1{L_x}\int_0^{L_x}
 \frac{\mathrm dx}{R(0,x,\Psi_0(x,\Psi_*^{-1}(z)))}\right]^{-1},
\]
which proves \eqref{eq:1.2.14}--\eqref{eq:1.2.15}.
This is the profile in \eqref{eq:4.59}, where the flux is averaged
at fixed \(z\), rather than at fixed \(\eta\).

For the quadratic correction, we apply Lemma~\ref{lem:5.1} twice.
The slice estimate
\begin{equation*}
 \sup_{x\in\mathbb T_x}\|\theta_0(x,\cdot)\|_{H_z^{s-1}}
 \le C\|\theta_0\|_{H^1(\mathbb T_x;H_z^{s-1})}
 \le C\|\theta_0\|_{H^s(\Omega)}
\end{equation*}
permits the first application to
\(f=\gamma^{-1}\theta_0(x,\cdot)\), with \(n=s-1\ge3\),
uniformly in \(x\), after reducing \(\varepsilon_{\mathrm g}(s)\).
Set \(h_0(x,\eta)=\Psi_0(x,\eta)-\eta\) and
\(h_*(\eta)=\langle h_0(\cdot,\eta)\rangle_x\).
Vertical Fourier truncations, whose coefficients are continuous in
\(x\), establish strong measurability in \(H_z^{s-1}\), and the bound
\eqref{eq:5.8} justifies this Bochner average.
Averaging \eqref{eq:5.8}--\eqref{eq:5.9} at fixed \(\eta\), we obtain
\begin{equation}\label{eq:5.13}
\begin{aligned}
 &\sup_x\|h_0(x,\cdot)\|_{H_z^{s-1}}+\|h_*\|_{H_z^{s-1}} \le C\|\theta_0\|_{H^s},\\
 &\left\|h_*+\frac{\langle\theta_0\rangle_x}{\gamma}\right\|_{H_z^{s-2}} \le C\|\theta_0\|_{H^s}^2,
\end{aligned}
\end{equation}
and
\begin{equation}\label{eq:5.14}
 \left\|h_*+\frac{\langle\theta_0\rangle_x}{\gamma}-\frac1{2\gamma^2}\partial_z\langle\theta_0^2\rangle_x\right\|_{H_z^{s-3}}
 \le C\|\theta_0\|_{H^s}^3.
\end{equation}
Here the initial profiles and their vertical derivatives are evaluated at \(\eta\), and the quadratic term uses
\(\langle\theta_0\partial_z\theta_0\rangle_x=\frac12\partial_z\langle\theta_0^2\rangle_x\).

By \eqref{eq:5.13}, a further reduction of the common threshold
\(\varepsilon_{\mathrm g}(s)\) permits the second application to
\(f=h_*\), again with \(n=s-1\).
Since \(\Psi_*^{-1}-\mathrm{id}=\gamma^{-1}S_\infty\),
\eqref{eq:5.9} gives, now with all profiles evaluated at \(z\),
\begin{equation}\label{eq:5.15}
 \left\|\frac{S_\infty}{\gamma}+h_*-h_*h_*'\right\|_{H_z^{s-3}}
 \le C\|h_*\|_{H_z^{s-1}}^3
 \le C\|\theta_0\|_{H^s}^3.
\end{equation}
The product difference can be written as
\[
 h_*h_*'-\frac{\langle\theta_0\rangle_x\langle\partial_z\theta_0\rangle_x}{\gamma^2}
 =\left(h_*+\frac{\langle\theta_0\rangle_x}{\gamma}\right)h_*'
 -\frac{\langle\theta_0\rangle_x}{\gamma}
       \left(h_*'+\frac{\langle\partial_z\theta_0\rangle_x}{\gamma}\right).
\]
Using \eqref{eq:5.13} in the algebra \(H_z^{s-3}\), we deduce
\begin{equation}\label{eq:5.16}
 \left\|h_*h_*'-\frac{\langle\theta_0\rangle_x\langle\partial_z\theta_0\rangle_x}{\gamma^2}\right\|_{H_z^{s-3}}
 \le C\left\|h_*+\frac{\langle\theta_0\rangle_x}\gamma\right\|_{H_z^{s-2}}
       \left(\|h_*\|_{H_z^{s-1}}+\|\langle\theta_0\rangle_x\|_{H_z^{s-1}}\right)
 \le C\|\theta_0\|_{H^s}^3.
\end{equation}
Since
\[
 \partial_z\langle\theta_0^2\rangle_x
 -2\langle\theta_0\rangle_x\langle\partial_z\theta_0\rangle_x
 =\partial_z\langle\Theta_0^2\rangle_x,
\]
estimates \eqref{eq:5.14}--\eqref{eq:5.16} prove the temperature
remainder in \eqref{eq:1.2.17}.
The correction satisfies \eqref{eq:1.2.16}, because
\[
 \|\partial_z\langle\Theta_0^2\rangle_x\|_{H_z^{s-1}}
 \le C\|\Theta_0^2\|_{H^s}\le C\|\theta_0\|_{H^s}^2.
\]
Both the correction and the temperature remainder have zero vertical
mean.  Moreover,
\[
 \partial_z^{-1}\partial_z\langle\Theta_0^2\rangle_x
 =\langle\Theta_0^2\rangle_x-\langle\Theta_0^2\rangle.
\]
Thus \(p_\infty=g\beta\partial_z^{-1}S_\infty\) and
\eqref{eq:2.1.3} establish the cubic pressure remainder in
\(H_z^{s-2}\), with the prescribed zero-mean normalization.
At \(s=4\), the product estimate above uses the algebra \(H_z^1\),
and the temperature and pressure remainders belong to \(H_z^1\)
and \(H_z^2\), respectively.

Finally, let \((v^{\mathrm{lin}},\Theta^{\mathrm{lin}})\) solve
\eqref{eq:3.1.1} with initial data
\((\Pi_{\ne0}^xv_0,\Theta_0)\), and set
\(w^{\mathrm{lin}}=-\mathcal Tv^{\mathrm{lin}}\).
For \(|\alpha|\le s-2\), the multiplier bounds ensure
\[
 \|\partial^\alpha v_0\|_2+\|\mathcal T\partial^\alpha v_0\|_2
 +\|\partial^\alpha\Theta_0\|_{H_z^2(\Omega)}
 \le C\bigl(\|v_0\|_{H^{s-1}}+\|\theta_0\|_{H^s}\bigr).
\]
Applying Proposition~\ref{prop:3.1}\textup{(ii)} to the
differentiated, horizontally projected solution and summing over
\(\alpha\), we obtain
\begin{equation}\label{eq:5.17}
 \|w^{\mathrm{lin}}(t)\|_{H^{s-2}}
 \le C\varepsilon(1+t)^{-1},
 \qquad
 \|\Theta^{\mathrm{lin}}(t)\|_{H^{s-2}}
 \le C\varepsilon(1+t)^{-\frac12}.
\end{equation}
The algebra property of \(H^{s-2}(\Omega)\), with \(s-2\ge2\),
therefore implies
\[
 \int_0^\infty\|\langle w^{\mathrm{lin}}\Theta^{\mathrm{lin}}\rangle_x
                   \|_{H_z^{s-2}}\,\mathrm dt
 \le C\varepsilon^2\int_0^\infty(1+t)^{-\frac32}\,\mathrm dt<\infty.
\]
On every finite interval, Lemma~\ref{lem:3.1} and
\(\partial_t\Theta^{\mathrm{lin}}=-\gamma w^{\mathrm{lin}}\)
ensure \(\Theta^{\mathrm{lin}}\in C^1([0,T];H^{s-2})\).
The product rule and horizontal averaging give, in \(H_z^{s-2}\),
\[
 \int_0^T\langle w^{\mathrm{lin}}\Theta^{\mathrm{lin}}\rangle_x\,\mathrm dt
 =\frac1{2\gamma}\left(\langle\Theta_0^2\rangle_x
                   -\langle(\Theta^{\mathrm{lin}}(T))^2\rangle_x\right).
\]
By \eqref{eq:5.17}, the last term tends to zero in \(H_z^{s-2}\).
Letting \(T\to\infty\) in that space and applying
\(\partial_z:H_z^{s-2}\to H_z^{s-3}\) proves
\eqref{eq:1.2.18} and completes the proof.
\end{proof}

We next characterize the terminal state as an energy minimizer under
the periodic label-distribution constraint.  No smallness assumption
is imposed on the competing states.
\begin{proposition}\label{prop:5.1}
Let \(s\ge4\) be an integer, and let \(\Phi_0,\Psi_*\) and
\(S_\infty\) be as in Theorem~\ref{thm:1.2}.
Let \((\widetilde v,\vartheta)\in H^s(\Omega)\times H^s(\Omega)\)
be real-valued and satisfy
\[
 \langle\widetilde v\rangle_z=\langle\vartheta\rangle_z=0,
 \qquad \gamma+\partial_z\vartheta>0.
\]
Suppose that \(\Phi_\vartheta=z+\gamma^{-1}\vartheta\) and \(\Phi_0\)
have the same distribution modulo \(L_z\), so that for every
continuous \(L_z\)-periodic function \(f\),
\[
 \langle f(\Phi_\vartheta)\rangle=\langle f(\Phi_0)\rangle.
\]
Then
\begin{equation}\label{eq:5.18}
 \|\vartheta\|_2^2-\|S_\infty\|_2^2
 =\gamma^2\int_0^{L_x}\int_0^{L_z}
       |\Psi_\vartheta(x,\eta)-\Psi_*(\eta)|^2\,\mathrm d\eta\,\mathrm dx,
\end{equation}
and
\begin{equation}\label{eq:5.19}
 \frac{1}{2}\|\widetilde v\|_2^2
 +\frac{g\beta}{2\gamma}\|\vartheta\|_2^2
 \ge\frac{g\beta}{2\gamma}\|S_\infty\|_2^2.
\end{equation}
Equality in \eqref{eq:5.19} holds if and only if
\((\widetilde v,\vartheta)=(0,S_\infty)\).
Thus \((0,S_\infty)\) is the unique minimizer of the quadratic
perturbation energy in this class.
\end{proposition}
\begin{proof}
Theorems~\ref{thm:1.1} and \ref{thm:1.2}, together with
\eqref{eq:5.5}, show that \((0,S_\infty)\) is admissible and
that its inverse label is \(\Psi_*\).
For any competing state, the distribution identity \eqref{eq:5.4}
implies \(A_\vartheta'=\Psi_*'\).
The zero-mean normalization \eqref{eq:5.3} fixes the additive
constant, so for every \(\eta\in\mathbb R\),
\[
 \langle\Psi_\vartheta(\cdot,\eta)\rangle_x
 =A_\vartheta(\eta)=\Psi_*(\eta).
\]
Apply \eqref{eq:5.6} to \(\vartheta\) and \(S_\infty\), extending
functions of \(\eta\) constantly in \(x\).
Horizontal orthogonality then gives
\[
\begin{aligned}
 \frac{\|\vartheta\|_2^2-\|S_\infty\|_2^2}{\gamma^2}
 ={}&\|\Psi_\vartheta-\mathrm{id}\|_{L_{x,\eta}^2}^2
       -\|\Psi_*-\mathrm{id}\|_{L_{x,\eta}^2}^2\\
 ={}&\|\Psi_\vartheta-\Psi_*\|_{L_{x,\eta}^2}^2,
\end{aligned}
\]
which proves \eqref{eq:5.18}.
Multiplication by \(\frac{g\beta}{2\gamma}\) and addition of the
kinetic energy yield \eqref{eq:5.19}.
Since \(g,\beta,\gamma>0\), equality holds precisely when
\(\widetilde v=0\) and \(\Psi_\vartheta=\Psi_*\) almost everywhere.
By continuity and inversion, the latter condition is equivalent to
\(\vartheta=S_\infty\), completing the proof.
\end{proof}

%=====================================================================
% End of Section 5
%=====================================================================

%=====================================================================
% Section 6. Proof of Theorem 1.3
%=====================================================================
\section{Sharp algebraic decay}\label{sec:6}
We prove the decay estimates in Theorem~\ref{thm:1.3} using fluctuation energy estimates, spectral bounds, and forced response estimates.
The inverse-map identities of Section~\ref{sec:5} transfer fluctuation decay to the temperature difference from equilibrium,
while the diagnostic pressure formula controls the pressure difference.
Exact nonlinear solutions in Proposition~\ref{prop:6.1} provide matching lower bounds.
For related spectral and energy arguments for the Boussinesq system, see \cite[Section~6]{JangKim2023}.

\begin{proof}[Proof of the upper bounds in Theorem~\ref{thm:1.3}]
The zero datum is immediate.  For \(\varepsilon>0\), let
\((v,w,\theta,p)\) be the global solution of Theorem~\ref{thm:1.1}.
We use the bounds
\begin{equation}\label{eq:6.1}
 \sup_{t\ge0}\bigl(\|v(t)\|_{H^s}+\|\theta(t)\|_{H^s}\bigr)
 +\|S_\infty\|_{H_z^s}\le C_0\varepsilon,
 \qquad
 \int_0^\infty\mathcal C_v(t)\,\mathrm dt\le C_0\varepsilon.
\end{equation}
Choose \(\varepsilon_{\mathrm d}\le\min\{1,\varepsilon_{\mathrm g}\}\)
small enough that \(E_1^{\frac12}\le\min\{1,\delta_{\mathrm M}\}\)
and \(\|\theta_z\|_\infty\le\frac\gamma2\) for all time.
Proposition~\ref{prop:4.2} then applies on every finite interval.
The threshold will be reduced below after fixing the bootstrap
constants.

We first derive an energy inequality for the horizontal fluctuation.
Write \(S=\langle\theta\rangle_x\), \(\sigma=\gamma+S'\), and
\(q_0=\sigma^{-1}\).  The preceding smallness condition gives
\(\frac\gamma2\le\sigma\le\frac{3\gamma}{2}\).
For a sufficiently small fixed \(\lambda_{\mathrm d}>0\), set
\[
 E_4=\|v\|_2^2+g\beta\int_\Omega q_0\Theta^2
       +\lambda_{\mathrm d}
          \langle\mathcal L^{-1}\mathcal Tv,\Theta\rangle_{L^2}.
\]
Since \(\mathcal L^{-1}\mathcal T\) is bounded on \(L^2\),
\begin{equation}\label{eq:6.2}
 c\bigl(\|v\|_2^2+\|\Theta\|_2^2\bigr)
 \le E_4\le C\bigl(\|v\|_2^2+\|\Theta\|_2^2\bigr).
\end{equation}
On each finite interval, \eqref{eq:2.2.5} gives
\(v,\Theta\in H_t^1L^2\), \(q_0\in C_tW_z^{1,\infty}\), and
\(q_{0,t}\in C_tL_z^\infty\), so the weighted energy is absolutely
continuous.  The fluctuation equation is
\[
 \Theta_t+v\Theta_x+w\Theta_z+\sigma w=-S_t.
\]
Since \(\langle\Theta\rangle_x=\langle w\rangle_x=0\),
\[
 \int_\Omega q_0\Theta S_t=0,
 \qquad
 \int_\Omega vp_x=-g\beta\int_\Omega w\Theta.
\]
Pairing the fluctuation equation with \(g\beta q_0\Theta\) and
adding the momentum energy identity yields
\begin{equation}\label{eq:6.3}
 \frac{\mathrm d}{\mathrm dt}
  \left(\|v\|_2^2+g\beta\int_\Omega q_0\Theta^2\right)
 +2\|\mathcal L^{\frac12}v\|_2^2
 =g\beta\int_\Omega(q_{0,t}+wq_{0,z})\Theta^2.
\end{equation}
By the mean equation \eqref{eq:2.2.10},
\[
 q_{0,t}+wq_{0,z}
 =\frac{\partial_{zz}\langle w\Theta\rangle_x-wS''}{\sigma^2}.
\]
For almost every time, Sobolev embedding therefore implies
\[
 \|q_{0,t}+wq_{0,z}\|_\infty
 \le  \|w\|_{W^{2,\infty}}
        \bigl(\|\Theta\|_{W^{2,\infty}}+\|S\|_{W_z^{2,\infty}}\bigr)
 \le E_1^{\frac12}\mathcal C_v.
\]

The cross term supplies the temperature dissipation.  The diagnostic
formula and Poincar\'e's inequality give
\(\|w^0\|_2\le C\|\mathcal L^{\frac12}v\|_2\) and
\[
 \|\langle w\rangle_z\|_{L_x^2}
 \le C\bigl(\|v_x\|_2\|\theta\|_\infty
             +\|v\|_2\|\theta_x\|_\infty\bigr)
 \le CE_1^{\frac12}\|\mathcal L^{\frac12}v\|_2.
\]
Thus \(\|w\|_2\le C\|\mathcal L^{\frac12}v\|_2\).
Writing the nonlinearities in divergence form, we obtain
\[
 \sum_{i=1}^2\|\mathcal N_i\|_{H^{-1}}
 \le C\bigl(\|v\|_\infty+\|\theta\|_\infty\bigr)
           (\|v\|_2+\|w\|_2)
 \le CE_1^{\frac12}\|\mathcal L^{\frac12}v\|_2.
\]
Apply \eqref{eq:2.3.12} at order zero to \((v,\Theta)\), with
sources \(\Pi_{\ne0}^z\mathcal N_1\) and \(\mathcal Q\mathcal N_2\).
By \eqref{eq:2.3.3} and \eqref{eq:2.3.5}, the source pairings satisfy
\[
 |\langle\Pi_{\ne0}^z\mathcal N_1,
                 \mathcal L^{-1}\mathcal T\Theta\rangle|
 \le C\|\mathcal N_1\|_{H^{-1}}\|\mathcal M^{\frac12}\Theta\|_2
 \le CE_1^{\frac12}\|\mathcal L^{\frac12}v\|_2
                         \|\mathcal M^{\frac12}\Theta\|_2,
\]
and
\[
 |\langle\mathcal Q\mathcal N_2,
                 \mathcal L^{-1}\mathcal Tv\rangle|
 \le C\|\mathcal N_2\|_{H^{-1}}\|v\|_2
 \le CE_1^{\frac12}\|\mathcal L^{\frac12}v\|_2^2.
\]
The linear terms are controlled by \eqref{eq:2.3.8} and
Poincar\'e's inequality.
Applying Young's inequality with \(E_1^{\frac12}\le1\), we find
\begin{equation}\label{eq:6.4}
 \frac{\mathrm d}{\mathrm dt}
   \langle\mathcal L^{-1}\mathcal Tv,\Theta\rangle_{L^2}
 +\frac1{2\gamma}\|\mathcal M^{\frac12}\Theta\|_2^2
 \le C\|\mathcal L^{\frac12}v\|_2^2.
\end{equation}
Taking \(\lambda_{\mathrm d}\) smaller if necessary, we add
\(\lambda_{\mathrm d}\) times \eqref{eq:6.4} to \eqref{eq:6.3}.
By \eqref{eq:6.2}, the resulting inequality is
\begin{equation}\label{eq:6.5}
 E_4'+c\bigl(\|\mathcal L^{\frac12}v\|_2^2
                  +\|\mathcal M^{\frac12}\Theta\|_2^2\bigr)
 \le CE_1^{\frac12}\mathcal C_vE_4.
\end{equation}
The dual bounds above justify \eqref{eq:2.3.12} at strong-solution
regularity.  Hence \(E_4\in W^{1,1}(0,T)\), and \eqref{eq:6.5}
holds almost everywhere on each finite interval.

To obtain decay, we use the symbol bound \eqref{eq:2.3.1} and
Poincar\'e's inequality in the form
\[
 \|\Theta\|_{H^{-2}}\le C\|\mathcal M^{\frac12}\Theta\|_2,
 \qquad
 \|v\|_{H^{-2}}\le C\|\mathcal L^{\frac12}v\|_2.
\]
For \(f=v,\Theta\), Fourier interpolation yields
\[
 \|f\|_2^2\le\|f\|_{H^{-2}}^{\frac{2s}{s+2}}
                   \|f\|_{H^s}^{\frac4{s+2}}.
\]
Together with \eqref{eq:6.1} and \eqref{eq:6.2}, this implies
\[
 E_4\le C\varepsilon^{\frac4{s+2}}
       \bigl(\|\mathcal L^{\frac12}v\|_2^2
                   +\|\mathcal M^{\frac12}\Theta\|_2^2\bigr)^{\frac{s}{s+2}}.
\]
It follows from \eqref{eq:6.5} that
\begin{equation}\label{eq:6.6}
 E_4'+c\varepsilon^{-\frac4s}E_4^{1+\frac2s}
 \le C\varepsilon\mathcal C_v(t)E_4.
\end{equation}
Set \(I(t)=C\varepsilon\int_0^t\mathcal C_v(\tau)\,\mathrm d\tau\) and \(E_6=e^{-I}E_4\).
By \eqref{eq:6.1} and \eqref{eq:6.6}, we have
\[
 0\le I(t)\le C\varepsilon^2,\qquad
 E_6'+c\varepsilon^{-\frac4s}e^{\frac{2I}{s}}E_6^{1+\frac2s}\le0.
\]
The scalar comparison \eqref{eq:2.3.15}--\eqref{eq:2.3.16},
with \(\alpha=\frac s2\) and zero forcing, gives
\[
E_6(t)\le C\varepsilon^2(1+t)^{-\frac s2},
\]
since
\(E_6(0)\le C\varepsilon^2\) and \(e^{\frac{2I}{s}}\ge1\).
By \eqref{eq:6.1}--\eqref{eq:6.2} and interpolation, for
\(0\le q\le s\),
\begin{equation}\label{eq:6.7}
 \|v(t)\|_{H^q}+\|\Theta(t)\|_{H^q}
 \le C_q\varepsilon(1+t)^{-\frac{s-q}{4}}.
\end{equation}
We recover the full temperature difference from the inverse-map
identities.  Fix \(t\ge0\), and use the normalized lifts
\(\Phi(t),\Phi_\infty\) and inverse \(\Psi(t)\) from
Section~\ref{sec:5}.
By \eqref{eq:4.55} and \eqref{eq:1.2.15}, the vertical derivatives of
\(\Phi(t)\), \(z+\gamma^{-1}S(t,z)\), and \(\Phi_\infty\) lie between
\(\frac12\) and \(\frac32\).  Let \(\Psi_{\mathrm{av}}\) be the
inverse of \(z+\gamma^{-1}S(t,z)\) at this fixed time.
The conservation law \eqref{eq:5.11} gives
\(\langle\Psi\rangle_x=\Psi_*\), so horizontal orthogonality yields
\[
 \|\Psi-\Psi_{\mathrm{av}}\|_{L_{x,\eta}^2}^2
 =\|\Psi-\Psi_*\|_{L_{x,\eta}^2}^2
  +L_x\|\Psi_*-\Psi_{\mathrm{av}}\|_{L_\eta^2}^2.
\]
Apply \eqref{eq:5.7} with \(a=\frac12\), \(b=\frac32\), first to
\(\Phi(t),\Phi_\infty\) and then to
\(\Phi(t),z+\gamma^{-1}S(t,z)\).  Since
\(\Phi_\infty^{-1}=\Psi_*\), we obtain
\[
 \|\theta-S_\infty\|_2^2
 \le C\gamma^2\|\Psi-\Psi_*\|_{L_{x,\eta}^2}^2
 \le C\gamma^2\|\Psi-\Psi_{\mathrm{av}}\|_{L_{x,\eta}^2}^2
 \le C\|\Theta\|_2^2.
\]
The reverse bound follows from horizontal orthogonality.
Together with \eqref{eq:6.7}, this proves
\[
 \|\Theta(t)\|_2\le\|\theta(t)-S_\infty\|_2
 \le C\|\Theta(t)\|_2
 \le C\varepsilon(1+t)^{-\frac s4}.
\]
For \(0\le r<s\), interpolation with \eqref{eq:6.1} gives
\[
 \|\theta(t)-S_\infty\|_{H^r}
 \le\|\theta(t)-S_\infty\|_2^{1-\frac rs}
       \|\theta(t)-S_\infty\|_{H^s}^{\frac rs}
 \le C_r\varepsilon(1+t)^{-\frac{s-r}{4}}.
\]
This and \eqref{eq:6.7} prove \eqref{eq:1.2.19}.

For the componentwise estimates, Lemma~\ref{lem:3.3} applies with
\(F_i=\Pi_{\ne0}^z\mathcal N_i\), as in
Subsection~\ref{subsec:4.2}.  The projections are contractions in
the source norms.
At \(s=4\), only \(r=0\) is admissible for the componentwise
estimates.  It follows from \eqref{eq:4.36}, \eqref{eq:4.38},
and \eqref{eq:4.56} that
\begin{equation}\label{eq:6.8}
 \|\langle w(t)\rangle_z\|_{L_x^2}
       +\sum_{i=1}^2\|\mathcal N_i(t)\|_2
 \le C\varepsilon^2(1+t)^{-2}.
\end{equation}
Combining this estimate with the bound for \(w^0\) in
\eqref{eq:4.56} and using \(w=w^0+\langle w\rangle_z\),
we obtain \(\|w(t)\|_2\le C\varepsilon(1+t)^{-2}\).
For the horizontal velocity, interpolation between \eqref{eq:6.7}
and \eqref{eq:4.56} gives
\[
 \|\mathcal M^{\frac34}\Theta(t)\|_2
 \le\|\Theta(t)\|_2^{\frac14}
          \|\mathcal M\Theta(t)\|_2^{\frac34}
 \le C\varepsilon(1+t)^{-\frac74}.
\]
Substituting this bound and \eqref{eq:6.8} into \eqref{eq:3.2.4},
we use \eqref{eq:2.3.14} to get
\[
 \|v(t)\|_2\le C\varepsilon(1+t)^{-\frac74}
                     +C\varepsilon^2(1+t)^{-2}
             \le C\varepsilon(1+t)^{-\frac74}.
\]
This proves \eqref{eq:1.2.20}--\eqref{eq:1.2.21} at \(s=4\).
The pressure will be estimated for all admissible indices below.

We now assume \(s\ge5\).  To estimate the horizontal transport,
we use, for \(r\ge0\) and \(\mathcal Qf=f\),
\begin{equation}\label{eq:6.9}
 \|\partial_xf\|_{\mathcal W^r}
 \le C\|\mathcal Mf\|_{H^r}^{\frac13}
                \|f\|_{H^{r+5}}^{\frac23}.
\end{equation}
The constant is independent of \(r\), since
\[
 \sum_{\mathcal K_*^x}
 \frac{k^2}{\rho^{\frac23}\langle k,m\rangle^{\frac{20}{3}}}
 \le C\sum_{\mathcal K_*^x}\langle k,m\rangle^{-\frac{10}{3}}<\infty.
\]
Cauchy--Schwarz bounds the left-hand side of \eqref{eq:6.9} by
\(C\|\mathcal M^{\frac13}f\|_{H^{r+\frac{10}{3}}}\).
Interpolation between \(\|\mathcal Mf\|_{H^r}\) and
\(\|f\|_{H^{r+5}}\) proves the estimate.

It suffices to close a joint bootstrap at \(r=0,s-5\), with a single index when \(s=5\).  Interpolation will cover the remaining orders.
For \(A\ge1\) to be chosen below, we impose the following bootstrap bounds at these indices for all \(t\in[0,T]\):
\begin{equation}\label{eq:6.10}
\left\{
\begin{aligned}
 &E_{5,r}(t)^{\frac12}
 \le A\varepsilon(1+t)^{-\frac{s-r+4}{4}},\\
 &\|v(t)\|_{H^r}
 \le A\varepsilon(1+t)^{-\frac{s-r+3}{4}},\\
 &\|w^0(t)\|_{H^r}
 \le A\varepsilon(1+t)^{-\frac{s-r+4}{4}}.
\end{aligned}
\right.
\end{equation}
Strong-solution regularity ensures continuity of the quantities in
\eqref{eq:6.10}.  Their initial bounds follow from \eqref{eq:4.34}
and the diagnostic estimate
\(\|w^0(0)\|_{H^r}\le C\|v_0\|_{H^{r+1}}\).
Thus sufficiently large \(A\) starts the strict bootstrap.

For \(f=\mathcal Qv\) or \(f=\Theta\),
\eqref{eq:6.9}, \eqref{eq:6.7}, and \eqref{eq:6.10} imply
\[
 \|\partial_xf(t)\|_{\mathcal W^r}
 \le C\|\mathcal Mf(t)\|_{H^r}^{\frac13}
       \|f(t)\|_{H^{r+5}}^{\frac23}
 \le C_rA^{\frac13}\varepsilon(1+t)^{-\frac{s-r-2}{4}}.
\]
Using the bounds at orders \(r\) and zero, and
\(\partial_xv=\partial_x\mathcal Qv\), we apply
\eqref{eq:2.1.6} to obtain
\begin{equation}\label{eq:6.11}
\begin{aligned}
 \|vv_x\|_{H^r}+\|v\Theta_x\|_{H^r}
 \le{}&C_r\|v\|_{H^r}
          \bigl(\|\partial_x\mathcal Qv\|_{\mathcal W^0}
                 +\|\Theta_x\|_{\mathcal W^0}\bigr)\\
 &+C_r\|v\|_2
          \bigl(\|\partial_x\mathcal Qv\|_{\mathcal W^r}
                 +\|\Theta_x\|_{\mathcal W^r}\bigr)\\
 \le{}&C_rA^{\frac43}\varepsilon^2
                    (1+t)^{-\frac{2s-r+1}{4}}.
\end{aligned}
\end{equation}
The undifferentiated velocity is the full \(v\), including its
horizontal mean.

For the spectral energy, set \(\alpha_r=\frac{s-r+4}{2}\).
The symbol bound \eqref{eq:2.3.1} implies, on \(\mathcal K_*^x\),
\[
 \rho^{2-\alpha_r}\langle k,m\rangle^{2r}
 =\rho^{-\frac{s-r}{2}}\langle k,m\rangle^{2r}
 \le C_r\langle k,m\rangle^{2s}.
\]
Hence, applying H\"older's inequality to the Fourier sum and then using \eqref{eq:6.1}, we have
\[
\begin{aligned}
 \|\mathcal M\Theta\|_{H^r}^2
 \le{}&\left(\|\mathcal M^{\frac32}\Theta\|_{H^r}^2\right)^{\frac{\alpha_r}{\alpha_r+1}}
 \left(L_xL_z\sum_{\mathcal K_*^x}
  \rho^{2-\alpha_r}\langle k,m\rangle^{2r}|\widehat\Theta|^2
 \right)^{\frac1{\alpha_r+1}}\\
 \le{}&C_r\varepsilon^{\frac2{\alpha_r+1}}
       \|\mathcal M^{\frac32}\Theta\|_{H^r}^{\frac{2\alpha_r}{\alpha_r+1}}.
\end{aligned}
\]
For the velocity, the boundedness of \(\mathcal M\) and
Poincar\'e's inequality imply
\[
 \|\mathcal Mv\|_{H^r}^2
 \le C_r\varepsilon^{\frac2{\alpha_r+1}}
       \|\mathcal L^{\frac12}\mathcal Mv\|_{H^r}^{\frac{2\alpha_r}{\alpha_r+1}}.
\]
Substitution into \eqref{eq:4.35}, using \eqref{eq:4.34} and
\eqref{eq:6.11}, gives
\begin{equation*}
 E_{5,r}'+c_r\varepsilon^{-\frac2{\alpha_r}}
                  E_{5,r}^{1+\frac1{\alpha_r}}
 \le C_rA^{\frac83}\varepsilon^4
                  (1+t)^{-\frac{2s-r+1}{2}}.
\end{equation*}
The exponents satisfy
\begin{equation}\label{eq:6.12}
 \frac{2s-r+1}{2}-(\alpha_r+1)=\frac{s-5}{2}\ge0.
\end{equation}
Choose \(K\ge1\), independently of \(A\), so that at both endpoints
\[
 E_{5,r}(0)\le K\varepsilon^2,
 \qquad
 c_rK^{1+\frac1{\alpha_r}}-\alpha_rK\ge1.
\]
Provided \(C_rA^{\frac83}\varepsilon^2\le1\), the comparison
\eqref{eq:2.3.15}--\eqref{eq:2.3.16} applies by \eqref{eq:6.12}
and yields
\begin{equation}\label{eq:6.13}
 E_{5,r}(t)\le K\varepsilon^2(1+t)^{-\alpha_r},
 \qquad
 \|\mathcal M\Theta(t)\|_{H^r}
 \le C_r\sqrt K\,\varepsilon(1+t)^{-\frac{s-r+4}{4}}.
\end{equation}
The case \(s=5\) is included, with equality in \eqref{eq:6.12}.

To improve the velocity bounds, use \eqref{eq:4.36},
\eqref{eq:4.38}, and \eqref{eq:6.11} to obtain
\begin{equation}\label{eq:6.14}
 \|\langle w\rangle_z\|_{H_x^r}
       +\sum_{i=1}^2\|\mathcal N_i\|_{H^r}
 \le C_r\varepsilon\|w^0\|_{H^r}
       +C_rA^{\frac43}\varepsilon^2
                       (1+t)^{-\frac{2s-r+1}{4}}.
\end{equation}
The second term decays faster than
\((1+t)^{-\frac{s-r+4}{4}}\).  By \eqref{eq:3.2.3},
\eqref{eq:6.10}, \eqref{eq:6.13}--\eqref{eq:6.14}, and the
convolution bound \eqref{eq:2.3.14},
\[
 \sup_{0\le t\le T}(1+t)^{\frac{s-r+4}{4}}
                \|w^0(t)\|_{H^r}
 \le C_r\sqrt K\,\varepsilon
       +C_r\varepsilon^2\bigl(A+A^{\frac43}\bigr).
\]
For the slow contribution to the horizontal velocity,
\begin{equation}\label{eq:6.15}
 \|\mathcal M^{\frac34}\Theta\|_{H^r}
 \le\|\Theta\|_{H^r}^{\frac14}
        \|\mathcal M\Theta\|_{H^r}^{\frac34}
 \le C_rK^{\frac38}\varepsilon
                         (1+t)^{-\frac{s-r+3}{4}}.
\end{equation}
Using \eqref{eq:6.15} in \eqref{eq:3.2.4}, the same
convolution estimate gives
\[
 \sup_{0\le t\le T}(1+t)^{\frac{s-r+3}{4}}
                \|v(t)\|_{H^r}
 \le C_rK^{\frac38}\varepsilon
       +C_r\varepsilon^2\bigl(A+A^{\frac43}\bigr).
\]
Let \(C_*\ge1\) dominate the initial and estimate constants at both
endpoints.  Fix \(A\ge4C_*(1+\sqrt K)\), and then decrease
\(\varepsilon_{\mathrm d}\) so that
\[
 C_*A^{\frac83}\varepsilon_{\mathrm d}^2\le1,
 \qquad
 C_*\varepsilon_{\mathrm d}(A+A^{\frac43})\le\frac A4.
\]
All three bounds in \eqref{eq:6.10} improve with \(A\) replaced by
\(\frac A2\).  A first-exit argument extends these estimates to
all finite intervals, with constants independent of \(T\).

For \(s>5\), interpolate the endpoint bounds for
\(\mathcal Mv,\mathcal M\Theta,v,w^0\).
Their exponents are affine in \(r\), so the same rates hold for
\(0\le r\le s-5\).  Interpolating \eqref{eq:6.11} in the same way
also gives
\[
 \|vv_x\|_{H^r}+\|v\Theta_x\|_{H^r}
 \le C_r\varepsilon^2(1+t)^{-\frac{2s-r+1}{4}}.
\]
For \(s=5\), these are the bounds already proved at \(r=0\).
The mean-velocity estimate \eqref{eq:4.36} therefore implies
\begin{equation}\label{eq:6.16}
\begin{aligned}
 &\|w^0(t)\|_{H^r}
 \le C_r\varepsilon(1+t)^{-\frac{s-r+4}{4}},\\
 &\|\langle w\rangle_z(t)\|_{H_x^r}
 \le C_r\varepsilon^2(1+t)^{-\frac{s-r+4}{4}}.
\end{aligned}
\end{equation}
Since \(w=w^0+\langle w\rangle_z\), \eqref{eq:6.16}
and the bound for \(v\) prove
\eqref{eq:1.2.20}--\eqref{eq:1.2.21} for the full velocity.

It remains to estimate the pressure for
\(0\le r\le\max\{0,s-5\}\).  On \(\mathcal K_*^x\),
\[
 \frac1{m^4\rho(k,m)}
 =\frac{a(k,m)}{N^2k^2m^2}
 =\frac1{N^2}\left(\frac{\mu_x}{m^2}+\frac{\mu_z}{k^2}\right)\le C.
\]
Thus \(|m|^{-1}\le C\rho^{\frac14}\).  The temperature and spectral
bounds imply
\begin{equation}\label{eq:6.17}
 \|\partial_z^{-1}\Theta(t)\|_{H^r}
 \le C\|\mathcal M^{\frac14}\Theta(t)\|_{H^r}
 \le C\|\Theta(t)\|_{H^r}^{\frac34}
          \|\mathcal M\Theta(t)\|_{H^r}^{\frac14}
 \le C_r\varepsilon(1+t)^{-\frac{s-r+1}{4}}.
\end{equation}
For the horizontal mean, we combine \eqref{eq:2.2.10} with
\eqref{eq:4.59} and apply \(\partial_z^{-1}\) to obtain
\[
 \partial_z^{-1}(\langle\theta(t)\rangle_x-S_\infty)
 =\Pi_{\ne0}^z\int_t^\infty
                \langle w(\tau)\Theta(\tau)\rangle_x\,\mathrm d\tau.
\]
Here the projection removes the constant part of the flux.
Since \(r+2\le s\), \eqref{eq:2.1.6}--\eqref{eq:2.1.7} and
the decay bounds give
\[
\begin{aligned}
 \|\langle w\Theta\rangle_x\|_{H_z^r}
 \le{}&C_r\bigl(\|w\|_{H^r}\|\Theta\|_{H^2}
                       +\|w\|_2\|\Theta\|_{H^{r+2}}\bigr)\\
 \le{}&C_r\varepsilon^2(1+t)^{-\frac{2s-r+2}{4}}.
\end{aligned}
\]
The full \(w\) is controlled by \eqref{eq:6.8} and
\eqref{eq:4.56} at \(s=4\), and by \eqref{eq:6.16} at \(s\ge5\).
The time exponent exceeds one, so the flux is absolutely integrable
in \(H_z^r\), with
\[
 \|\partial_z^{-1}(\langle\theta(t)\rangle_x-S_\infty)\|_{H_z^r}
 \le C_r\varepsilon^2(1+t)^{-\frac{2s-r-2}{4}}.
\]
The difference from the exponent in \eqref{eq:6.17} is
\(\frac{s-3}{4}>0\), so this contribution decays faster.
For the quadratic pressure term, symmetry in Fourier convolution
gives
\[
\begin{aligned}
 \|\langle v^2\rangle_z-\langle v^2\rangle\|_{H^r}
 \le{}&C_r\|v^2\|_{H^r}
 \le C_r\|v\|_{\mathcal W^0}\|v\|_{H^r}\\
 \le{}&C_r\varepsilon^2(1+t)^{-\frac{s-r+3}{4}},
\end{aligned}
\]
where \(\|v\|_{\mathcal W^0}\le C\varepsilon\) follows from
\eqref{eq:6.1} and \eqref{eq:2.1.7}.  Substitution into
\eqref{eq:2.2.2}, with \(p_\infty=g\beta\partial_z^{-1}S_\infty\),
proves \eqref{eq:1.2.22}.

At \(s=4\), only \(r=0\) enters the componentwise estimates.
For \(s\ge5\), all additional smallness conditions were imposed
at \(r=0,s-5\).  Interpolation and pressure recovery require no
further restriction, so a single \(\varepsilon_{\mathrm d}(s)>0\),
independent of \(r\), suffices for all the upper bounds.
\end{proof}

To prove optimality, we fix the horizontal frequency and let the
vertical frequency increase.  The transport terms vanish for the
following family of exact nonlinear solutions.
\begin{proposition}\label{prop:6.1}
Let \(s\ge4\) be an integer.  There exists
\(\eta_*\in(0,\varepsilon_{\mathrm d}]\), depending only on \(s\),
the fixed parameters, and the periods, such that for every
\(0<\eta\le\eta_*\) there are smooth global solutions
\((v_n,w_n,\theta_n,p_n)\) of \eqref{eq:1.1.1} and times
\(t_n\to\infty\) such that
\[
 \langle v_n\rangle_z=\langle\theta_n\rangle_z=0,\qquad
 \gamma+\partial_z\theta_n\ge\frac\gamma2,\qquad
 \|v_n(0)\|_{H^s}+\|\theta_n(0)\|_{H^s}=\eta.
\]
The terminal profiles vanish, and \(\Theta_n=\theta_n\).
For every \(0\le r\le s\),
\begin{equation}\label{eq:6.18}
 \|\theta_n(t_n)\|_{H^r}
 \ge c_r\eta(1+t_n)^{-\frac{s-r}{4}}.
\end{equation}
For \(0\le r\le\max\{0,s-5\}\),
\begin{align}
 \|v_n(t_n)\|_{H^r}
 \ge{}&c_r\eta(1+t_n)^{-\frac{s-r+3}{4}},\label{eq:6.19}\\
 \|w_n(t_n)\|_{H^r}
 \ge{}&c_r\eta(1+t_n)^{-\frac{s-r+4}{4}},\label{eq:6.20}\\
 \|p_n(t_n)\|_{H^r}
 \ge{}&c_r\eta(1+t_n)^{-\frac{s-r+1}{4}}.\label{eq:6.21}
\end{align}
These bounds hold for all sufficiently large \(n\), with \(c_r>0\)
and the lower frequency cutoff independent of \(\eta\).
\end{proposition}
\begin{proof}
Fix \(k_*=\frac{2\pi}{L_x}\), set \(m_n=\frac{2\pi n}{L_z}\),
and write \(a_n=a(k_*,m_n)\), \(\rho_n=\rho(k_*,m_n)\), and
\(\lambda_n=\lambda_+(k_*,m_n)\).
For sufficiently large \(n\), this root is real and negative, and
\[
 \lambda_n=-\frac{2\rho_n}{1+\sqrt{1-\frac{4\rho_n}{a_n}}},
 \qquad \frac{\lambda_n}{\rho_n}\longrightarrow-1,
 \qquad \rho_nm_n^4\longrightarrow\frac{N^2k_*^2}{\mu_z}.
\]
Put \(h_n=\frac{m_n\lambda_n}{\gamma k_*}\), and, for
\(0<\eta\le\eta_*\), choose
\[
 A_n=\frac{\eta}
 {\sqrt{\frac{L_xL_z}{2}}\,(1+|h_n|)\langle k_*,m_n\rangle^s}.
\]
Consider
\begin{equation}\label{eq:6.22}
\left\{
\begin{aligned}
 &\theta_n(t,x,z)=A_ne^{\lambda_nt}\cos(k_*x+m_nz),\\
 &v_n(t,x,z)=h_n\theta_n(t,x,z),\qquad
       w_n(t,x,z)=-\frac{\lambda_n}{\gamma}\theta_n(t,x,z),\\
 &p_n(t,x,z)=\frac{g\beta A_n}{m_n}e^{\lambda_nt}\sin(k_*x+m_nz).
\end{aligned}
\right.
\end{equation}
Since \(k_*h_n=\frac{m_n\lambda_n}{\gamma}\),
\[
 k_*v_n+m_nw_n=0,\qquad
 (v_n\partial_x+w_n\partial_z)f(k_*x+m_nz)=0
\]
for every smooth \(f\).  Thus the velocity is divergence free
and both transport terms vanish.  The vertical averages are
\[
 \langle v_n\theta_n\rangle_z=\frac12h_nA_n^2e^{2\lambda_nt},
 \qquad
 \langle v_n^2\rangle_z=\langle v_n^2\rangle
                       =\frac12h_n^2A_n^2e^{2\lambda_nt}.
\]
The first is independent of \(x\), so the compatible vertical mean
is zero.  The quadratic terms in the pressure formula cancel, and
\(\langle p_n\rangle=0\).  The temperature and hydrostatic equations
follow directly from \eqref{eq:6.22}.  The momentum equation reduces to
\[
 (\lambda_n+a_n)h_n+\frac{g\beta k_*}{m_n}
 =\frac{m_n}{\gamma k_*}
     \left(\lambda_n^2+a_n\lambda_n+\frac{N^2k_*^2}{m_n^2}\right)=0.
\]
Hence \eqref{eq:6.22} satisfies the full system and both diagnostic
relations.

The data have zero horizontal and vertical means, and the choice of
\(A_n\) ensures
\(\|v_n(0)\|_{H^s}+\|\theta_n(0)\|_{H^s}=\eta\).
Moreover,
\[
 \sup_{t\ge0}\|\partial_z\theta_n(t)\|_\infty
 \le m_nA_n\le C_s\eta.
\]
Choose \(\eta_*\le\varepsilon_{\mathrm d}\) with
\(C_s\eta_*\le\frac\gamma2\), independently of \(n\).
The family is strictly stratified and belongs to the stated
small-data class.  Each fixed member decays exponentially and has
zero horizontal temperature mean, so its terminal profiles vanish
and \(\Theta_n=\theta_n\).

Take \(t_n=\rho_n^{-1}\).  Then \(1+t_n\) is comparable to
\(m_n^4\), \(e^{\lambda_nt_n}\to e^{-1}\), and
\[
 m_n^3|h_n|\longrightarrow\frac{N^2k_*}{\gamma\mu_z},
 \qquad
 m_n^4\frac{|\lambda_n|}{\gamma}
       \longrightarrow\frac{N^2k_*^2}{\gamma\mu_z}.
\]
Parseval's identity gives
\[
 \|\theta_n(t_n)\|_{H^r}
 =\frac{\eta}{1+|h_n|}\langle k_*,m_n\rangle^{r-s}
                          e^{\lambda_n\rho_n^{-1}},
\]
while the other components have the exact norms
\[
 (\|v_n(t_n)\|_{H^r},\|w_n(t_n)\|_{H^r},\|p_n(t_n)\|_{H^r})
 =\left(|h_n|,\frac{|\lambda_n|}{\gamma},\frac{g\beta}{m_n}\right)
                         \|\theta_n(t_n)\|_{H^r}.
\]
Here sine and cosine have the same Sobolev norm.  The asymptotics
above prove \eqref{eq:6.18}, \eqref{eq:6.19}, \eqref{eq:6.20},
and \eqref{eq:6.21}, with constants independent of \(n\) and \(\eta\).
The temperature bound also includes \(r=s\).
\end{proof}

\begin{proof}[Completion of the proof of Theorem~\ref{thm:1.3}]
Fix \(0<\eta\le\eta_*\).  For each stated norm and admissible
index \(r\), an increase of the decay exponent by \(\zeta>0\)
in an estimate uniform over the stated data class would contradict
Proposition~\ref{prop:6.1}, since
\[
 0<c_r\le C_r(1+t_n)^{-\zeta}\longrightarrow0.
\]
Since \(S_{\infty,n}=0\) and \(\Theta_n=\theta_n\), the
temperature lower bound applies separately to the fluctuation
and the difference from equilibrium.  At \(r=s\),
\eqref{eq:6.18} excludes any positive uniform algebraic decay
rate for either quantity in \(H^s\).
Together with the upper bounds proved above, this completes
the proof.
\end{proof}

%=====================================================================
% End of Section 6
%=====================================================================

\end{document}